\documentclass[11pt]{article}
\pdfoutput=1

\RequirePackage{amsmath, amsthm, amssymb, mathtools}
\RequirePackage{natbib}
\RequirePackage{float}
\RequirePackage{tikz}

\newcommand{\hl}[1]{#1}
\newcommand{\hlm}[1]{#1}
\newcommand{\hlp}[1]{#1}
\newcommand{\hlc}[1]{#1}

\usepackage[margin=1.2in]{geometry}
\usepackage{libertine}
\usepackage{courier}
\usepackage{mathpazo}
\usepackage{xcolor}
\usepackage{pifont}
\usepackage{authblk}
\usepackage[backref=page]{hyperref}
\usepackage{enumerate}
\RequirePackage{tabularray}
\RequirePackage{nicematrix}
\NiceMatrixOptions{
code-for-first-row = \tiny ,
code-for-last-row = \tiny ,
code-for-first-col = \tiny ,
code-for-last-col = \tiny
}

\usepackage{titlesec}

\definecolor{darkbrown}{RGB}{170,100,0} 
\definecolor{darkdarkbrown}{RGB}{110,70,0} 
\titleformat{\section}
{\color{darkdarkbrown}\normalfont\Large\bfseries}
{\color{darkdarkbrown}\thesection}{1em}{}
\titleformat{\subsection}
{\color{darkdarkbrown}\normalfont\Large\bfseries}
{\color{darkdarkbrown}\thesubsection}{1em}{}

\hypersetup{
	colorlinks=true,
  linkcolor=darkbrown,
  citecolor=darkbrown,
  filecolor=darkbrown,
  urlcolor=darkbrown
}

 \renewcommand*{\backrefalt}[4]{%
    \ifcase #1%
     \or (page:~#2)%
     \else (pages:~#2)%
    \fi%
    }

    \makeatletter
\def\@fnsymbol#1{\ensuremath{\ifcase#1\or \dagger\or \ddagger\or
   \mathsection\or \mathparagraph\or \|\or **\or \dagger\dagger
   \or \ddagger\ddagger \else\@ctrerr\fi}}
    \makeatother

\newtheorem{theorem}{Theorem}[section]
\newtheorem{lemma}{Lemma}[section]
\newtheorem{definition}{Definition}[section]
\newtheorem{proposition}{Proposition}[section]
\newtheorem{example}{Example}[section]
\newtheorem{corollary}{Corollary}[section]
\newtheorem{remark}{Remark}[section]

\providecommand{\keywords}[1]
{
  \small	
  \textbf{\textit{Keywords---}} #1
}
\newcommand{\set}[1]{\left\{ #1 \right\}}
\newcommand{\R}{\mathbb{R}}
\newcommand{\Z}{\mathbb{Z}}

\definecolor{darkgreen}{rgb}{0.2, 0.5, 0.2}
\DeclareMathOperator{\diag}{diag}
\DeclareMathOperator*{\argmin}{arg\,min}
\DeclareMathOperator*{\argmax}{arg\,max}

\DeclareMathOperator{\Span}{Span}
\DeclareMathOperator{\Ker}{Ker}
\DeclareMathOperator{\Range}{Im}
\DeclareMathOperator{\Id}{Id}

\newcommand{\maxentmc}{U}
\newcommand{\pred}[1]{\delta\left[#1\right]}
\newcommand{\eqdef}{\triangleq}
\newcommand{\N}{\mathbb{N}}

\newcommand{\trn}{^\intercal}

\newcommand{\abs}[1]{\left| #1 \right|}

\newcommand{\PR}[2][]{\mathbb{P}_{#1}\left( #2 \right)}
\newcommand{\E}[2][]{\mathbb{E}_{#1}\left[ #2 \right]}

\newcommand{\eps}{\varepsilon}
\newcommand{\stoch}{\mathfrak{s}}

\newcommand{\kl}[2]{D\left(#1 \| #2\right)}
\newcommand{\kld}[2]{D^\star\left(#1 \| #2\right)}

\newcommand{\at}[1]{\Bigr|_{#1}}
\newcommand{\cyc}{_{\mathsf{cyc}}}
\newcommand{\rev}{_{\mathsf{rev}}}

\newcommand{\sym}{_{\mathsf{sym}}}
\newcommand{\bis}{_{\mathsf{bis}}}
\newcommand{\iid}{_{\mathsf{iid}}}

\newcommand{\fshr}{\mathfrak{g}}

\newcommand{\calX}{\mathcal{X}}
\newcommand{\calY}{\mathcal{Y}}
\newcommand{\calZ}{\mathcal{Z}}
\newcommand{\calP}{\mathcal{P}}
\newcommand{\calE}{\mathcal{E}}
\newcommand{\calF}{\mathcal{F}}
\newcommand{\calB}{\mathcal{B}}

\newcommand{\calW}{\mathcal{W}}
\newcommand{\calC}{\mathcal{C}}
\newcommand{\calV}{\mathcal{V}}
\newcommand{\calD}{\mathcal{D}}
\newcommand{\calL}{\mathcal{L}}
\newcommand{\calG}{\mathcal{G}}
\newcommand{\calJ}{\mathcal{J}}

\newcommand{\calR}{\mathcal{R}}
\newcommand{\calS}{\mathcal{S}}

\newcommand{\origin}{\odot}
\newcommand{\IPr}{I}
\newcommand{\rIPr}{rI}
\newcommand{\ymax}{\check{y}}

\DeclareMathOperator{\Emb}{Emb}
\DeclareMathOperator{\MEmb}{M-Emb}
\DeclareMathOperator{\CEmb}{C-Emb}
\DeclareMathOperator{\MMEmb}{MM-Emb}
\newcommand{\mEmb}{\gamma^{(m)}\mbox{-}\Emb}
\newcommand{\eEmb}{\gamma^{(e)}\mbox{-}\Emb}
\newcommand{\Embk}{\Emb_\kappa}
\newcommand{\Embkc}{\CEmb_{\kappa}}
\newcommand{\MEmbk}{\MEmb_\kappa}
\newcommand{\MMEmbk}{\MMEmb_{\kappa}}
\newcommand{\EEmb}{e\mbox{-}\Emb}
\newcommand{\EEmbk}{e\mbox{-}\Emb_\kappa}
\newcommand{\EEmbkc}{e\mbox{-}\Emb_\kappa^C}

\title{Geometric Aspects of 
\\ Data-Processing of Markov Chains}

\author[1]{Geoffrey Wolfer \thanks{email: geoffrey.wolfer@riken.jp. \\
The author is supported by the Special Postdoctoral Researcher
Program (SPDR) of RIKEN and by the Japan Society for the Promotion of
Science KAKENHI under Grant 23K13024.
Part of this project was conducted when the author was an International Research Fellow of Japan Society for the Promotion of Science, and supported under KAKENHI Grant 21F20378.}}
\author[2]{Shun Watanabe \thanks{email: shunwata@cc.tuat.ac.jp. \\ The author is supported in part by Japan Society for the Promotion of Science KAKENHI under Grant 20H02144.}}
\affil[1]{RIKEN Center for AI Project}
\affil[2]{Department of Computer and Information Sciences \protect\\ Tokyo University of Agriculture and Technology}

\date{\today}

\begin{document}

\maketitle

\vspace{2cm}
\begin{abstract}
\hl{We examine data-processing of Markov chains through the lens of information geometry. We first establish a theory of congruent Markov morphisms within the framework of stochastic matrices. Specifically, we introduce and justify the concept of a linear right inverse (congruent embedding) for lumping, a well-known operation used in Markov chains to extract coarse information. Furthermore, we inspect information projections onto geodesically convex sets of stochastic matrices, and show that under some conditions, projecting (m-projection) onto doubly convex submanifolds can be regarded as a form of data-processing. Finally, we show that the family of lumpable stochastic matrices can be meaningfully endowed with the structure of a foliated manifold and motivate our construction in the context of embedded models and inference.}
\end{abstract}


\clearpage
\tableofcontents
\clearpage

\section{Introduction}
\label{section:introduction}
\hl{The information divergence rate 
between two stochastic processes $Y = (Y_t)_{t \in \N}, Y' = (Y'_t)_{t \in \N}$ within a finite space $\calY$
quantifies the average discrepancy between these processes per unit time.}
\hl{It} is defined, when the limit exists, by
\begin{equation}
\label{eq:divergence-rate-random-process}
\kl{Y}{Y'} \eqdef \lim_{k \to \infty} \frac{1}{k} \kl{Y_1, \dots, Y_{k}}{Y'_1, \dots, Y'_{k}},
\end{equation}
where $D$ \hl{represents} the Kullback--Leibler divergence \citep{kullback1951information}.
Information monotonicity dictates that merging symbols in $\calY$ and \hl{observing} the processes on the resulting space $\calX$ \hl{should inevitably result in a decreased divergence between the processes}, originating from an information loss. 
When the two processes under consideration are independent and identically distributed (iid) according to respective discrete distributions $\mu, \nu \in \calP(\calY)$, \hl{where $\calP(\calY)$ represents the probability simplex over $\calY$,}
this property \hl{is effectively captured} by the action of a memoryless channel $W \colon \calP(\calY) \to \calP(\calX)$, \hl{as expressed by the inequality, }
\begin{equation}
\label{eq:action-of-memoryless-channel}
    \kl{\mu}{\nu} \geq \kl{\mu W}{ \nu W}.
\end{equation}
\hl{Alternatively, we can explore embeddings of distributions into a larger space $\calZ$, which may not necessarily be implemented by a channel $W$.}
\hl{Equality is known to be achieved} in \eqref{eq:action-of-memoryless-channel} for any pair of distributions if and only if the embedding \hl{is a} congruent Markov morphism \citep{doi:10.1080/02331887808801428}.
\hl{These principles are firmly established in the iid setting, and in the present study, our aim is to  extend and build upon these concepts within the context of Markov processes.}
While the divergence rate between two Markov chains can still be expressed in terms of their transition \hl{matrices}, \hl{tackling Markov processes presents additional challenges}.
Foremost, processing them, \hl{even through the simple act of merging symbols}, can \hl{disrupt} the Markov property.
Furthermore, although simulating the process resulting from the action of a Markov morphism in the iid setting is straightforward \footnote{In the information theory literature, this amounts to stating that there exists an operational definition. See e.g. \citet{issa2019operational}.}, it \hl{remains unclear} what actions on \hl{stochastic matrices} would also allow us\hl{,} given some trajectory of observations, to simulate the trajectory obtained from the processed transition matrix.
These considerations \hl{prompt} us to put forward the below listed desiderata.
\begin{enumerate}
    \item[\hl{(D1)}] \hl{Data-processing of Markov chains} should be expressible in terms of the action of an operator upon its transition \hl{matrix}.
    \item[\hl{(D2)}] \hl{Markovianity should be preserved throughout data-processing}.
    \item[\hl{(D3)}] \hl{Processing} should \hl{bear} operational meaning, \hl{implying the existence of} a \hl{possibly randomized} mapping defined on trajectories of observations. \hl{Applying} this mapping to the original trajectory \hl{should yield} another trajectory \hl{whose} dynamics are governed by the transition matrix \hl{resulting from the operation described in (D1)}.
\end{enumerate}
The natural \hl{procedure} of reducing the state space whilst preserving Markovianity is commonly referred to as lumping \citep{kemeny1983finite}. Lumping satisfy all of our above requirements, \hl{and will serve as the starting point for our study}.

\hl{Our} first question \hl{pertains to the} natural ``inverse'' operation of lumping\hl{---namely} embedding\footnote{In this note, the \hl{concept} of embedding is distinct from \hl{the embedding problem of}
\citet{elfving1937theorie}.
} a Markov chain into a possibly larger state space, while still \hl{conforming to (D1), (D2) and (D3)}. This \hl{exploration} will lead us to define \hl{linear right inverses (congruent embeddings)} with respect to lumpings of \hl{stochastic matrices}, and result in \hl{characterizing this morphism class} in terms of Markov embeddings, the central notion we introduce in Section~\ref{section:embeddings-markov}.

\hl{Adopting} an information geometry perspective, we \hl{will regard} irreducible \hl{stochastic matrices} as dually flat information manifolds \citep{nagaoka2017exponential}. In this context, embeddings \hl{represent} injective morphisms, \hl{acting as} structure-preserving maps. 
\hl{An intriguing inquiry involves determining which structures} (Fisher information metric, dual affine connections, exponential families,...) \hl{are} preserved under Markov embeddings.
\hl{We aim to address natural questions, including the exploration of larger classes of embeddings that maintain the same structure, albeit potentially compromising some of the desiderata (D1), (D2), (D3)}. Conversely, we will \hl{identify} sub-classes of embeddings that preserve additional structure of interest, and \hl{elucidate how established} embeddings, \hl{such as} Hudson expansions \citep[Section~6.5]{kemeny1983finite}, \hl{align with our framework}.

\hl{Lastly}, we will \hl{conduct a comprehensive examination of} the geometric structure of families of lumpable \hl{stochastic matrices}.
Some families of \hl{stochastic matrices} are known to enjoy \hl{advantageous} properties; \hl{for instance} reversible Markov chains form \hl{both} an e-family (exponential family, see Definition~{\ref{definition:e-family-parametric}}) and an m-family (mixture family, see Definition~{\ref{definition:m-family-parametric})} \citep{wolfer2021information}. \hl{We will clarify} that \hl{while} lumpable \hl{stochastic matrices} generally do not form e-families nor m-families, we still \hl{can} endow the family with the structure of a foliated manifold, and \hl{will provide insightful} interpretation of this construction.

\subsection{Related work}
\label{section:related-work}
The question of whether processing a Markov chain \hl{through} a function retains its Markovianity can be traced back to \citet{burke1958markovian} (see also \citet{pitman1981markov}).
Chains that \hl{maintain} this property under merging \hl{of} some of their states were later termed lumpable by \citet{kemeny1983finite}.
\hl{While a comprehensive} survey of lumpability is beyond the scope of this paper, \hl{notable} works \hl{include those of} \citet{rubino1989weak} and \citet{buchholz1994exact}.
The \hl{concept} was also extended in \hl{various} ways, \hl{such as} quasi-lumpability \citep{franceschinis1994bounds} where the transition matrix is lumpable modulo some perturbation, 
or higher-order lumpability \citep{gurvits2005markov, geiger2014lumpings}, where a lumped Markov chain may lose its first-order Markov property but retains a $k$th-order Markov property.
\hl{Additionally, } the problem of lumpability is directly related to \hl{the} identifiability of hidden Markov sources \citep{ito1992identifiability,kabayashi1991equivalence, hayashi2019local}.

Following an axiomatic approach, 
\citet{doi:10.1080/02331887808801428}
first introduced, motivated and analyzed Markov morphisms
as the statistical mappings of interest for data-processing.
More recently, similar approaches have been taken to put forward several classes of embeddings for conditional models
\citep{lebanon2005axiomatic, lebanon2012extended, montufar2014fisher}.
Although a \hl{stochastic matrix} corresponds to some conditional model, in this paper we also think of it as a stochastic process, with a stationary distribution. \hl{This perspective prompts} us to consider a more restricted class of natural embeddings than \hl{those discussed in} the aforementioned works.

Exponential tilting of stochastic matrices \hl{was} first found in \citet{miller1961convexity}.
The large deviation theory for Markov chains, was \hl{further} developed by \citet{donsker1975asymptotic, gartner1977large, dembo2011large}.
\citet{csiszar1987conditional} first recognized
the exponential structure of the family of irreducible \hl{stochastic matrices}, and
\citet{nagaoka2017exponential} later \hl{provided a comprehensive} treatment in the language of information geometry.
We \hl{direct} the reader to \hlc{the survey of \citet{wolfer2023information}}
for \hl{an historical overview} of the series of works
\citep{ito1988, takeuchi1998asymptotically, takeuchi2007exponential,takeuchi2017asymptotic, nakagawa1993converse}
that contributed to this construction.

The first appearance of the Pythagorean theorem for e-projections onto m-families of distributions can be found in \citet{chentsov1968nonsymmetrical}.
For a complete treatment of the theory of projections onto $\alpha$-families of distributions, we refer the reader to the excellent monograph of \citet{amari2007methods}.
In the context of Markov chains, Pythagorean identities for orthogonal projections can be found in \citet{hayashi2014information}.
In \citet{wolfer2021information}, closed-form expressions are given for the e/m-projections onto the e/m-family of reversible Markov chains.

More general information projections onto convex sets of distributions are also a well-studied topic. The Pythagorean inequality in this context is credited to \citet{csiszar1975divergence} and \citet{topsoe1979information} (see also \citet{csiszar1984sanov}).
Inequalities involving the reverse information projection have also been devised, for example a four-point property in \citet{csiszar1984information}, or a Pythagorean inequality on log-convex sets \hl{in}
\citet{csiszar2003information}. The reader is invited to consult \citet{csiszar2004information} for a complete exposition.
In the Markovian setting, we mention \hlc{\citet{boza1971asymptotically} and \citet{csiszar1987conditional},} who considered information projections of \hl{stochastic matrices} onto convex sets of edge measures. To the best of our knowledge, the analysis of reverse information projections onto e-convex sets of \hl{stochastic matrices} had not yet been carried out.

\hlp{
\section{Summary of the main contributions}
\label{section:main-contributions}
}
\hlp{We introduce only minimal notations to provide a summary of the main contributions, directing the reader to Section~\ref{section:preliminaries} for more details.
We let
$(\calX, \calD)$ and $(\calY, \calE)$ be two strongly connected directed graphs with finite vertex sets $\calX, \calY$, where $\abs{\calX} \leq \abs{\calY}$ and respective edge sets $\calD$ and $\calE$. We denote $\calW(\calX, \calD)$ the set of irreducible row-stochastic matrices over the graph $(\calX, \calD)$.
We fix
$\kappa \colon \calY \to \calX$ a (surjective) function for which $\calD = \set{(\kappa(y), \kappa(y') \colon (y,y') \in \calE }$ and we write
$\biguplus_{x \in \calX} \calS_x = \calY$, for the partition associated with $\kappa$, where for $x \in \calX, \calS_x = \kappa^{-1}(\set{x})$.
The function $\kappa$, by merging symbols together, erases information. In the context of Markov processes, $\kappa$ is called a lumping function \citep{kemeny1983finite}, and induces a push-forward for stochastic matrices,
\begin{equation*}
    \kappa_\star \colon \calW_\kappa(\calY, \calE)  \to \calW(\calX, \calD),
\end{equation*}
where the domain $\calW_\kappa(\calY, \calE)$ is the subset of lumpable stochastic matrices.
Lumpable matrices are characterized by the property that a Markov chain sampled according to them will retain its Markovianity when observations are merged by $\kappa$ at each time coordinate.
To avoid trivialities, we assume compatibility conditions (Proposition~\ref{proposition:lumpable-not-empty}) on $\kappa$, 
and the connection digraphs guaranteeing that
$\calW_\kappa(\calX, \calE) \neq \emptyset$.

\subsection{Contribution 1: Propose a natural embedding operation for Markov chains}
\label{section:contribution-natural-embedding}

Our first chief contribution is to introduce a concept of embedding for families of stochastic matrices, serving as a right inverse of the lumping operation.
For a matrix 
$\Lambda$ positive over $\calE$, and such that for any $y \in \calY$ and $x' \in \calX$, when $(\kappa(y), x') \in \calD$$, \Lambda(y,\cdot)$, is a distribution on $\calS_{x'}$,
we define an induced push-forward $\Lambda_\star$ for stochastic matrices (Definition~\ref{definition:markov-embeddings}). This push-forward is referred to as a Markov embedding, and is expressed as
\begin{equation*}
    \begin{split}
        \Lambda_\star \colon \calW(\calX, \calD)  &\to \calW_\kappa(\calY, \calE),
    \end{split}
\end{equation*}
where for any $(y,y') \in \calE$,
$$\Lambda_\star P(y,y') = P(\kappa(y), \kappa(y'))\Lambda(y,y').$$
By construction, desiderata (D1) and (D2) are satisfied.
Additionally, our definition is motivated by the following properties.
\begin{enumerate}
    \item Markov embeddings have an operational interpretation (D3). Namely, there exists a randomized function which can map a trajectory sampled from $P \in \calW(\calX, \calD)$ to a simulated trajectory sampled from $\Lambda_\star P$ (see Figure~\ref{figure:trajectory-simulation}), without knowing the transition matrix $P$.
    \item Markov embeddings form a strict superset of the Hudson expansion (Theorem~\ref{theorem:hudson-as-markov-embedding}).
    The latter was previously proposed by \citet{kemeny1983finite} as the natural converse operation to lumping. This expansion can be understood as an operation that embeds, using a sliding window approach, the family of first-order stochastic matrices as a sub-family of second-order stochastic matrices (Section~\ref{section:hudson-embedding-sliding-window}).
    \item Markov embeddings maintain the information geometric structure on stochastic matrices introduced by  \citet{nagaoka2017exponential}.
Namely, 
for two points $P, P' \in \calW(\calX, \calD)$, it holds that
\begin{equation*}
    \kl{\Lambda_\star P}{\Lambda_\star P'} = \kl{P}{P'},
\end{equation*}
and as a consequence, the Fisher metric and dual affine connections are preserved (Lemma~\ref{lemma:markov-embedding-preserve-information-geometry}).
    \item A strong justification for Markov embeddings is given by Theorem~\ref{theorem:congruent-embeddings-are-lambda-embeddings},
where we show that a linear map is monotonic, preserves irreducibility and is a right inverse of a lumping map, if and only if it belongs to our class of Markov embeddings.
This characterization mirrors the counterpart established for Markov morphisms in the context of finite measure spaces \citep[Example~5.2]{ay2017}.
\end{enumerate}

\subsection{Contribution 2: Examine the geometric structure of lumpable stochastic matrices}
\label{section:contribution-geometry-lumpable}

Our second main contribution is to provide an analysis of the information geometric structure of the family of lumpable stochastic matrices.
We briefly motivate our analysis as follows.
Geometric properties of families of stochastic matrices translate into statistical properties of the associated Markov models. For instance, in exponential families and curved exponential families of stochastic matrices, for the problem of parametric estimation, it is possible to construct a simple estimator that asymptotically attains the Cram\'{e}r--Rao bound (efficient), given by the inverse of the Fisher information matrix \citep{hayashi2014information}, which is easily computable.
Furthermore, the dually flat structure induced from the information divergence rate between Markov chains leads to decomposition theorems \citep{csiszar1987conditional} in terms of information projections.
For instance, the e/m-nature of reversible stochastic matrices yields Pythagorean identities for reversible e-projection and m-projection
\citep{wolfer2021information}.

Given a family of irreducible stochastic matrices $\calW(\calY, \calE)$ and a fixed lumping function $\kappa \colon \calY \to \calX$, it is not hard to verify that the family of $\kappa$-lumpable stochastic matrices $\calW_\kappa(\calY, \calE)$ does not generally constitute an e-family or an m-family in $\calW(\calY, \calE)$.
One significant application of our notion of Markov embedding (refer to \ref{section:contribution-natural-embedding} Contribution 1) is that it naturally facilitates the decomposition of the lumpable family $\calW_\kappa(\calY, \calE)$ in terms of simpler mathematical structures.
Observing (Lemma~\ref{lemma:construct-canonical-embedding}) that
any $P_{\origin} \in \calW_\kappa(\calY, \calE)$ induces a canonical embedding $\Lambda^{(P_{\origin})}_\star$ satisfying
$P_{\origin} = \Lambda_\star^{(P_{\origin})} \kappa_\star P_{\origin}$,
we define,
\begin{equation*}
    \calJ(P_{\origin}) \eqdef \set{ \Lambda_\star^{(P_{\origin})}\bar{P} \colon \bar{P} \in \calW(\calX, \calD) } \subset \calW_\kappa(\calY, \calE),
\end{equation*}
the image of the base family of stochastic matrices by this canonical embedding.
Fixing some origin $\bar{P}_{0} \in \calW(\calX, \calD)$, it is also natural to define
\begin{equation*}
    \calL(\bar{P}_{0}) \eqdef \set{ P \in \calW_\kappa(\calY, \calE) \colon \kappa_\star P = \bar{P}_{0}},
\end{equation*}
the family of stochastic matrices which lump into $\bar{P}_{0}$.
We prove (Lemma~\ref{lemma:L-and-J-are-special-families}) that
\begin{enumerate}
    \item $\calJ(P_{\origin})$ forms an e-family in $\calW(\calY, \calE)$,
    \item
    $\calL(\bar{P}_0)$ forms an m-family in $\calW(\calY, \calE)$.
\end{enumerate}
Furthermore, we show that their dimension is given by
\begin{equation*}
    \dim \calJ(P_{\origin}) = \abs{\calD} - \abs{\calX}, \qquad \dim \calL(\bar{P}_0) =
    \abs{\calE} - \sum_{(x,x') \in \calD} \abs{\calS_x}.
\end{equation*}
The aforementioned structures enable us to construct a foliation on the manifold of lumpable stochastic matrices (Theorem~\ref{theorem:foliation-of-lumpable-kernels}),
\begin{equation*}
    \calW_\kappa(\calY, \calE) = \biguplus_{P \in \calL(\bar{P}_0)} \calJ(P),
\end{equation*}
and to determine that
\begin{equation*}
    \dim \calW_\kappa(\calY, \calE) = \abs{\calE} - \sum_{(x,x') \in \calD} \abs{\calS_x} + \abs{\calD} - \abs{\calX}.
\end{equation*}
As a result, for each base kernel $\bar{P}_0$, we can construct a mixed coordinate system \citep[Chapter~3.7]{amari2007methods} over $\calW_\kappa(\calY, \calE)$ with $\abs{\calE} - \sum_{(x,x') \in \calD} \abs{\calS_x}$ m-coordinates and $\abs{\calD} - \abs{\calX}$ e-coordinates.
The above foliation is mutually dual \citep[p.75]{amari2007methods}. Specifically, when we fix $\bar{P}_0 \in \calW(\calX, \calD)$ and consider
arbitrarily chosen stochastic matrices $P_{\origin}, P \in \calL(\bar{P}_0)$ and $P' \in \calJ(P_{\origin})$, the following Pythagorean relation holds (Theorem~\ref{theorem:foliation-pythagorean-theorem}),
\begin{equation*}
    \kl{P}{P'} = \kl{P}{P_{\origin}} + \kl{P_{\origin}}{P'},
\end{equation*}
and $P_{\origin}$ is the e-projection of $P'$ onto $\calL(\bar{P}_0)$ and the m-projection of $P$ onto $\calJ(P_{\origin})$.

\subsection{Applications}

Below, we enumerate two applications of our framework. For further motivation, we encourage the reader to refer to Section~\ref{section:foliation-interpretations-applications}.

\subsubsection{Maximum entropy principle for refining Markov models}

A scientist is provided with a base Markov model in the form of a stochastic matrix $\bar{P}_0 \in \calW(\calX, \calD)$. They aim to refine this model by increasing the granularity of the states, corresponding to an embedding operation.
Their initial (prior) model, before confronting it with newly collected data, could be guided by the maximum entropy principle.
In our framework, applying this principle reduces to computing (Section~\ref{section:max-entropy}) the e-projection of the maxentropic chain $U \in \calW(\calY, \calE)$ onto $\calL(\bar{P}_0)$,
\begin{equation*}
    P = \argmin_{P' \in \calL(\bar{P}_0)} \kl{P'}{U}.
\end{equation*}
This involves optimizing a convex objective function on a convex domain, arising from $\calL(\bar{P}_0)$ being an m-family (refer to \ref{section:contribution-geometry-lumpable} Contribution 2).

\subsubsection{Inference in Markov chains from a single trajectory}

Markov embeddings have practical applications in inference for Markov models, enabling the relaxation of certain assumptions in the model.
For example, by employing the symmetrization embedding (Corollary~\ref{corollary:embed-reversible-to-symmetric}),
\citet{wolfer2023geometric}
reduced the problem of identity testing of $\pi$-reversible Markov chains to the better-understood problem of identity testing for symmetric Markov chains (Section~\ref{section:application-identity-testing-reversible-markov-chains}). This approach not only simplifies the methodology but also yields state-of-the-art results.

\subsection{Outline and additional contributions}
\label{section:outline}
Our paper is organized as follows.

In Section~\ref{section:introduction}, we provide a brief introduction to data processing and discuss the challenges of extending from iid to Markov processes.

In Section~\ref{section:main-contributions}, we summarize of our main contributions.

In Section~\ref{section:preliminaries}, we set out our notations, revisit the information geometric structure of irreducible stochastic matrices introduced by \citet{nagaoka2017exponential}, and review a known characterization of lumping of Markov chains in terms of their transition matrices.

In Section~\ref{section:embeddings-markov}, we first extend constructively (Contribution 1) the well-known notion of a Markov morphism in the context of distributions, to that of a Markov morphism in the context of stochastic matrices (Definition~\ref{definition:markov-embeddings}).
We additionally develop (Section~\ref{section:embeddings-congruent}) a theory of lumpable linear operators, and define congruent embeddings (Definition~\ref{definition:congruent-embedding}) in this context.
We proceed to state and prove that congruent embeddings coincide with the Markov morphisms we constructed (Theorem~\ref{theorem:congruent-embeddings-are-lambda-embeddings}).
In Section \ref{section:embeddings-exponential}, we subsequently expand the class of Markov embeddings to the broader class of exponential embeddings. While exponential embeddings are generally not isometric, we show that they still preserve e-structures within families of irreducible stochastic matrices (Theorem~\ref{theorem:exponential-embedding-geodesic-affinity}).
Moving on to Section~\ref{section:embeddings-notable-examples},
we delve into a few notable classes of embeddings.
We start with the special case of Hudson expansions, initially presented by \citet{kemeny1983finite} as the natural inverse operation to lumping.
We express Hudson expansions as Markov embeddings, and propose an interpretation of the embedding of a family of irreducible stochastic matrices as a first-order subfamily of second-order stochastic matrices.
We proceed by examining a natural subset of Markov embeddings, referred to as memoryless Markov embeddings, which not only preserve the m-structure of irreducible stochastic matrices but also maintain reversibility. In particular, we demonstrate that this class is sufficiently rich to embed any rational stochastic matrix into the set of bistochastic matrices (refer to Corollary~\ref{corollary:embedding-to-doubly-stochastic-kernel}). We close this section by more systematically investigating reversibility preservation of Markov embeddings, and discuss some of the advantages that arise from this property.
Concluding this section, we systematically explore the preservation of reversibility in Markov embeddings and discuss some of the advantages that arise from this property. For a comprehensive nomenclature of the various classes of embeddings discussed, we invite the reader to consult Table~\ref{table:nomenclature}. Additionally, Figure~\ref{figure:embeddings-landscape} provides an illustration of their hierarchical structure.

In Section~\ref{section:information-projection}, we complement the theory of information projection on geodesically convex sets of stochastic matrices by establishing a Pythagorean inequality for reverse information projections onto e-convex sets (Proposition~\ref{proposition:pythagorean-inequality-e-convex}).
Furthermore, we delve into the Markovian equivalent of the four-point property and demonstrate that under favorable conditions, the m-projection of two stochastic matrices onto an em-family (both an e-family and an m-family) can be regarded as a form of data processing.

In Section~\ref{section:information-geometry-lumpable}, we analyze the family of lumpable stochastic matrices from an information geometric standpoint (Contribution 2). 
We motivate our construction by considering Pythagorean projections on leaves and propose an interpretation in the context of estimation for embedded models. Beyond the examples provided in this paper, our results hold potential applications in various domains. For instance, in certain coding problems in information theory, the lumpability of sources plays a crucial role in deriving fundamental limits \citep{hayashi2016uniform, hayashi2020finite}.

Finally, Section~\ref{section:higher-order} 
briefly discusses compositions of embeddings, and higher-order lumping and embedding.
We show, by providing examples, that composition opens the door to a class of embeddings that is significantly richer than Markov embeddings. 

Many technical proofs have been deferred to  Section~\ref{section:proofs} for clarity of the exposition.
}

\section{Notation and preliminaries}
\label{section:preliminaries}
We write $[n] = \set{1, 2, \dots, n}$ for $n \in \N$, and $\delta[\cdot] \to \set{0, 1}$ for the predicate indicator function.
Let $\calX, \calY$ be sets such that $\abs{\calX} = n, \abs{\calY} = m$ with $n < \infty,  m < \infty$, where to avoid trivialities, we additionally assume that $n, m > 1$.
\hl{For some index set $I \subset \N$, $\calY = \biguplus_{i \in I} S_i$ denotes a partition of $\calY$, meaning that $\calY = \bigcup_{i \in I} S_i$ and $\forall i,j \in I, i \neq j \implies S_i \cap S_j = \emptyset$.}
We denote $\calP(\calX)$ the probability simplex over $\calX$, and $\calP_{+}(\calX) = \set{ \mu \in \calP(\calX) \colon \forall x \in \calX, \mu(x) > 0}$. 
All vectors will be written as row-vectors, for $x \in \calX$, $e_x$ is the unit vector verifying $\forall x' \in \calX$, $e_x(x') = \delta[x = x']$.
For some real matrices $A$ and $B$, \hl{$A(x,x')$ is the entry in the $x$-th row and $x'$-th column of $A$}, $\rho(A)$ is 
the spectral radius of $A$, 
$A \circ B$ is the Hadamard product of $A$ and $B$, \hl{for $t \in \R,\; A^{\circ t}$ is the matrix such that $\forall x,x' \in \calX, A^{\circ t}(x,x') = A(x,x')^t$,} $A > 0$ (resp. $A \geq 0$) means that $A$ is entry-wise positive (resp. non-negative).
We will routinely identify a function $f \colon \calX^2 \to \R$ with the linear operator \hl{$f \colon \R^ \calX \to \R^\calX, v \mapsto fv$ where $\forall x \in \calX, (fv)(x) = \sum_{x' \in \calX}f(x,x')v(x')$}, and use the shorthand $A(x,S) = \sum_{s \in S} A(x,s)$ when $S \subset \calX$.

\subsection{Irreducible Markov chains}
\label{section:irreducible-kernels}
We let $(\calY, \calE)$ be a strongly connected directed graph, where $\calY$ is its set of vertices, and $\calE \subset \calY^2$ is its set of directed edges.
Let $\calF(\calY, \calE)$ be the set of all real functions over the set $\calE$, identified with the \hl{set} of functions over $\calY^2$ that are null outside of $\calE$, and let $\calF_+(\calY, \calE) \subset \calF(\calY, \calE)$ be the subset of positive functions over $\calE$. 
For a fully connected graph, we write $\calF(\calY) = \calF(\calY, \calY^2)$, and we identify $\calF(\calY)$ with the set of \hl{$\abs{\calY} \times \abs{\calY}$} real square matrices. In particular,
the following inclusions hold
\begin{equation*}
\begin{split}
    \calF_+(\calY, \calE) \subset \calF(\calY, \calE) \subset \calF(\calY).
\end{split}
\end{equation*}
We write $\calW(\calY)$ for the set of (not necessarily irreducible) row-stochastic \hl{transition matrices} over the state space $\calY$, $\calW(\calY, \calE)$ for the subset of irreducible \hl{stochastic matrices} whose support is $\calE$, and $\calW_+(\calY)$ when $\calE = \calY^2$. For $P \in \calW(\calY)$, $P(y,y')$ corresponds to the transition probability\footnote{\hl{In the information theory literature, $P(x' | x)$ is sometimes used to denote $P(x, x')$. Our notation follows the applied probability literature (see e.g. }\hlc{\cite{levin2009markov}}\hl{). It will allow us to extend the notation to more general matrices, as we do when we discuss lumpable matrices, which do not bear the meaning of a conditional distribution.}} from state $y$ to state $y'$. Formally,
\begin{equation*}
\begin{split}
\calW(\calY) &\eqdef \set{ P \in \calF(\calY) \colon P \geq 0, \forall y \in \calY, e_y P 1 \trn = 1} ,\\
\calW(\calY, \calE) &\eqdef \calF_+(\calY, \calE) \cap \calW(\calY).
\end{split}
\end{equation*}
For $P \in \calW(\calY, \calE)$, there exists a unique $\pi \in \calP_+(\calY)$, such that $\pi P =\pi$ \citep[Corollary~1.17]{levin2009markov}, called the stationary distribution of $P$. 
We write 
$Q = \diag (\pi) P$ for the edge measure matrix, \citep[(7.5)]{levin2009markov}, which encodes stationary pair-probabilities of $P$, i.e. 
$Q(y, y') = \PR[\pi]{Y_t = y, Y_{t+1} = y'}$.
Following notations in \citet[Table~1]{wolfer2021information}, 
$$\calW \rev(\calY, \calE), \calW \bis (\calY, \calE), \calW \sym (\calY, \calE), \calW \iid (\calY, \calY^2),$$  will respectively denote the subsets of reversible ($Q = Q \trn$), doubly stochastic ($P \trn \in \calW(\calY, \calE \trn)$, \hl{where $\calE \trn \eqdef \set{(y', y) \colon  (y, y') \in \calE}$}), symmetric ($P = P \trn$), and memoryless ($P = 1 \trn \pi$) \hl{transition matrices} that are irreducible over $(\calY, \calE)$.
We define a stochastic rescaling \footnote{
This rescaling is closely related to the notion of exponential change of measure, also known as tilting, which can be traced back to \citet{miller1961convexity} in the context of \hl{stochastic matrices}.} mapping  
$\stoch$
that constructs a proper irreducible stochastic matrix from any non-negative irreducible matrix over
$(\calY, \calE)$,
\begin{equation}
\label{eq:stochastic-rescaling}
\begin{split}
    \stoch \colon \calF_+(\calY,\calE) &\to \calW(\calY,\calE) \\
    A(y,y') &\mapsto P(y,y') = \frac{A(y,y') v_A(y')}{\rho_A v_A(y)}, 
\end{split}
\end{equation}
where $\rho_A$ 
and $v_A$ are respectively the
Perron-Frobenius (PF) root and right PF eigenvector of $A$ \citep[Chapter~8]{meyer2000matrix}, which we will henceforth refer to as the right PF pair.
The mapping $\stoch$ is invariant under scaling of the argument by a positive constant or conjugation of the argument by a positive diagonal matrix. Namely, for all $\alpha \in \R_+$, $v \in \R_+^{\abs{\calY}}$, and $A \in \calF_+(\calY, \calE)$,
\begin{equation}
\label{eq:stoch-properties}
    \stoch(\alpha A) =  \stoch(A), \qquad \stoch(\diag(v)^{-1} A \diag (v)) =  \stoch(A). \\
\end{equation}

\hlp{\subsection{Information geometry of stochastic matrices}}
\label{section:preliminaries-information-geometry}

Following \citet{nagaoka2017exponential}, we take a differential geometry perspective, and \hl{regard} $ \calW(\calY, \calE)$ as a smooth manifold of dimension 
\begin{equation}
\label{eq:dimension-w}
d = \dim \calW(\calY, \calE) = \abs{\calE} - \abs{\calY},    
\end{equation}
where
for each $P \in \calW(\calY, \calE)$,
\hl{$T_P$ is the $d$-dimensional tangent space at $P$}.
On $\calW(\calY, \calE)$, we introduce the Riemannian metric $\mathfrak{g}$, expressed in some chart\footnote{\hl{As is customary in differential geometry, we will write $\theta$ for both the chart map and the coordinate.}} $\hlm{\theta} \colon \calW(\calY, \calE) \to \Theta \subset \R^d$ by
\begin{equation}
\label{eq:mc-fisher-metric}
\begin{split}
\fshr_{ij}(\theta) = \fshr_{ij}(P_\theta) &\eqdef \fshr_{P_\theta} \left( \partial_i, \partial_j \right) = \sum_{(y, y') \in \calE} Q_\theta(y, y') \partial_i \log P_\theta(y, y') \partial_j \log P_\theta(y, y'),
\end{split}
\end{equation}
where $\partial_i = \partial/\partial \theta_i$, $P_\theta$ is the stochastic matrix at coordinates $\theta$, and $Q_\theta$ is the edge measure pertaining to $P_\theta$ \footnote{We will use the notation $\fshr_{ij}(P_\theta)$ instead of $\fshr_{ij}(\theta)$ when we want to emphasize that we refer to the $\fshr$ metric defined on $\calW(\calY, \calE)$ containing $P_\theta$.}. \hl{In our study, we always assume that there exists a global chart\footnote{\hl{This assumption is not restrictive for our study, and is common in the
information geometry literature.}}.} We also recall the pair of torsion-free affine connections $\nabla^{(e)}$ and $\nabla^{(m)}$ respectively termed e-connection and m-connection, expressed by their Christoffel symbols as
\begin{equation}
\label{eq:mc-affine-connection}
\begin{split}
\Gamma^{(e)}_{ij, k}(\theta) &\eqdef \fshr_{P_\theta}\left( \nabla^{(e)}_{\partial_i} \partial_j, \partial_k \right) = \sum_{(y, y') \in \calE} \partial_i \partial_j \log P_\theta(y, y') \partial_k Q_\theta(y, y'), \\
\Gamma^{(m)}_{ij, k}(\theta) &\eqdef \fshr_{P_\theta} \left( \nabla^{(m)}_{\partial_i} \partial_j, \partial_k \right) = \sum_{(y, y') \in \calE} \partial_i \partial_j  Q_\theta(y, y') \partial_k \log P_\theta(y, y'). \\
\end{split}
\end{equation}
The Fisher metric and connections defined in \eqref{eq:mc-fisher-metric} and \eqref{eq:mc-affine-connection} are natural counterparts of the ones defined in the context of distributions \citep{amari2007methods}.
The connections $\nabla^{(e)}$ and $\nabla^{(m)}$ are dual with respect to $\mathfrak{g}$ in the sense that for any vector fields $X, Y, Z$,
\begin{equation*}
\begin{split}
X \mathfrak{g}(Y, Z) = \mathfrak{g}\left(\nabla^{(e)}_X Y, Z\right) + \mathfrak{g}\left( Y, \nabla^{(m)}_X Z\right).
\end{split}
\end{equation*}
The tuple 
\begin{equation*}
\left(\calW(\calY, \calE), \fshr, \nabla^{(e)}, \nabla^{(m)}\right)
\end{equation*}
encodes the information geometric structure of $\calW(\calY, \calE)$. Namely, it defines the notions of straight lines, parallelism, and distances between points.
The information divergence of a \hl{stochastic matrix} $P \in \calW(\calY, \calE)$ from another \hl{transition matrix} $P' \in \calW(\calY, \calE)$ is given by
\begin{equation}
\label{equation:kl-divergence}
    \kl{P}{P'} \eqdef \sum_{(y, y') \in \calE} \pi(y) P(y, y') \log \frac{P(y, y')}{P'(y, y')},
\end{equation}
while the dual divergence verifies $\kld{P}{P'} = \kl{P'}{P}$. 
Notably, \eqref{equation:kl-divergence} corresponds to the divergence rate  {\eqref{eq:divergence-rate-random-process}}  of the Markov processes induced from $P$ and $P'$,
\begin{equation*}
    \kl{P}{P'} = \lim_{k \to \infty} \frac{1}{k} \kl{Q^{k}}{(Q')^k}, 
\end{equation*}
where for $k \in \N$,
\begin{equation*}
    Q^{k}(y_1, y_2, \dots, y_k) = \pi(y_1) \prod_{t=1}^{k-1} P(y_t, y_{t+1}),
\end{equation*}
defines the distribution of stationary paths of length $k$ induced from $P$.

\subsection{Mixture family and exponential family}

\begin{definition}[e-family of \hl{stochastic matrices}]
\label{definition:e-family-parametric}
Let $\Theta \hlm{=} \R^d$.
We say that the parametric family of irreducible \hl{stochastic matrices} $$\calV_e = \set{P_\theta \colon \theta = (\theta^1, \dots, \theta^d) \in \Theta} \subset \calW( \calY, \calE)$$ 
is 
an exponential family (e-family) of \hl{stochastic matrices} with natural parameter $\theta$, when
there exist functions $K, g_1, \dots, g_d \in \calF(\calY, \calE)$, 
$R \in \R^{\Theta \times \calY}$ and $\psi \in \R^\Theta$,
    such that for any $ (y,y') \in \calE$ and $\theta \in \Theta$,
    \begin{equation}
        \label{eq:e-family-expression}
        \log P_\theta(y, y') =  K(y, y') + \sum_{i = 1}^{d} \theta^i g_i(y, y') + R(\theta, y') - R(\theta, y)  - \psi(\theta).
    \end{equation}
When fixing $\theta \in \Theta$, we will write for convenience $\psi_\theta$ for $\psi(\theta)$ and  $R_\theta$ for $R(\theta, \cdot) \in \R^ \calY$.
\end{definition}

\begin{remark} 
An e-family $\calV_e$ can be identified with some affine space as follows.
Denote, 
\begin{equation}
\label{definition:anti-shift-functions}
\begin{split}
\mathcal{N}(\calY, \calE) \eqdef \bigg\{ &h \in \calF(\calY, \calE) \colon \exists (c, f) \in (\R, \R^\calY), \\ &\forall (y, y') \in \calE, h(y, y') = f(y') - f(y) + c \bigg\}.
\end{split}
\end{equation}
Then 
$\mathcal{N}(\calY, \calE)$ is an $\abs{\calY}$-dimensional vector space \citep[Section~3]{nagaoka2017exponential}. \hl{We can thus} define
the quotient linear space 
$$\mathcal{G}(\calY, \calE) \eqdef \calF(\calY, \calE)/\mathcal{N}(\calY, \calE),$$
\hlp{and the diffeomorphism
\begin{equation*}
\begin{split}
    \Delta \colon \calG(\calY, \calE) &\to \calW(\calY, \calE), \qquad  g \mapsto \Delta(g) \eqdef \stoch(\exp \circ g),
    \end{split}
\end{equation*}
where $\circ$ here denotes function composition and $\exp$ is understood to be entry-wise.
There is a one-to-one correspondence between linear subspaces of $\calG(\calY, \calE)$ and e-families \citep[Theorem~2]{nagaoka2017exponential}.
}
We identify a coset of $\mathcal{G}(\calY, \calE)$ with a representative function in that coset.
For Definition~\ref{definition:e-family-parametric}, and unless stated otherwise, we will assume that the functions $g_1, \dots, g_d$ form an independent family in $\mathcal{G}( \calY, \calE)$.
In this case the family is said to be minimal and $\dim \calV_e = d$.
\end{remark}

\begin{definition}[m-family of \hl{stochastic matrices} {\citep[2.35]{amari2007methods}}]
\label{definition:m-family-parametric}
We say that a family of irreducible \hl{stochastic matrices} $\calV_m$ 
is a mixture family (m-family) of irreducible \hl{stochastic matrices} on $( \calY, \calE)$ when the following holds.
     There exists $C, F_1, \dots, F_d \in \calF( \calY, \calE)$,
such that $C, C + F_1, \dots, C + F_d$ are affinely independent, $$\sum_{(y,y') \in \calE} C(y,y') = 1, \qquad \sum_{(y,y') \in \calE} F_i(y,y') = 0, \forall i \in [d],$$
and 
$$\calV_m = \set{ P_\xi \in \calW( \calY, \calE) \colon Q_\xi = C + \sum_{i =1}^{d} \xi^i F_i, \xi \in \Xi }$$ 
where $\Xi= \set{ \xi \in \R^d\colon Q_\xi(y,y') > 0, \forall (y,y') \in \calE }$, and $Q_\xi$ is the edge measure that pertains to $P_\xi$.
Note that $\Xi$ is an open set, $\xi$ is called the mixture parameter and $d$ is the dimension of the family $\calV_m$.
See also \citep[Definition~1]{wolfer2021information} for alternative equivalent definitions of an m-family.
\end{definition}

\subsection{Lumping of Markov chains}

Let $P \in \calW(\calY, \calE)$, and let $Y_1, Y_2, \dots, Y_k$ be a sequence of observations sampled \hl{according to} $P$.
For $\phi \colon \calY \to \calX$, a possibly random mapping,
we call \hl{memoryless} data processing the application of $\phi$ onto the trajectory sampled from $P$.
\begin{equation*}
    \begin{split}
        Y_1, Y_2,\dots,Y_k \mapsto \phi(Y_1), \phi(Y_2),\dots,\phi(Y_k).
    \end{split}
\end{equation*}
This processing pertains to the action of a memoryless black box that takes a Markovian trajectory as input, and returns the image stochastic process.
Crucially, the output process is generally not Markovian itself \citep{kelly1982markovian}, but rather corresponds to a functional hidden Markov model.

Lumping \footnote{In this work, lumpability always refers to the notion of strong lumpability, which is independent of the starting distribution, as opposed to weak lumpability defined e.g. in \citet[\S~6.4]{kemeny1983finite}.} is a particular type of \hl{memoryless} deterministic processing that projects a chain onto a state space of smaller size where subsets of symbols are merged together. In the distribution setting, this operation is also referred to as a statistic \citep[(1.2)]{ay2017}.
More formally, let us \hl{define} a lumping as a surjective map
$$\kappa \colon \calY \to \calX = \kappa(\calY), \text{ with } \abs{\calX} \leq \abs{\calY} < \infty.$$
Observe that a lumping is completely characterized by a partition $\biguplus_{x \in \calX} \calS_x = \calY$, where for $x \in \calX, \calS_x = \kappa^{-1}(\set{x})$ is the collection of symbols in $\calY$ that are mapped to the new symbol $x$.
For $k \in \N$, a lumping map then naturally defines a \hl{data} processing function on the product space $\kappa_k \colon \calY^k \to \calX^k$. In particular, for  $(y,y') \in \calE$, $\kappa_2(y,y') = (\kappa(y), \kappa(y'))$.
When the Markovian nature of the lumped process is preserved ---regardless of the initial distribution--- we say that $P$ is $\kappa$-lumpable.
In this case, we define the lumped edge set $\calD$ by
\begin{equation}
\label{eq:lumped-edge-set}
        \calD \eqdef \kappa_2(\calE) = \set{(x,x') \in \calX^2 \colon \exists (y,y') \in \calE, \kappa(y) = x, \kappa(y') = x'} \subset \calX^2.
\end{equation}
Irreducibility is preserved, and there exists \hl{a push-forward stochastic matrix} $\kappa_\star P \in \calW(\calX, \calD)$, such that $(\kappa(Y_t))_{t \in \N}$ is sampled according to $\kappa_\star P$.
For a fixed lumping $\kappa$ the set of all $\kappa$-lumpable irreducible \hl{transition matrices} over $(\calY, \calE)$ is denoted $\calW_{\kappa}(\calY, \calE)$.
We also refer to the push-forward map
\begin{equation*}
    \begin{split}
        \kappa_\star \colon \calW_{\kappa}(\calY, \calE) &\to \calW(\kappa(\calY) = \calX, \kappa_2(\calE) = \calD). \\
        P &\mapsto \kappa_\star P.
    \end{split}
\end{equation*}
The following theorem characterizes lumpable chains in terms of their transition \hl{matrix}.
\begin{theorem}[{\citet[Theorem~6.3.2]{kemeny1983finite}}]
\label{theorem:equivalence-lumpable}
Let $\kappa \colon \calY \to \calX$ be a lumping function with associated partition $\biguplus_{x \in \calX} \calS_{x} = \calY$, and let $P \in \calW(\calY, \calE)$. Then $P \in \calW_\kappa(\calY, \calE)$ if and only if
for all $x,x' \in \calX$, and for all $y_1,y_2 \in \calS_x$,
$$P(y_1, \calS_{x'}) = P(y_2, \calS_{x'}),$$
and in this case,
$$\kappa_\star P(x,x') = P(y,\calS_{x'}), y \in \calS_{x}.$$
\end{theorem}

\begin{corollary}
\label{corollary:lumped-stationary-distribution-and-edge-measure}
Denoting $\kappa_\star \pi$ and $\kappa_\star Q$ the respective stationary distribution and edge measure of the lumped \hl{stochastic matrix} $\kappa_\star P$, it is straightforward to verify that for any $x,x' \in \calX$,
\begin{equation*}
\begin{split}
\kappa_\star \pi (x) = \pi(\calS_{x}), \qquad \kappa_\star Q (x,x') = Q(\calS_{x}, \calS_{x'}) \hlm{ \eqdef \sum_{(x,x') \in (\calS_{x}, \calS_{x'})} Q(y,y') }. \\
\end{split}
\end{equation*}
\end{corollary}

Note that when defining the lumping of a class of \hl{transition matrices}, it can be that $\calW_\kappa(\calY, \calE) = \emptyset$.

\begin{example}
Consider $\calX = \set{0,1}, \calY = \set{0,1,2}$, $\kappa$ such that $\kappa(0) = 0, \kappa(1) = \kappa(2) = 1$, and $\calE = \set{(0,0), (0,1), (1,0), (1,2), (2,1)}$. 
In this case, no element in $\calW(\calY, \calE)$ is lumpable.
\end{example}

To avoid this trivial case, we must assume that $\calE$ and $\kappa$ are compatible in the sense described in the next proposition. 

\begin{proposition}
\label{proposition:lumpable-not-empty}
$\calW_\kappa(\calY ,\calE) \neq \emptyset$ if and only if $\forall (x,x') \in \calD = \kappa_2(\calE)$, $\forall y \in \calS_x$, $\exists y' \in \calS_{x'}$ such that $(y,y') \in \calE$.
\end{proposition}

\begin{proof}
Easily verified with Theorem~\ref{theorem:equivalence-lumpable}.
\end{proof}

\section{Embeddings of stochastic matrices}
\label{section:embeddings}
Theorem~\ref{theorem:equivalence-lumpable} enables us to formulate lumping of a chain as a function of its \hl{transition matrix}. In a similar spirit, we define embeddings of Markov chains by viewing their \hl{transition matrices} as first-class citizens.
Namely, a \hl{stochastic matrix} embedding \label{term:Emb} \hl{$E$} is a map from a smooth submanifold of irreducible \hl{stochastic matrices} $\calV \subset \calW(\calX, \calD)$ to another family $\calW(\calY, \calE)$, and such that \hl{$E$} is a diffeomorphism \footnote{Note that \hl{$E$} being a diffeomorphism implies that necessarily $\dim \calW(\calY, \calE) \geq \dim \calV$.} onto its image.
\begin{equation*}
    \hlm{E} \colon \calV \to \calW(\calY, \calE).
\end{equation*}
By abuse of notation, \hl{$E\pi$} and \hl{$EQ$} will denote respectively the stationary distribution and edge measure of the embedded \hl{stochastic matrix} \hl{$EP$}.

\subsection{Markov embeddings}
\label{section:embeddings-markov}

We first recall the notion of a Markov morphism\footnote{The mappings introduced in Definition~\ref{definition:markov-morphism} are more commonly \citep{doi:10.1080/02331887808801428, campbell1986extended, ay2017} referred to as congruent Markov morphisms. To avoid confusion with the later defined congruent embeddings, we will simply refer to them as Markov morphisms.} in the context of probability distributions.

\begin{definition}[\citet{doi:10.1080/02331887808801428, campbell1986extended}]
\label{definition:markov-morphism}
Let there be some partition $\biguplus_{x\in \calX} \calS_{x} = \calY$.
To each $x \in \calX$, we associate a distribution $W^x \in \calP(\calY)$ concentrated on $\calS_x$.
We define the push-forward of $W=(W^x: x \in {\calX})$, referred to as Markov morphism, by
\begin{equation*}
    W_\star \colon \calP_+(\calX) \to \calP_+(\calY),
\end{equation*}
where for any $\mu \in \calP_+(\calX)$, and $y \in \calY$,
\begin{equation*}
    W_\star \mu(y) = \sum_{x \in \calX} W^{x}(y) \mu(x).
\end{equation*}
\end{definition}

\begin{example}
Let $\calX = \set{0, 1}$ and $\mu = (\eta, 1 - \eta)$ for $\eta \in (0,1)$.
We can embed $\mu$ into the larger space $\calY = \set{0,1,2}$ by considering the mapping induced from the channel 
\begin{equation*}
    W_p = \begin{pmatrix} 1 & 0 & 0 \\ 0 & p & 1 - p \end{pmatrix}, p \in (0,1).
\end{equation*}
The resulting distribution is $\mu_p = (\eta, p(1-\eta), (1-p)(1-\eta))$.
We can alternatively give a probabilistic definition for this embedding, where for a sequence of observations $X_1, X_2, \dots$ sampled iid from $\mu$, we record $Y_t = 0$ when $X_t = 0$, and flip a coin with bias $p$ when $X_t = 1$, $\PR{Y_t = 1 | X_{t} = 1} = p$, $\PR{Y_t = 2 | X_{t} = 1} = 1 - p$.
The new process $Y_1, Y_2, \dots$ is drawn from the embedded distribution $\mu_p$.
\end{example}

We now introduce Markov morphisms in the context of \hl{stochastic matrices}, which we term Markov embeddings, and which embed a stochastic matrix over some state space $\calX$ into a stochastic matrix over a space $\calY$ that is no smaller than $\calX$.

\hlc{The well-definedness of Definition~\ref{definition:markov-embeddings} will be discussed after the definition.}

\begin{definition}[Markov embedding]
\label{definition:markov-embeddings}
We call Markov embedding, a mapping
\begin{equation*}
    \begin{split}
        \Lambda_\star \colon  \calW(\calX, \calD)  &\to \calW(\calY, \calE) \\ P &\mapsto \Lambda_\star P, \\
    \end{split}
\end{equation*}
such that for any $(y,y') \in \calE$,
$$\Lambda_\star P(y,y') = P(\kappa(y), \kappa(y'))\Lambda(y,y'),$$
where $\kappa$ and $\Lambda$ satisfy the following requirements
\begin{enumerate}
    \item[$(i)$] $\kappa \colon \calY \to \calX$ is a lumping function for which $\kappa_2(\calE) = \calD$, and such that $\calE$ and $\kappa$ meet the non-triviality condition of Proposition~\ref{proposition:lumpable-not-empty}.
\item[$(ii)$] $\Lambda \in \calF_+(\calY, \calE)$.
\item[$(iii)$] 
Writing
$\biguplus_{x \in \calX} \calS_x = \calY$ for the associated partition of $\kappa$,
$\Lambda$ is such that for any $y \in \calY$ and $x' \in \calX$,
\begin{equation*}
\begin{split}
    (\kappa(y), x') \in \calD &\implies (\Lambda(y,y'))_{y' \in \calS_{x'}} \in \calP(\calS_{x'}).
\end{split}
\end{equation*}
\end{enumerate}
\end{definition}
Note that if $\calE$ and $\kappa$ fail to satisfy the condition of Proposition~\ref{proposition:lumpable-not-empty}, then no $\Lambda_\star$ satisfying $(iii)$ exists.
\label{term:MEmb}
It is instructive to observe that a valid $\Lambda$ corresponds to a block matrix,
\begin{equation*}
    \Lambda = \begin{pmatrix} 
    W_{1,1} & W_{1,2} & \cdots &  & W_{1,n} \\ 
    \vdots &  &  &  &  \\ 
    W_{x,1} & \cdots & W_{x,x'} &  \cdots & W_{x,n} \\ 
    \vdots &  & &   & \vdots \\ 
    W_{n,1} &  & \cdots&  & W_{n,n} \\ 
    \end{pmatrix},
\end{equation*}
where each block $W_{x,x'}$ is either a channel from $\calS_x$ to $\calS_{x'}$ when $(x,x') \in \calD$, or is set to $0$ when $(x,x')\not \in \calD$.

\hlp{
\paragraph{Well-definedness of Definition~\ref{definition:markov-embeddings}}

It is straightforward to verify that  Markov embeddings are well-defined in the sense that a stochastic matrix irreducible in $(\calX, \calD)$ is mapped to a stochastic matrix over $\calW(\calY, \calE)$}\footnote{Note that we can embed a sub-family $\calV \subset \calW(\calX, \calD)$ into $\calW(\calY, \calE)$ by using a restriction of $\Lambda_\star$ to $\calV$.}
Condition $(ii)$ ensures that Markov embeddings preserve irreducibility.
Crucially, when $P \in \calW(\calX, \calD)$, $\Lambda_\star P \in \calW_\kappa(\calY, \calE)$, where
$\kappa$ is the lumping function associated with the embedding defined at Definition~\ref{definition:markov-embeddings}-$(i)$.
We say that an embedding is $\kappa$-compatible \label{term:Embk} \label{term:MEmbk} when it produces $\kappa$-lumpable \hl{stochastic matrices}.

\begin{example}[The weather model]
\label{example:weather-model}
Consider the simple Markovian weather model with two states \emph{Sun} (numbered as $0$) and \emph{Rain} (numbered as $1$), where the probability of weather conditions given the preceding day are given by the following transition matrix,
\begin{equation*}
    P = \left(
\begin{tblr}{c c}
    4/5 & 1/5 \\
    1/2 & 1/2
    \end{tblr}
\right).
\end{equation*}
Suppose we are interested in a finer model, where there are two types of rainy states called \emph{Showers} and \emph{Thunderstorm}. This corresponds to splitting the state \emph{Rain}, and refining the transition probabilities to the newly defined states.
We can represent this splitting operation naturally by a Markov embedding where $\calS_0 = \set{0}, \calS_{1} = \set{1, 2}$, and
\begin{equation*}
    \Lambda = \left(
\begin{tblr}{c|[dashed]cc}
1 & 1/2 & 1/2 \\ \hline[dashed]
1 & 3/5 & 2/5 \\
1 & 2/5 & 3/5 \\
\end{tblr}
\right).
\end{equation*}
See Figure~\ref{figure:embedding-weather-model}. The transition matrix of the embedded Markov chain is 
\begin{equation*}
    \Lambda_\star P = \left(
\begin{tblr}{c|[dashed]cc}
4/5 & 1/10 & 1/10 \\ \hline[dashed]
1/2 & 3/10 & 1/5 \\
1/2 & 1/5 & 3/10 \\
\end{tblr}
\right).
\end{equation*}
\end{example}

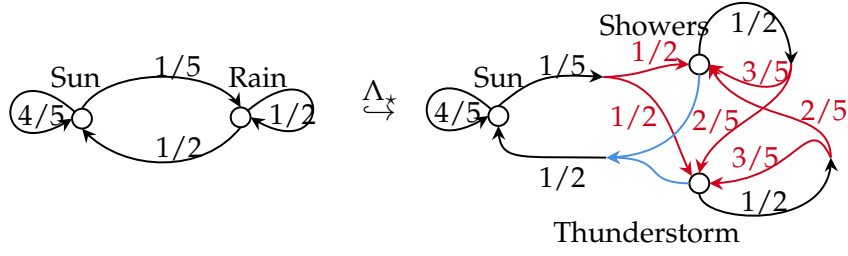
\begin{figure}%
\centering

\tikzset{every picture/.style={line width=0.75pt}} %

\begin{tikzpicture}[x=0.50pt,y=0.50pt,yscale=-1,xscale=1]

\draw   (58,101.5) .. controls (58,97.36) and (61.36,94) .. (65.5,94) .. controls (69.64,94) and (73,97.36) .. (73,101.5) .. controls (73,105.64) and (69.64,109) .. (65.5,109) .. controls (61.36,109) and (58,105.64) .. (58,101.5) -- cycle ;
\draw   (178,100.5) .. controls (178,96.36) and (181.36,93) .. (185.5,93) .. controls (189.64,93) and (193,96.36) .. (193,100.5) .. controls (193,104.64) and (189.64,108) .. (185.5,108) .. controls (181.36,108) and (178,104.64) .. (178,100.5) -- cycle ;
\draw    (65.5,94) .. controls (83.54,63.12) and (162.89,61.4) .. (183.99,90.69) ;
\draw [shift={(185.5,93)}, rotate = 239.71] [fill={rgb, 255:red, 0; green, 0; blue, 0 }  ][line width=0.08]  [draw opacity=0] (10.72,-5.15) -- (0,0) -- (10.72,5.15) -- (7.12,0) -- cycle    ;
\draw    (67.77,111.44) .. controls (98.08,142.41) and (157.71,143.42) .. (185.5,108) ;
\draw [shift={(65.5,109)}, rotate = 48.49] [fill={rgb, 255:red, 0; green, 0; blue, 0 }  ][line width=0.08]  [draw opacity=0] (10.72,-5.15) -- (0,0) -- (10.72,5.15) -- (7.12,0) -- cycle    ;
\draw    (60,96.33) .. controls (0.6,43.86) and (-8.81,142.33) .. (56.01,102.74) ;
\draw [shift={(58,101.5)}, rotate = 147.41] [fill={rgb, 255:red, 0; green, 0; blue, 0 }  ][line width=0.08]  [draw opacity=0] (10.72,-5.15) -- (0,0) -- (10.72,5.15) -- (7.12,0) -- cycle    ;
\draw    (191,95.33) .. controls (254.36,42.86) and (257.93,140.35) .. (194.93,101.71) ;
\draw [shift={(193,100.5)}, rotate = 32.76] [fill={rgb, 255:red, 0; green, 0; blue, 0 }  ][line width=0.08]  [draw opacity=0] (10.72,-5.15) -- (0,0) -- (10.72,5.15) -- (7.12,0) -- cycle    ;
\draw   (373,99.5) .. controls (373,95.36) and (376.36,92) .. (380.5,92) .. controls (384.64,92) and (388,95.36) .. (388,99.5) .. controls (388,103.64) and (384.64,107) .. (380.5,107) .. controls (376.36,107) and (373,103.64) .. (373,99.5) -- cycle ;
\draw   (525,60.5) .. controls (525,56.36) and (528.36,53) .. (532.5,53) .. controls (536.64,53) and (540,56.36) .. (540,60.5) .. controls (540,64.64) and (536.64,68) .. (532.5,68) .. controls (528.36,68) and (525,64.64) .. (525,60.5) -- cycle ;
\draw    (380.5,92) .. controls (407.45,69.46) and (418.55,67.09) .. (457.58,69.83) ;
\draw [shift={(460,70)}, rotate = 184.18] [fill={rgb, 255:red, 0; green, 0; blue, 0 }  ][line width=0.08]  [draw opacity=0] (10.72,-5.15) -- (0,0) -- (10.72,5.15) -- (7.12,0) -- cycle    ;
\draw    (380.61,110.2) .. controls (381.9,134.61) and (393.84,130.04) .. (462,131) ;
\draw [shift={(380.5,107)}, rotate = 88.98] [fill={rgb, 255:red, 0; green, 0; blue, 0 }  ][line width=0.08]  [draw opacity=0] (10.72,-5.15) -- (0,0) -- (10.72,5.15) -- (7.12,0) -- cycle    ;
\draw    (375,94.33) .. controls (315.6,41.86) and (306.19,140.33) .. (371.01,100.74) ;
\draw [shift={(373,99.5)}, rotate = 147.41] [fill={rgb, 255:red, 0; green, 0; blue, 0 }  ][line width=0.08]  [draw opacity=0] (10.72,-5.15) -- (0,0) -- (10.72,5.15) -- (7.12,0) -- cycle    ;
\draw    (532.5,53) .. controls (532.01,1.78) and (599.92,-1.9) .. (601.95,58.21) ;
\draw [shift={(602,61)}, rotate = 270] [fill={rgb, 255:red, 0; green, 0; blue, 0 }  ][line width=0.08]  [draw opacity=0] (10.72,-5.15) -- (0,0) -- (10.72,5.15) -- (7.12,0) -- cycle    ;
\draw   (525,150.5) .. controls (525,146.36) and (528.36,143) .. (532.5,143) .. controls (536.64,143) and (540,146.36) .. (540,150.5) .. controls (540,154.64) and (536.64,158) .. (532.5,158) .. controls (528.36,158) and (525,154.64) .. (525,150.5) -- cycle ;
\draw    (532.5,158) .. controls (532.99,180.66) and (630.99,187.79) .. (632.02,133.52) ;
\draw [shift={(632,131)}, rotate = 87.99] [fill={rgb, 255:red, 0; green, 0; blue, 0 }  ][line width=0.08]  [draw opacity=0] (10.72,-5.15) -- (0,0) -- (10.72,5.15) -- (7.12,0) -- cycle    ;
\draw [color={rgb, 255:red, 208; green, 2; blue, 27 }  ,draw opacity=1 ]   (460,70) .. controls (493.6,69.04) and (499.55,59.79) .. (522.09,60.37) ;
\draw [shift={(525,60.5)}, rotate = 183.43] [fill={rgb, 255:red, 208; green, 2; blue, 27 }  ,fill opacity=1 ][line width=0.08]  [draw opacity=0] (10.72,-5.15) -- (0,0) -- (10.72,5.15) -- (7.12,0) -- cycle    ;
\draw [color={rgb, 255:red, 208; green, 2; blue, 27 }  ,draw opacity=1 ]   (460,70) .. controls (506.32,70.97) and (512.6,118.49) .. (525.56,142.49) ;
\draw [shift={(527,145)}, rotate = 238.67] [fill={rgb, 255:red, 208; green, 2; blue, 27 }  ,fill opacity=1 ][line width=0.08]  [draw opacity=0] (10.72,-5.15) -- (0,0) -- (10.72,5.15) -- (7.12,0) -- cycle    ;
\draw [color={rgb, 255:red, 74; green, 144; blue, 226 }  ,draw opacity=1 ]   (465.25,130.9) .. controls (517.59,128.6) and (532.01,94.46) .. (532.5,68) ;
\draw [shift={(462,131)}, rotate = 358.96] [fill={rgb, 255:red, 74; green, 144; blue, 226 }  ,fill opacity=1 ][line width=0.08]  [draw opacity=0] (10.72,-5.15) -- (0,0) -- (10.72,5.15) -- (7.12,0) -- cycle    ;
\draw [color={rgb, 255:red, 74; green, 144; blue, 226 }  ,draw opacity=1 ]   (465.54,130.96) .. controls (509.79,130.99) and (486.98,150.01) .. (525,150.5) ;
\draw [shift={(462,131)}, rotate = 358.83] [fill={rgb, 255:red, 74; green, 144; blue, 226 }  ,fill opacity=1 ][line width=0.08]  [draw opacity=0] (10.72,-5.15) -- (0,0) -- (10.72,5.15) -- (7.12,0) -- cycle    ;
\draw [color={rgb, 255:red, 208; green, 2; blue, 27 }  ,draw opacity=1 ]   (602,60) .. controls (602,79.6) and (578.95,85.75) .. (542.26,61.99) ;
\draw [shift={(540,60.5)}, rotate = 33.86] [fill={rgb, 255:red, 208; green, 2; blue, 27 }  ,fill opacity=1 ][line width=0.08]  [draw opacity=0] (10.72,-5.15) -- (0,0) -- (10.72,5.15) -- (7.12,0) -- cycle    ;
\draw [color={rgb, 255:red, 208; green, 2; blue, 27 }  ,draw opacity=1 ]   (602,60) .. controls (602.98,95.1) and (543.11,105.48) .. (533.16,140.27) ;
\draw [shift={(532.5,143)}, rotate = 281.46] [fill={rgb, 255:red, 208; green, 2; blue, 27 }  ,fill opacity=1 ][line width=0.08]  [draw opacity=0] (10.72,-5.15) -- (0,0) -- (10.72,5.15) -- (7.12,0) -- cycle    ;
\draw [color={rgb, 255:red, 208; green, 2; blue, 27 }  ,draw opacity=1 ]   (632,131) .. controls (632,90.61) and (564.08,107.47) .. (541.02,62.59) ;
\draw [shift={(540,60.5)}, rotate = 65.15] [fill={rgb, 255:red, 208; green, 2; blue, 27 }  ,fill opacity=1 ][line width=0.08]  [draw opacity=0] (10.72,-5.15) -- (0,0) -- (10.72,5.15) -- (7.12,0) -- cycle    ;
\draw [color={rgb, 255:red, 208; green, 2; blue, 27 }  ,draw opacity=1 ]   (632,131) .. controls (614.27,93.57) and (607.21,148.31) .. (542.97,150.44) ;
\draw [shift={(540,150.5)}, rotate = 359.57] [fill={rgb, 255:red, 208; green, 2; blue, 27 }  ,fill opacity=1 ][line width=0.08]  [draw opacity=0] (10.72,-5.15) -- (0,0) -- (10.72,5.15) -- (7.12,0) -- cycle    ;

\draw (40,60) node [anchor=north west][inner sep=0.75pt]   [align=left] {Sun};
\draw (174,60) node [anchor=north west][inner sep=0.75pt]   [align=left] {Rain};
\draw (15,88) node [anchor=north west][inner sep=0.75pt]   [align=left] {4/5};
\draw (118,51) node [anchor=north west][inner sep=0.75pt]   [align=left] {1/5};
\draw (205,87) node [anchor=north west][inner sep=0.75pt]   [align=left] {1/2};
\draw (119,112) node [anchor=north west][inner sep=0.75pt]   [align=left] {1/2};
\draw (360,60) node [anchor=north west][inner sep=0.75pt]   [align=left] {Sun};
\draw (455,20) node [anchor=north west][inner sep=0.75pt]   [align=left] {Showers};
\draw (330,86) node [anchor=north west][inner sep=0.75pt]   [align=left] {4/5};
\draw (409,50) node [anchor=north west][inner sep=0.75pt]   [align=left] {1/5};
\draw (554,18) node [anchor=north west][inner sep=0.75pt]   [align=left] {1/2};
\draw (409,135) node [anchor=north west][inner sep=0.75pt]   [align=left] {1/2};
\draw (562,153) node [anchor=north west][inner sep=0.75pt]   [align=left] {1/2};
\draw (479,40) node [anchor=north west][inner sep=0.75pt]  [color={rgb, 255:red, 208; green, 2; blue, 27 }  ,opacity=1 ] [align=left] {1/2};
\draw (463,89) node [anchor=north west][inner sep=0.75pt]  [color={rgb, 255:red, 208; green, 2; blue, 27 }  ,opacity=1 ] [align=left] {1/2};
\draw (420,177) node [anchor=north west][inner sep=0.75pt]   [align=left] {Thunderstorm};
\draw (563,55) node [anchor=north west][inner sep=0.75pt]  [color={rgb, 255:red, 208; green, 2; blue, 27 }  ,opacity=1 ] [align=left] {3/5};
\draw (525,92) node [anchor=north west][inner sep=0.75pt]  [color={rgb, 255:red, 208; green, 2; blue, 27 }  ,opacity=1 ] [align=left] {2/5};
\draw (607,84) node [anchor=north west][inner sep=0.75pt]  [color={rgb, 255:red, 208; green, 2; blue, 27 }  ,opacity=1 ] [align=left] {2/5};
\draw (555,120) node [anchor=north west][inner sep=0.75pt]  [color={rgb, 255:red, 208; green, 2; blue, 27 }  ,opacity=1 ] [align=left] {3/5};
\draw (275,70) node [anchor=north west][inner sep=0.75pt]   [align=left] {$\Lambda_\star$};
\draw (275,90) node [anchor=north west][inner sep=0.75pt]   [align=left] {$\hookrightarrow$};

\end{tikzpicture}

\caption{Embedding the weather model (Example~\ref{example:weather-model}).}
\label{figure:embedding-weather-model}

\end{figure}

\begin{example}[Non-full support range]
Let $\calX = \set{0, 1}$, $\calY = \set{0, 1, 2}$, 
$$\calE = \set{(0,0), (0,2), (1,0), (1,2), (2,0), (2,1)} \neq \calY^2,$$ 
and the lumping function $\kappa \colon \calY \to \calX$ such that $\kappa(0) = 0, \kappa(1) = \kappa(2) = 1$.
Then $\calD = \calX^2$ and from \eqref{eq:dimension-w}, $\dim \calW(\calX, \calD) = 2, \dim \calW(\calY, \calE) = 3$, and
any elements $\bar{P} \in \calW(\calX, \calD)$ and $P \in \calW(\calY, \calE)$ can be written,
\begin{equation*}
    \bar{P} = \left( \begin{tblr}{cc}
1 - p & p  \\
q & 1 - q 
\end{tblr}\right), \qquad P =\left( \begin{tblr}{ccc}
1-p & 0 & p \\
q_1 & 0 & 1-q_1 \\
q_2 & 1 - q_2 & 0 \\
\end{tblr}\right),
\end{equation*}
for some $p, q, q_1, q_2 \in (0,1)$.
The only possible Markov embedding is defined by the matrix
\begin{equation*}
    \Lambda = \left( \begin{tblr}{c|[dashed]cc}
1 & 0 & 1 \\ \hline[dashed]
1 & 0 & 1 \\
1 & 1 & 0 \\
\end{tblr}\right).
\end{equation*}
\hl{An} element $P_\kappa \in \calW_\kappa(\calY, \calE)$ can thus be expressed as
\begin{equation*}
    P_\kappa = \left( \begin{tblr}{c|[dashed]cc}
1-p & 0 & p \\ \hline[dashed]
q & 0 & 1-q \\
q & 1 - q & 0 \\
\end{tblr}\right),
\end{equation*}
for some $p, q\in (0,1)$, hence $\dim \calW_\kappa (\calY, \calE) = 2 < \dim \calW(\calY, \calE)$. 
\end{example}

\begin{example}[Non-full support domain]
Let $\calX = \set{0, 1}$, $\calY = \set{0, 1, 2}$, 
$$\calE = \set{(0,0), (0,1), (0,2), (1,0), (2,0)} \neq \calY^2,$$ 
and the lumping function $\kappa \colon \calY \to \calX$ such that $\kappa(0) = 0, \kappa(1) = \kappa(2) = 1$.
Notice that 
$$\calD = \set{ (0,0), (0,1), (1,0) } \neq \calX^2,$$
and that $\dim \calW(\calX, \calD) = 1$, $\dim \calW(\calY, \calE) = 2$.
In this case, Markov embeddings are defined by $\kappa$ and the matrix
\begin{equation*}
    \Lambda = \left( \begin{tblr}{c|[dashed]cc}
1 & 1-\lambda & \lambda \\ \hline[dashed]
1 & 0 & 0 \\
1 & 0 & 0 \\
\end{tblr}\right),
\end{equation*}
that enjoys a degree of freedom $\lambda \in (0,1)$.
Any member $P$ of $\calW_\kappa(\calY, \calE)$ is expressed (see forthcoming Lemma~\ref{lemma:construct-canonical-embedding}) as
\begin{equation*}
    P = \left( \begin{tblr}{c|[dashed]cc}
1 - p & (1-\lambda)p & \lambda p \\ \hline[dashed]
1 & 0 & 0 \\
1 & 0 & 0 \\
\end{tblr}\right), (p, \lambda) \in (0,1)^2,
\end{equation*}
thus $\dim \calW_\kappa(\calY, \calE) = 2$, and $\calW_\kappa(\calY, \calE) = \calW(\calY, \calE)$.
\end{example}

We proceed to give a more operational definition of a Markov embedding by showing that 
it can be interpreted as
randomly embedding the sequence of observations of a Markov chain into a larger space.

Let $P \in \calW(\calX, \calD)$, and $\mu \in \calP_+(\calX)$ some initial distribution. Consider a single Markovian trajectory $X_1, X_2, \dots$ sampled according to $P$ and started from a state sampled from $\mu$. Let $\nu \in \calP_+(\calY)$ be such that for any $x \in \calX$,
\begin{equation*}
    \mu(x) = \sum_{y \in \calS_{x}} \nu(y),
\end{equation*}
and for $x \in \calX$, define the conditional probability distribution $\nu_{| x} \in \calP(\calY)$ concentrated on $\calS_{x}$,
\begin{equation*}
    \nu_{| x}(y)  = \frac{\nu(y)}{\mu(x)}.
\end{equation*}

\begin{figure}
    \centering
    \hlp{

\tikzset{every picture/.style={line width=0.75pt}} %

\begin{tikzpicture}[x=0.55pt,y=0.55pt,yscale=-1,xscale=1]

\draw   (74,74) -- (104.1,74) -- (91.2,90.57) -- (61.1,90.57) -- cycle ;
\draw   (89.9,159.43) -- (120,159.43) -- (107.1,176) -- (77,176) -- cycle ;
\draw   (77,176) -- (107.1,176) -- (94.2,192.57) -- (64.1,192.57) -- cycle ;
\draw   (102.8,142.86) -- (132.9,142.86) -- (120,159.43) -- (89.9,159.43) -- cycle ;
\draw  [color={rgb, 255:red, 208; green, 2; blue, 27 }  ,draw opacity=1 ] (128.6,109.71) -- (158.7,109.71) -- (145.8,126.29) -- (115.7,126.29) -- cycle ;
\draw   (64.1,192.57) -- (94.2,192.57) -- (81.3,209.14) -- (51.2,209.14) -- cycle ;
\draw [color={rgb, 255:red, 208; green, 2; blue, 27 }  ,draw opacity=1 ] [dash pattern={on 0.84pt off 2.51pt}]  (117,57.43) -- (158.7,109.71) ;
\draw [color={rgb, 255:red, 208; green, 2; blue, 27 }  ,draw opacity=1 ] [dash pattern={on 0.84pt off 2.51pt}]  (74,74) -- (102.8,142.86) ;
\draw  [color={rgb, 255:red, 208; green, 2; blue, 27 }  ,draw opacity=1 ][fill={rgb, 255:red, 255; green, 255; blue, 255 }  ,fill opacity=1 ] (86.9,57.43) -- (117,57.43) -- (104.1,74) -- (74,74) -- cycle ;
\draw  [color={rgb, 255:red, 208; green, 2; blue, 27 }  ,draw opacity=1 ] (115.7,126.29) -- (145.8,126.29) -- (132.9,142.86) -- (102.8,142.86) -- cycle ;
\draw  [color={rgb, 255:red, 208; green, 2; blue, 27 }  ,draw opacity=1 ] (198.9,159.43) -- (229,159.43) -- (216.1,176) -- (186,176) -- cycle ;
\draw  [color={rgb, 255:red, 208; green, 2; blue, 27 }  ,draw opacity=1 ] (186,176) -- (216.1,176) -- (203.2,192.57) -- (173.1,192.57) -- cycle ;
\draw  [color={rgb, 255:red, 0; green, 0; blue, 0 }  ,draw opacity=1 ] (237.6,109.71) -- (267.7,109.71) -- (254.8,126.29) -- (224.7,126.29) -- cycle ;
\draw  [color={rgb, 255:red, 208; green, 2; blue, 27 }  ,draw opacity=1 ] (173.1,192.57) -- (203.2,192.57) -- (190.3,209.14) -- (160.2,209.14) -- cycle ;
\draw [color={rgb, 255:red, 208; green, 2; blue, 27 }  ,draw opacity=1 ] [dash pattern={on 0.84pt off 2.51pt}]  (200.2,90.57) -- (190.3,209.14) ;
\draw [color={rgb, 255:red, 208; green, 2; blue, 27 }  ,draw opacity=1 ] [dash pattern={on 0.84pt off 2.51pt}]  (213.1,74) -- (241.9,142.86) ;
\draw [color={rgb, 255:red, 208; green, 2; blue, 27 }  ,draw opacity=1 ] [dash pattern={on 0.84pt off 2.51pt}]  (170.1,90.57) -- (160.2,209.14) ;
\draw  [color={rgb, 255:red, 0; green, 0; blue, 0 }  ,draw opacity=1 ][fill={rgb, 255:red, 255; green, 255; blue, 255 }  ,fill opacity=1 ] (195.9,57.43) -- (226,57.43) -- (213.1,74) -- (183,74) -- cycle ;
\draw  [color={rgb, 255:red, 0; green, 0; blue, 0 }  ,draw opacity=1 ] (224.7,126.29) -- (254.8,126.29) -- (241.9,142.86) -- (211.8,142.86) -- cycle ;
\draw   (401,74) -- (431.1,74) -- (418.2,90.57) -- (388.1,90.57) -- cycle ;
\draw   (416.9,159.43) -- (447,159.43) -- (434.1,176) -- (404,176) -- cycle ;
\draw   (404,176) -- (434.1,176) -- (421.2,192.57) -- (391.1,192.57) -- cycle ;
\draw   (429.8,142.86) -- (459.9,142.86) -- (447,159.43) -- (416.9,159.43) -- cycle ;
\draw  [color={rgb, 255:red, 208; green, 2; blue, 27 }  ,draw opacity=1 ] (455.6,109.71) -- (485.7,109.71) -- (472.8,126.29) -- (442.7,126.29) -- cycle ;
\draw   (391.1,192.57) -- (421.2,192.57) -- (408.3,209.14) -- (378.2,209.14) -- cycle ;
\draw [color={rgb, 255:red, 208; green, 2; blue, 27 }  ,draw opacity=1 ] [dash pattern={on 0.84pt off 2.51pt}]  (444,57.43) -- (485.7,109.71) ;
\draw [color={rgb, 255:red, 208; green, 2; blue, 27 }  ,draw opacity=1 ] [dash pattern={on 0.84pt off 2.51pt}]  (431.1,74) -- (459.9,142.86) ;
\draw [color={rgb, 255:red, 208; green, 2; blue, 27 }  ,draw opacity=1 ] [dash pattern={on 0.84pt off 2.51pt}]  (401,74) -- (429.8,142.86) ;
\draw  [color={rgb, 255:red, 208; green, 2; blue, 27 }  ,draw opacity=1 ][fill={rgb, 255:red, 255; green, 255; blue, 255 }  ,fill opacity=1 ] (413.9,57.43) -- (444,57.43) -- (431.1,74) -- (401,74) -- cycle ;
\draw  [color={rgb, 255:red, 208; green, 2; blue, 27 }  ,draw opacity=1 ] (442.7,126.29) -- (472.8,126.29) -- (459.9,142.86) -- (429.8,142.86) -- cycle ;
\draw  [color={rgb, 255:red, 208; green, 2; blue, 27 }  ,draw opacity=1 ] (183,74) -- (213.1,74) -- (200.2,90.57) -- (170.1,90.57) -- cycle ;
\draw  [color={rgb, 255:red, 208; green, 2; blue, 27 }  ,draw opacity=1 ] (211.8,142.86) -- (241.9,142.86) -- (229,159.43) -- (198.9,159.43) -- cycle ;
\draw  [color={rgb, 255:red, 74; green, 144; blue, 226 }  ,draw opacity=1 ] (94,65) .. controls (94,63.9) and (94.9,63) .. (96,63) .. controls (97.1,63) and (98,63.9) .. (98,65) .. controls (98,66.1) and (97.1,67) .. (96,67) .. controls (94.9,67) and (94,66.1) .. (94,65)(91,65) .. controls (91,62.24) and (93.24,60) .. (96,60) .. controls (98.76,60) and (101,62.24) .. (101,65) .. controls (101,67.76) and (98.76,70) .. (96,70) .. controls (93.24,70) and (91,67.76) .. (91,65) ;
\draw  [color={rgb, 255:red, 74; green, 144; blue, 226 }  ,draw opacity=1 ] (122.3,134.57) .. controls (122.3,133.47) and (123.2,132.57) .. (124.3,132.57) .. controls (125.4,132.57) and (126.3,133.47) .. (126.3,134.57) .. controls (126.3,135.68) and (125.4,136.57) .. (124.3,136.57) .. controls (123.2,136.57) and (122.3,135.68) .. (122.3,134.57)(119.3,134.57) .. controls (119.3,131.81) and (121.54,129.57) .. (124.3,129.57) .. controls (127.06,129.57) and (129.3,131.81) .. (129.3,134.57) .. controls (129.3,137.33) and (127.06,139.57) .. (124.3,139.57) .. controls (121.54,139.57) and (119.3,137.33) .. (119.3,134.57) ;
\draw [color={rgb, 255:red, 208; green, 2; blue, 27 }  ,draw opacity=1 ] [dash pattern={on 0.84pt off 2.51pt}]  (104.1,74) -- (132.9,142.86) ;
\draw  [color={rgb, 255:red, 208; green, 2; blue, 27 }  ,draw opacity=1 ] (307.9,159.43) -- (338,159.43) -- (325.1,176) -- (295,176) -- cycle ;
\draw  [color={rgb, 255:red, 208; green, 2; blue, 27 }  ,draw opacity=1 ] (295,176) -- (325.1,176) -- (312.2,192.57) -- (282.1,192.57) -- cycle ;
\draw  [color={rgb, 255:red, 0; green, 0; blue, 0 }  ,draw opacity=1 ] (346.6,109.71) -- (376.7,109.71) -- (363.8,126.29) -- (333.7,126.29) -- cycle ;
\draw  [color={rgb, 255:red, 208; green, 2; blue, 27 }  ,draw opacity=1 ] (282.1,192.57) -- (312.2,192.57) -- (299.3,209.14) -- (269.2,209.14) -- cycle ;
\draw [color={rgb, 255:red, 208; green, 2; blue, 27 }  ,draw opacity=1 ] [dash pattern={on 0.84pt off 2.51pt}]  (309.2,90.57) -- (299.3,209.14) ;
\draw [color={rgb, 255:red, 208; green, 2; blue, 27 }  ,draw opacity=1 ] [dash pattern={on 0.84pt off 2.51pt}]  (322.1,74) -- (350.9,142.86) ;
\draw [color={rgb, 255:red, 208; green, 2; blue, 27 }  ,draw opacity=1 ] [dash pattern={on 0.84pt off 2.51pt}]  (279.1,90.57) -- (269.2,209.14) ;
\draw  [color={rgb, 255:red, 0; green, 0; blue, 0 }  ,draw opacity=1 ][fill={rgb, 255:red, 255; green, 255; blue, 255 }  ,fill opacity=1 ] (304.9,57.43) -- (335,57.43) -- (322.1,74) -- (292,74) -- cycle ;
\draw  [color={rgb, 255:red, 0; green, 0; blue, 0 }  ,draw opacity=1 ] (333.7,126.29) -- (363.8,126.29) -- (350.9,142.86) -- (320.8,142.86) -- cycle ;
\draw  [color={rgb, 255:red, 208; green, 2; blue, 27 }  ,draw opacity=1 ] (292,74) -- (322.1,74) -- (309.2,90.57) -- (279.1,90.57) -- cycle ;
\draw  [color={rgb, 255:red, 208; green, 2; blue, 27 }  ,draw opacity=1 ] (320.8,142.86) -- (350.9,142.86) -- (338,159.43) -- (307.9,159.43) -- cycle ;
\draw  [color={rgb, 255:red, 74; green, 144; blue, 226 }  ,draw opacity=1 ] (205.5,167.71) .. controls (205.5,166.61) and (206.4,165.71) .. (207.5,165.71) .. controls (208.6,165.71) and (209.5,166.61) .. (209.5,167.71) .. controls (209.5,168.82) and (208.6,169.71) .. (207.5,169.71) .. controls (206.4,169.71) and (205.5,168.82) .. (205.5,167.71)(202.5,167.71) .. controls (202.5,164.95) and (204.74,162.71) .. (207.5,162.71) .. controls (210.26,162.71) and (212.5,164.95) .. (212.5,167.71) .. controls (212.5,170.48) and (210.26,172.71) .. (207.5,172.71) .. controls (204.74,172.71) and (202.5,170.48) .. (202.5,167.71) ;
\draw   (509,74) -- (539.1,74) -- (526.2,90.57) -- (496.1,90.57) -- cycle ;
\draw   (524.9,159.43) -- (555,159.43) -- (542.1,176) -- (512,176) -- cycle ;
\draw   (512,176) -- (542.1,176) -- (529.2,192.57) -- (499.1,192.57) -- cycle ;
\draw   (537.8,142.86) -- (567.9,142.86) -- (555,159.43) -- (524.9,159.43) -- cycle ;
\draw  [color={rgb, 255:red, 208; green, 2; blue, 27 }  ,draw opacity=1 ] (563.6,109.71) -- (593.7,109.71) -- (580.8,126.29) -- (550.7,126.29) -- cycle ;
\draw   (499.1,192.57) -- (529.2,192.57) -- (516.3,209.14) -- (486.2,209.14) -- cycle ;
\draw [color={rgb, 255:red, 208; green, 2; blue, 27 }  ,draw opacity=1 ] [dash pattern={on 0.84pt off 2.51pt}]  (552,57.43) -- (593.7,109.71) ;
\draw [color={rgb, 255:red, 208; green, 2; blue, 27 }  ,draw opacity=1 ] [dash pattern={on 0.84pt off 2.51pt}]  (539.1,74) -- (567.9,142.86) ;
\draw [color={rgb, 255:red, 208; green, 2; blue, 27 }  ,draw opacity=1 ] [dash pattern={on 0.84pt off 2.51pt}]  (509,74) -- (537.8,142.86) ;
\draw  [color={rgb, 255:red, 208; green, 2; blue, 27 }  ,draw opacity=1 ][fill={rgb, 255:red, 255; green, 255; blue, 255 }  ,fill opacity=1 ] (521.9,57.43) -- (552,57.43) -- (539.1,74) -- (509,74) -- cycle ;
\draw  [color={rgb, 255:red, 208; green, 2; blue, 27 }  ,draw opacity=1 ] (550.7,126.29) -- (580.8,126.29) -- (567.9,142.86) -- (537.8,142.86) -- cycle ;
\draw  [color={rgb, 255:red, 74; green, 144; blue, 226 }  ,draw opacity=1 ] (288.7,200.86) .. controls (288.7,199.75) and (289.6,198.86) .. (290.7,198.86) .. controls (291.8,198.86) and (292.7,199.75) .. (292.7,200.86) .. controls (292.7,201.96) and (291.8,202.86) .. (290.7,202.86) .. controls (289.6,202.86) and (288.7,201.96) .. (288.7,200.86)(285.7,200.86) .. controls (285.7,198.1) and (287.94,195.86) .. (290.7,195.86) .. controls (293.46,195.86) and (295.7,198.1) .. (295.7,200.86) .. controls (295.7,203.62) and (293.46,205.86) .. (290.7,205.86) .. controls (287.94,205.86) and (285.7,203.62) .. (285.7,200.86) ;
\draw  [color={rgb, 255:red, 74; green, 144; blue, 226 }  ,draw opacity=1 ] (462.2,118) .. controls (462.2,116.9) and (463.1,116) .. (464.2,116) .. controls (465.3,116) and (466.2,116.9) .. (466.2,118) .. controls (466.2,119.1) and (465.3,120) .. (464.2,120) .. controls (463.1,120) and (462.2,119.1) .. (462.2,118)(459.2,118) .. controls (459.2,115.24) and (461.44,113) .. (464.2,113) .. controls (466.96,113) and (469.2,115.24) .. (469.2,118) .. controls (469.2,120.76) and (466.96,123) .. (464.2,123) .. controls (461.44,123) and (459.2,120.76) .. (459.2,118) ;
\draw  [color={rgb, 255:red, 74; green, 144; blue, 226 }  ,draw opacity=1 ] (189.6,82.29) .. controls (189.6,81.18) and (190.5,80.29) .. (191.6,80.29) .. controls (192.7,80.29) and (193.6,81.18) .. (193.6,82.29) .. controls (193.6,83.39) and (192.7,84.29) .. (191.6,84.29) .. controls (190.5,84.29) and (189.6,83.39) .. (189.6,82.29)(186.6,82.29) .. controls (186.6,79.52) and (188.84,77.29) .. (191.6,77.29) .. controls (194.36,77.29) and (196.6,79.52) .. (196.6,82.29) .. controls (196.6,85.05) and (194.36,87.29) .. (191.6,87.29) .. controls (188.84,87.29) and (186.6,85.05) .. (186.6,82.29) ;
\draw  [color={rgb, 255:red, 74; green, 144; blue, 226 }  ,draw opacity=1 ] (298.6,82.29) .. controls (298.6,81.18) and (299.5,80.29) .. (300.6,80.29) .. controls (301.7,80.29) and (302.6,81.18) .. (302.6,82.29) .. controls (302.6,83.39) and (301.7,84.29) .. (300.6,84.29) .. controls (299.5,84.29) and (298.6,83.39) .. (298.6,82.29)(295.6,82.29) .. controls (295.6,79.52) and (297.84,77.29) .. (300.6,77.29) .. controls (303.36,77.29) and (305.6,79.52) .. (305.6,82.29) .. controls (305.6,85.05) and (303.36,87.29) .. (300.6,87.29) .. controls (297.84,87.29) and (295.6,85.05) .. (295.6,82.29) ;
\draw  [color={rgb, 255:red, 74; green, 144; blue, 226 }  ,draw opacity=1 ] (420.5,65.71) .. controls (420.5,64.61) and (421.4,63.71) .. (422.5,63.71) .. controls (423.6,63.71) and (424.5,64.61) .. (424.5,65.71) .. controls (424.5,66.82) and (423.6,67.71) .. (422.5,67.71) .. controls (421.4,67.71) and (420.5,66.82) .. (420.5,65.71)(417.5,65.71) .. controls (417.5,62.95) and (419.74,60.71) .. (422.5,60.71) .. controls (425.26,60.71) and (427.5,62.95) .. (427.5,65.71) .. controls (427.5,68.48) and (425.26,70.71) .. (422.5,70.71) .. controls (419.74,70.71) and (417.5,68.48) .. (417.5,65.71) ;
\draw  [color={rgb, 255:red, 74; green, 144; blue, 226 }  ,draw opacity=1 ] (528.5,65.71) .. controls (528.5,64.61) and (529.4,63.71) .. (530.5,63.71) .. controls (531.6,63.71) and (532.5,64.61) .. (532.5,65.71) .. controls (532.5,66.82) and (531.6,67.71) .. (530.5,67.71) .. controls (529.4,67.71) and (528.5,66.82) .. (528.5,65.71)(525.5,65.71) .. controls (525.5,62.95) and (527.74,60.71) .. (530.5,60.71) .. controls (533.26,60.71) and (535.5,62.95) .. (535.5,65.71) .. controls (535.5,68.48) and (533.26,70.71) .. (530.5,70.71) .. controls (527.74,70.71) and (525.5,68.48) .. (525.5,65.71) ;
\draw  [color={rgb, 255:red, 74; green, 144; blue, 226 }  ,draw opacity=1 ] (557.3,134.57) .. controls (557.3,133.47) and (558.2,132.57) .. (559.3,132.57) .. controls (560.4,132.57) and (561.3,133.47) .. (561.3,134.57) .. controls (561.3,135.68) and (560.4,136.57) .. (559.3,136.57) .. controls (558.2,136.57) and (557.3,135.68) .. (557.3,134.57)(554.3,134.57) .. controls (554.3,131.81) and (556.54,129.57) .. (559.3,129.57) .. controls (562.06,129.57) and (564.3,131.81) .. (564.3,134.57) .. controls (564.3,137.33) and (562.06,139.57) .. (559.3,139.57) .. controls (556.54,139.57) and (554.3,137.33) .. (554.3,134.57) ;
\draw [color={rgb, 255:red, 74; green, 144; blue, 226 }  ,draw opacity=1 ]   (96,65) .. controls (152.5,58.29) and (152,83) .. (191.6,82.29) ;
\draw [color={rgb, 255:red, 74; green, 144; blue, 226 }  ,draw opacity=1 ]   (300.6,82.29) .. controls (353,83) and (352,65) .. (422.5,65.71) ;
\draw [color={rgb, 255:red, 74; green, 144; blue, 226 }  ,draw opacity=1 ]   (191.6,82.29) -- (300.6,82.29) ;
\draw [color={rgb, 255:red, 74; green, 144; blue, 226 }  ,draw opacity=1 ]   (124.3,134.57) .. controls (175,130) and (167.9,168.43) .. (207.5,167.71) ;
\draw [color={rgb, 255:red, 74; green, 144; blue, 226 }  ,draw opacity=1 ]   (207.5,167.71) .. controls (258.2,163.14) and (251.1,201.57) .. (290.7,200.86) ;
\draw [color={rgb, 255:red, 74; green, 144; blue, 226 }  ,draw opacity=1 ]   (290.7,200.86) .. controls (360,204) and (392,119) .. (464.2,118) ;
\draw [color={rgb, 255:red, 74; green, 144; blue, 226 }  ,draw opacity=1 ]   (422.5,65.71) -- (530.5,65.71) ;
\draw [color={rgb, 255:red, 74; green, 144; blue, 226 }  ,draw opacity=1 ]   (464.2,118) .. controls (516,116) and (512,138) .. (559.3,134.57) ;
\draw [color={rgb, 255:red, 74; green, 144; blue, 226 }  ,draw opacity=1 ] [dash pattern={on 4.5pt off 4.5pt}]  (530.5,65.71) -- (650,65) ;
\draw [color={rgb, 255:red, 74; green, 144; blue, 226 }  ,draw opacity=1 ] [dash pattern={on 4.5pt off 4.5pt}]  (559.3,134.57) -- (648.01,134.03) -- (651,134) ;
\draw [color={rgb, 255:red, 74; green, 144; blue, 226 }  ,draw opacity=1 ] [dash pattern={on 4.5pt off 4.5pt}]  (124.3,134.57) .. controls (173.73,130.11) and (180.48,150.78) .. (217.49,151.15) ;
\draw [shift={(220.4,151.14)}, rotate = 178.97] [fill={rgb, 255:red, 74; green, 144; blue, 226 }  ,fill opacity=1 ][line width=0.08]  [draw opacity=0] (10.72,-5.15) -- (0,0) -- (10.72,5.15) -- (7.12,0) -- cycle    ;
\draw [color={rgb, 255:red, 74; green, 144; blue, 226 }  ,draw opacity=1 ] [dash pattern={on 4.5pt off 4.5pt}]  (124.3,134.57) .. controls (163.3,138.47) and (156.38,169.4) .. (191.79,183.25) ;
\draw [shift={(194.6,184.29)}, rotate = 198.99] [fill={rgb, 255:red, 74; green, 144; blue, 226 }  ,fill opacity=1 ][line width=0.08]  [draw opacity=0] (10.72,-5.15) -- (0,0) -- (10.72,5.15) -- (7.12,0) -- cycle    ;
\draw [color={rgb, 255:red, 74; green, 144; blue, 226 }  ,draw opacity=1 ] [dash pattern={on 4.5pt off 4.5pt}]  (124.3,134.57) .. controls (163.5,138.49) and (140.95,185.08) .. (179.28,199.98) ;
\draw [shift={(181.7,200.86)}, rotate = 198.38] [fill={rgb, 255:red, 74; green, 144; blue, 226 }  ,fill opacity=1 ][line width=0.08]  [draw opacity=0] (10.72,-5.15) -- (0,0) -- (10.72,5.15) -- (7.12,0) -- cycle    ;
\draw [color={rgb, 255:red, 74; green, 144; blue, 226 }  ,draw opacity=1 ] [dash pattern={on 4.5pt off 4.5pt}]  (207.5,167.71) .. controls (257.24,169.97) and (277.39,147.68) .. (327.11,150.97) ;
\draw [shift={(329.4,151.14)}, rotate = 184.61] [fill={rgb, 255:red, 74; green, 144; blue, 226 }  ,fill opacity=1 ][line width=0.08]  [draw opacity=0] (10.72,-5.15) -- (0,0) -- (10.72,5.15) -- (7.12,0) -- cycle    ;
\draw [color={rgb, 255:red, 74; green, 144; blue, 226 }  ,draw opacity=1 ] [dash pattern={on 4.5pt off 4.5pt}]  (207.5,167.71) .. controls (257.24,169.97) and (264.88,163.76) .. (314.22,167.53) ;
\draw [shift={(316.5,167.71)}, rotate = 184.61] [fill={rgb, 255:red, 74; green, 144; blue, 226 }  ,fill opacity=1 ][line width=0.08]  [draw opacity=0] (10.72,-5.15) -- (0,0) -- (10.72,5.15) -- (7.12,0) -- cycle    ;
\draw [color={rgb, 255:red, 74; green, 144; blue, 226 }  ,draw opacity=1 ] [dash pattern={on 4.5pt off 4.5pt}]  (207.5,167.71) .. controls (256.99,169.95) and (253.18,186.33) .. (300.63,184.43) ;
\draw [shift={(303.6,184.29)}, rotate = 176.93] [fill={rgb, 255:red, 74; green, 144; blue, 226 }  ,fill opacity=1 ][line width=0.08]  [draw opacity=0] (10.72,-5.15) -- (0,0) -- (10.72,5.15) -- (7.12,0) -- cycle    ;
\draw [color={rgb, 255:red, 74; green, 144; blue, 226 }  ,draw opacity=1 ] [dash pattern={on 4.5pt off 4.5pt}]  (290.7,200.86) .. controls (359.95,201.98) and (390.09,134.08) .. (448.61,134.5) ;
\draw [shift={(451.3,134.57)}, rotate = 182.44] [fill={rgb, 255:red, 74; green, 144; blue, 226 }  ,fill opacity=1 ][line width=0.08]  [draw opacity=0] (10.72,-5.15) -- (0,0) -- (10.72,5.15) -- (7.12,0) -- cycle    ;
\draw [color={rgb, 255:red, 74; green, 144; blue, 226 }  ,draw opacity=1 ] [dash pattern={on 4.5pt off 4.5pt}]  (464.2,118) .. controls (533.45,119.13) and (512.56,115.54) .. (569.54,117.89) ;
\draw [shift={(572.2,118)}, rotate = 182.44] [fill={rgb, 255:red, 74; green, 144; blue, 226 }  ,fill opacity=1 ][line width=0.08]  [draw opacity=0] (10.72,-5.15) -- (0,0) -- (10.72,5.15) -- (7.12,0) -- cycle    ;

\draw (70,57) node [anchor=north west][inner sep=0.75pt]  [color={rgb, 255:red, 0; green, 0; blue, 0 }  ,opacity=1 ] [align=left] {{\scriptsize $0$}};
\draw (57,74) node [anchor=north west][inner sep=0.75pt]  [color={rgb, 255:red, 0; green, 0; blue, 0 }  ,opacity=1 ] [align=left] {{\scriptsize $1$}};
\draw (112,110) node [anchor=north west][inner sep=0.75pt]  [color={rgb, 255:red, 0; green, 0; blue, 0 }  ,opacity=1 ] [align=left] {{\scriptsize $0$}};
\draw (99,127) node [anchor=north west][inner sep=0.75pt]  [color={rgb, 255:red, 0; green, 0; blue, 0 }  ,opacity=1 ] [align=left] {{\scriptsize $1$}};
\draw (85,144) node [anchor=north west][inner sep=0.75pt]  [color={rgb, 255:red, 0; green, 0; blue, 0 }  ,opacity=1 ] [align=left] {{\scriptsize $2$}};
\draw (73,160) node [anchor=north west][inner sep=0.75pt]  [color={rgb, 255:red, 0; green, 0; blue, 0 }  ,opacity=1 ] [align=left] {{\scriptsize $3$}};
\draw (60,175) node [anchor=north west][inner sep=0.75pt]  [color={rgb, 255:red, 0; green, 0; blue, 0 }  ,opacity=1 ] [align=left] {{\scriptsize $4$}};
\draw (49,191) node [anchor=north west][inner sep=0.75pt]  [color={rgb, 255:red, 0; green, 0; blue, 0 }  ,opacity=1 ] [align=left] {{\scriptsize $5$}};

\draw (82,32) node [anchor=north west][inner sep=0.75pt]  [color={rgb, 255:red, 0; green, 0; blue, 0 }  ,opacity=1 ] [align=left] {$X_1=0$};
\draw (189,32) node [anchor=north west][inner sep=0.75pt]  [color={rgb, 255:red, 0; green, 0; blue, 0 }  ,opacity=1 ] [align=left] {$X_2=1$};
\draw (301,32) node [anchor=north west][inner sep=0.75pt]  [color={rgb, 255:red, 0; green, 0; blue, 0 }  ,opacity=1 ] [align=left] {$X_3=1$};
\draw (409,32) node [anchor=north west][inner sep=0.75pt]  [color={rgb, 255:red, 0; green, 0; blue, 0 }  ,opacity=1 ] [align=left] {$X_4=0$};
\draw (517,32) node [anchor=north west][inner sep=0.75pt]  [color={rgb, 255:red, 0; green, 0; blue, 0 }  ,opacity=1 ] [align=left] {$X_5=0$};
\draw (36,220) node [anchor=north west][inner sep=0.75pt]  [color={rgb, 255:red, 0; green, 0; blue, 0 }  ,opacity=1 ] [align=left] {$(\Sigma_\Lambda X)_1 = 1$};
\draw (141,220) node [anchor=north west][inner sep=0.75pt]  [color={rgb, 255:red, 0; green, 0; blue, 0 }  ,opacity=1 ] [align=left] {$(\Sigma_\Lambda X)_2 = 3$};
\draw (254,220) node [anchor=north west][inner sep=0.75pt]  [color={rgb, 255:red, 0; green, 0; blue, 0 }  ,opacity=1 ] [align=left] {$(\Sigma_\Lambda X)_3 = 5$};
\draw (362,220) node [anchor=north west][inner sep=0.75pt]  [color={rgb, 255:red, 0; green, 0; blue, 0 }  ,opacity=1 ] [align=left] {$(\Sigma_\Lambda X)_4 = 0$};
\draw (475,220) node [anchor=north west][inner sep=0.75pt]  [color={rgb, 255:red, 0; green, 0; blue, 0 }  ,opacity=1 ] [align=left] {$(\Sigma_\Lambda X)_5 = 1$};
\draw (32,56) node [anchor=north west][inner sep=0.75pt]   [align=left] {$X$};
\draw (26,159) node [anchor=north west][inner sep=0.75pt]   [align=left] {$\Sigma_\Lambda X$};

\end{tikzpicture}}
    \caption{\hl{Simulating a trajectory from $\Lambda_\star P \colon \calW(\set{0,1}) \to \calW(\set{0,1,2,3,4,5})$}.}
    \label{figure:trajectory-simulation}
\end{figure}
We can verify that one can simulate \hlc{(see Figure~\ref{figure:trajectory-simulation})} a trajectory $Y_1, \dots, Y_t, \dots$ sampled according to the embedded \hl{transition matrix} $\Lambda_\star P$ and with initial distribution $\nu$ as follows:
\begin{enumerate}
    \item[$(i)$] Sample $Y_1 \sim \nu_{| X_1}$.
    \item[$(ii)$] For $t \in \N, t \geq 2$, sample $Y_t \sim (\Lambda(Y_{t-1}, y'))_{y' \in \calS_{X_t}}$.
\end{enumerate}
Markov embeddings of chains can therefore essentially be simulated from a single trajectory of the original chain, similar to lumpings,\label{operator:randomized-simulation} \hl{and we write $\Sigma_\Lambda \colon \calX^\infty\to \calY^\infty$ for the randomized operation on trajectories induced from $\Lambda$}.
In certain cases it is possible to obtain an expression for the stationary distribution $\Lambda_\star \pi$ of the embedded chain (see e.g. Lemma~\ref{lemma:memoryless-embedding-stationary-distribution} or Lemma~\ref{lemma:reversibility-preserving-embedding-stationary-distribution}). Setting $\nu = \Lambda_\star \pi$ then starts the embedded chain stationarily.
In the subsequent lemma, we show that the Fisher metric, dual connections and information divergence are preserved under Markov embeddings.

\begin{lemma}
\label{lemma:markov-embedding-preserve-information-geometry}
Let $\calV = \set{\bar{P}_\theta \colon \theta \in \Theta \subset \R^d} \subset \calW(\calX, \calD)$ be a parametric family of irreducible \hl{stochastic matrices}.
Let $\bar{P}_\theta , \bar{P}_{\theta'} \in \calV$. Let a Markov embedding 
$$\Lambda_\star \colon \calW(\calX, \calD) \to \calW(\calY, \calE),$$ 
and consider the embedded \hl{stochastic matrices} $P_\theta \eqdef \Lambda_\star \bar{P}_\theta$ and $ P_{\theta'} \eqdef \Lambda_\star \bar{P}_{\theta'}$. For any $i,j \in [d]$, it holds that
\begin{equation*}
\begin{split}
    \fshr_{ij}(P_\theta) &= \fshr_{ij}(\bar{P}_\theta), \qquad 
    \kl{P_\theta}{P_{\theta'}} = \kl{\bar{P}_\theta}{\bar{P}_{\theta'}}. \\
    \Gamma_{ij,k}^{(e)}(P_\theta) &= \Gamma_{ij,k}^{(e)}(\bar{P}_\theta), \qquad
    \Gamma_{ij,k}^{(m)}(P_\theta) = \Gamma_{ij,k}^{(m)}(\bar{P}_\theta). \\
\end{split}
\end{equation*}
\end{lemma}
\begin{proof}
See Section~\ref{proof:lemma-markov-embedding-preserve-information-geometry}.
\end{proof}

We conclude this section by showing that we can always view a lumpable matrix as the image of its lumped version
by some canonical Markov embedding.

\begin{lemma}[Canonical embedding]
\label{lemma:construct-canonical-embedding}
Let $P \in \calW_\kappa(\calY, \calE)$. There exists a $\kappa$-compatible Markov embedding, denoted by $\Lambda^{(P)}_\star$
such that 
$$P = \Lambda_\star^{(P)} \kappa_\star P.$$
\hlp{We call $\Lambda_\star^{(P)}$ the canonical embedding associated with $P$.
Moreover, $\Lambda^{(P)}_\star$ is induced from $\Lambda^{(P)} \in \calF_+(\calY, \calE)$ given by
$$\Lambda^{(P)}(y , y') = \frac{P(y,y')}{\kappa_\star P(\kappa(y), \kappa(y'))}, \text{ for any } (y,y') \in \calE.$$
}

\end{lemma}
\begin{proof}
A straightforward computation yields that $P = \Lambda_\star^{(P)} \kappa_\star P$.
\end{proof}

\subsection{Congruent embeddings}
\label{section:embeddings-congruent}
Let us fix a lumping function $\kappa \colon \calY \to \calX$, and compatible edge sets $\calE, \calD = \kappa_2(\calE)$ such that $\calW_\kappa(\calY, \calE) \neq \emptyset$ (see Proposition~\ref{proposition:lumpable-not-empty}).
Consider an arbitrary embedding $\hlm{E} \colon \calW(\calX, \calD) \to \calW(\calY, \calE)$, that does not necessarily follow the prescribed structure of Definition~\ref{definition:markov-embeddings}.
When composing the embedding with its associated lumping always yields back the original chain (\hl{$\kappa_\star E P = P$}), \hl{$E$} is a right-inverse of $\kappa_\star$.
Adapting the terminology of \citet{campbell1986extended,cencov1981statistical}, we say in this case that \hl{$E$} is a $\kappa$-congruent embedding (Definition~\ref{definition:congruent-embedding}).
In a finite space distribution setting, Markov morphisms and congruent embeddings are known to coincide (see e.g. Example~5.2 in \citet{ay2017}).
The proof strategy for this claim consists in expanding the notion of Markov morphisms to signed measures, and proving the claim for morphisms seen as linear operators over a real vector space.
We will show that a similar fact holds in the Markovian setting and for our definition of Markov morphisms (Definition~\ref{definition:markov-embeddings}). We begin by extending the definition of lumpable \hl{stochastic matrices} to the more general class of lumpable matrices.

\begin{definition}[$\kappa$-lumpable matrix]
\label{definition:lumpable-matrix}
Let $\kappa \colon \calY \to \calX$ a lumping function with associated partition $\biguplus_{x \in \calX} \calS_{x} = \calY$, and let $A \in \calF(\calY, \calE)$. Then $A$ is \hl{called} $\kappa$-lumpable \hl{when}
for all $x,x' \in \calX$, and for all $y_1,y_2 \in \calS_x$,
$$A(y_1, \calS_{x'}) = A(y_2, \calS_{x'}).$$
In this case, the lumped matrix $\kappa_\star A$ is \hl{defined as}
$$\kappa_\star A(x,x') = A(y,\calS_{x'}), y \in \calS_{x}.$$
We write $\calF_\kappa(\calY, \calE) \subset \calF(\calY, \calE)$ for the subset of all $\kappa$-lumpable matrices.
\end{definition}

Recall that $\calF(\calY, \calE)$ can be endowed with a real vector space structure of dimension $\dim \calF(\calY , \calE) = \abs{\calE}$. Our next step consists in viewing $\calF_{\kappa}(\calY, \calE)$ as a linear subspace of $\calF(\calY, \calE)$.

\begin{lemma}
\label{lemma:lumpable-matrix-vector-space}
 The following statements hold.
\begin{enumerate}
    \item[$(i)$] $\calF_\kappa(\calY, \calE)$ forms a vector subspace of $\calF(\calY, \calE)$.
    \item[$(ii)$] $\kappa_\star \colon \calF_\kappa(\calY, \calE) \to \calF(\calX = \kappa(\calY), \calD = \kappa_2(\calE))$ is a surjective linear map.
    \item[$(iii)$] For $\biguplus_{x \in \calX} \calS_{x} = \calY$ the partition associated with $\kappa$,
    $$\dim \calF_\kappa(\calY, \calE) = \abs{\calE} - \sum_{(x,x') \in \calD} \abs{\calS_x} + \abs{\calD}.$$
\end{enumerate}
\end{lemma}

\begin{proof}
See Section~\ref{proof:lemma-lumpable-matrix-vector-space}.
\end{proof}

\begin{remark}
Perhaps surprisingly, in the fully connected graph case $\calE = \calY^2$, 
$$\dim \calF_\kappa(\calY) = \abs{\calY}^2 - \abs{\calX}\abs{\calY} + \abs{\calX}^2,$$
which is oblivious to the exact partition defined by $\kappa$, and only depends on the alphabet sizes of its domain and range.
\end{remark}

We can expand the domain of Markov embeddings in Definition~\ref{definition:markov-embeddings} to subsets of $\calF(\calX, \calD)$, and verify that embedded matrices are lumpable (Definition~\ref{definition:lumpable-matrix}). It is noteworthy that Proposition~\ref{proposition:lumpable-not-empty} seamlessly extends to $\calF_\kappa(\calY, \calE)$. 
Inspired by the definition of congruent mappings in the context of distributions \citep[Definition~5.1]{ay2017} and statistics in the sense of \citep[Section~5.1.1]{ay2017}, we introduce embeddings of matrices that are congruent for lumpings.

\begin{definition}[$\kappa$-congruent embedding]
\label{definition:congruent-embedding}
Let a mapping
$$K_\star \colon \calF(\calX = \kappa(\calY), \calD = \kappa_2(\calE)) \to \calF_\kappa(\calY, \calE).$$ 
We say that $K_\star$ is a $\kappa$-congruent embedding\hl{\footnote{\hlc{Our notation $K_\star$ and terminology is based on \citet{ay2017}.}}} when
\begin{enumerate}
    \item[$(i)$] $K_\star$ is a linear map.
    \item[$(ii)$] $K_\star$ is monotonic\hl{\footnote{\hl{Also called non-negative.}}} in the sense that non-negative matrices are mapped to non-negative matrices, i.e. for any $A \in \calF(\calX, \calD)$,
    $$A \geq 0 \implies K_\star A \geq 0,$$
    \item[$(iii)$] $K_\star$ preserves irreducibility \footnote{Irreducibility preservation can also be interpreted as a form of monotonicity.},
    $$A \in \calF_+(\calX, \calD) \implies K_\star A \in \calF_+(\calY, \calE).$$
    \item[$(iv)$] $K_\star$ is a right inverse of $\kappa_\star$, i.e. for any $\forall A \in \calF(\calX, \calD)$,
$$\kappa_\star K_\star A = A.$$
\end{enumerate}
\end{definition}

The surjectivity of $\kappa_\star$  together with $\Ker \kappa_\star \neq \set{0}$ (Lemma~\ref{lemma:lumpable-matrix-vector-space}) guarantee the existence of multiple right inverses to $\kappa_\star$, i.e. potential candidates for congruent embeddings.
We will now show that $\kappa$-congruent embeddings are exactly the Markov embeddings whose partition of states coincides with the one defined by $\kappa$.
\begin{theorem}
\label{theorem:congruent-embeddings-are-lambda-embeddings}
Let $(\calX, \calD)$, $(\calY, \calE)$ be strongly connected digraphs, and $\kappa \colon \calY \to \calX$ a lumping function, such that $\kappa_2(\calE) = \calD$, and $\calF_\kappa(\calY, \calE) \neq \emptyset$. Then
$K_\star \colon \calF(\calX, \calD) \to \calF_\kappa(\calY, \calE)$ is a $\kappa$-congruent embedding (Definition~\ref{definition:congruent-embedding}), if and only if $K_\star$ is a $\kappa$-compatible Markov embedding (Definition~\ref{definition:markov-embeddings}).
\end{theorem}
\begin{proof}
See Section~\ref{proof:theorem-congruent-embeddings-are-lambda-embeddings}.
\end{proof}

\begin{corollary}
The $\kappa$-congruent embeddings from $\calV \subset \calW(\calX, \calD)$ are exactly the $\kappa$-compatible Markov embeddings from $\calV$.
\label{term:Embkc}
\end{corollary}

\begin{remark}
\label{remark:dof-markov-embedding}
Notice that a $\kappa$-congruent embedding enjoys
\begin{equation*}
    \abs{\calE} - \sum_{(x,x') \in \calD} \abs{\calS_x}
\end{equation*}
degrees of freedom.
In particular, when $\calE = \calY^2$, a $\kappa$-congruent embedding has $\abs{\calY}(\abs{\calY} - \abs{\calX})$ degrees of freedom. In the context of distributions, a Markov morphism from $\calP_+(\calX)$ to $\calP_+(\calY)$ would have $\abs{\calY} - \abs{\calX}$ degrees of freedom.
\end{remark}

\subsection{Exponential embeddings}
\label{section:embeddings-exponential}
We fix a lumping function $\kappa \colon \calY \to \calX$, compatible edge sets $\calE, \calD = \kappa_2(\calE)$, such that $\calW_\kappa(\calY, \calE)  \neq \emptyset$, and consider
\begin{equation*}
    \kappa_\star \colon \calW_\kappa(\calY, \calE) \to \calW(\calX, \calD).
\end{equation*}
We now introduce another class of embeddings, which we call exponential embeddings. We show that this class preserves certain geometric features of families of \hl{stochastic matrices}, and strictly encompasses the previously defined Markov embeddings.

\begin{definition}[Exponential embedding]
\label{definition:exponential-embedding}
Let $P_{\origin} \in \calW_\kappa(\calY, \calE)$, and write $\bar{P}_{\origin} \eqdef \kappa_\star P_{\origin} \in \calW(\calX, \calD)$.
For a given $\bar{P} \in \calW(\calX, \calD)$,
we let $\tilde{P} \in \calF_+(\calY, \calE)$ be such that for any $y,y' \in \calY$,
\begin{equation*}
\begin{split}
    \tilde{P}(y,y') &\eqdef P_{\origin}(y,y')\bar{P}(\kappa(y), \kappa(y')).
\end{split}
\end{equation*}
The mapping
\begin{equation*}
    \begin{split}
        \hlm{\Phi} \colon \calW(\calX, \calD) &\to \calW(\calY, \calE) , \qquad \bar{P} \mapsto \stoch(\tilde{P})
    \end{split}
\end{equation*}
is called the $\kappa$-compatible exponential embedding with origin $P_{\origin}$.
\label{term:EEmbk}
\label{term:EEmbkc}
\label{term:EEmb}
\end{definition}

\begin{proposition}
\label{proposition:explicit-phi-embedding}
Let $\bar{P} \in \calW(\calX, \calD)$, and consider the exponential embedding with $\kappa$-lumpable origin $P_{\origin}$,
\begin{enumerate}
    \item[$(i)$] For any $y,y' \in \calY$, 
    $$\hlm{\Phi} \bar{P}(y,y') = \tilde{P}(y,y') \frac{v(\kappa(y'))}{\rho v(\kappa(y))},$$ where $(\rho, v)$ is the right PF pair of $\bar{P}_{\origin} \circ \bar{P}$, and $\tilde{P}$ is as in Definition~\ref{definition:exponential-embedding}.
    \item[$(ii)$] $\hlm{\Phi} \bar{P}$ is $\kappa$-lumpable [i.e. $\Range \hlm{\Phi} \subset \calW_\kappa(\calY, \calE)$], with $\kappa_\star \hlm{\Phi} \bar{P} = \stoch(\bar{P}_{\origin} \circ \bar{P})$.
\end{enumerate}

\end{proposition}

\begin{proof}
See Section~\ref{proof:proposition-explicit-phi-embedding}.
\end{proof}

\begin{figure}%
\centering

\tikzset{every picture/.style={line width=0.75pt}} %

\begin{tikzpicture}[x=0.70pt,y=0.70pt,yscale=-1,xscale=1]

\draw   (101,94.59) .. controls (101,69.96) and (134.8,50) .. (176.5,50) .. controls (218.2,50) and (252,69.96) .. (252,94.59) .. controls (252,119.22) and (218.2,139.18) .. (176.5,139.18) .. controls (134.8,139.18) and (101,119.22) .. (101,94.59) -- cycle ;
\draw   (333,107.79) .. controls (333,72.56) and (381.35,44) .. (441,44) .. controls (500.65,44) and (549,72.56) .. (549,107.79) .. controls (549,143.01) and (500.65,171.57) .. (441,171.57) .. controls (381.35,171.57) and (333,143.01) .. (333,107.79) -- cycle ;
\draw  [color={rgb, 255:red, 208; green, 2; blue, 27 }  ,draw opacity=1 ] (195.28,61.72) -- (202.72,69.17)(202.72,61.72) -- (195.28,69.17) ;
\draw   (204.28,94.72) -- (211.72,102.17)(211.72,94.72) -- (204.28,102.17) ;
\draw   (435.28,78.72) -- (442.72,86.17)(442.72,78.72) -- (435.28,86.17) ;
\draw    (435,82) .. controls (364.35,51.15) and (334.3,49.02) .. (205.94,64.76) ;
\draw [shift={(204,65)}, rotate = 352.98] [fill={rgb, 255:red, 0; green, 0; blue, 0 }  ][line width=0.08]  [draw opacity=0] (10.72,-5.15) -- (0,0) -- (10.72,5.15) -- (7.12,0) -- cycle    ;
\draw   (131.28,99.72) -- (138.72,107.17)(138.72,99.72) -- (131.28,107.17) ;
\draw   (416.28,111.72) -- (423.72,119.17)(423.72,111.72) -- (416.28,119.17) ;
\draw    (211,99) .. controls (268.42,122.76) and (341.52,123) .. (410.9,115.24) ;
\draw [shift={(413,115)}, rotate = 173.48] [fill={rgb, 255:red, 0; green, 0; blue, 0 }  ][line width=0.08]  [draw opacity=0] (10.72,-5.15) -- (0,0) -- (10.72,5.15) -- (7.12,0) -- cycle    ;
\draw [color={rgb, 255:red, 74; green, 144; blue, 226 }  ,draw opacity=1 ]   (133,131) .. controls (136,97) and (127,66) .. (199,52) ;
\draw    (413,117) .. controls (327.86,131.85) and (219.2,137.88) .. (145.23,106.95) ;
\draw [shift={(143,106)}, rotate = 23.39] [fill={rgb, 255:red, 0; green, 0; blue, 0 }  ][line width=0.08]  [draw opacity=0] (10.72,-5.15) -- (0,0) -- (10.72,5.15) -- (7.12,0) -- cycle    ;
\draw [color={rgb, 255:red, 74; green, 144; blue, 226 }  ,draw opacity=1 ]   (191,138) .. controls (179,109) and (203,104) .. (243,74) ;
\draw [color={rgb, 255:red, 74; green, 144; blue, 226 }  ,draw opacity=1 ]   (333,107.79) .. controls (402,105) and (431,100) .. (466,170) ;
\draw (175,60) node [anchor=north west][inner sep=0.75pt]  [color={rgb, 255:red, 208; green, 2; blue, 27 }  ,opacity=1 ] [align=left] {$\bar{P}_{\origin}$};
\draw (445,66) node [anchor=north west][inner sep=0.75pt]   [align=left] {$P_{\origin}$};
\draw (266,45) node [anchor=north west][inner sep=0.75pt]   [align=left] {$\kappa_\star$};
\draw (50,51) node [anchor=north west][inner sep=0.75pt]   [align=left] {$\calW(\calX, \calD)$};
\draw (511,45) node [anchor=north west][inner sep=0.75pt]   [align=left] {$\calW_\kappa(\calY, \calE)$};
\draw (196,80) node [anchor=north west][inner sep=0.75pt]   [align=left] {$\bar{P}$};
\draw (264,132) node [anchor=north west][inner sep=0.75pt]   [align=left] {$\kappa_\star$};
\draw (107,80) node [anchor=north west][inner sep=0.75pt]   [align=left] {$\stoch(\bar{P}_{\origin} \circ \bar{P})$};
\draw (277,104) node [anchor=north west][inner sep=0.75pt]   [align=left] {\hl{$\Phi$}};
\draw (193,141) node [anchor=north west][inner sep=0.75pt]  [color={rgb, 255:red, 74; green, 144; blue, 226 }  ,opacity=1 ] [align=left] {$\calV$};
\draw (451,173) node [anchor=north west][inner sep=0.75pt]  [color={rgb, 255:red, 74; green, 144; blue, 226 }  ,opacity=1 ] [align=left] {$\hlm{\Phi} \calV$};
\draw (105,135) node [anchor=north west][inner sep=0.75pt]  [color={rgb, 255:red, 74; green, 144; blue, 226 }  ,opacity=1 ] [align=left] {$\kappa_\star \hlm{\Phi}  \calV$};

\end{tikzpicture}

\caption{Exponential embedding of a family $\calV$, with origin $\bar{P}_{\origin}$. When $\calV$ forms an e-family, so does $\hlm{\Phi} \calV$.}
\end{figure}

\begin{remark}
While $\kappa$-lumpability of the origin $P_{\origin}$ ensures $\kappa$-lumpability of all exponentially embedded chains, note that composing the embedding with $\kappa$-lumping
does not generally recover the original chain $\bar{P}$, but rather some translated version $\stoch(\bar{P}_{\origin} \circ \bar{P})$. This leads to non-congruency of the embedding, except for some well-chosen origin (Theorem~\ref{theorem:congruent-exponential-embeddings}).
\end{remark}

For $P_0, P_1 \in \calW(\calY, \calE)$,
the e-geodesic passing through $P_0$ and $P_1$ is defined \citep[Corollary~2]{nagaoka2017exponential} by
$$\gamma^{(e)}_{P_0, P_1} \colon \R \to \calW(\calY, \calE),  t \mapsto \stoch\left(P_0^{\circ 1 - t} \circ P_1^{\circ t} \right).$$
Essentially, $\gamma^{(e)}_{P_0, P_1}$ is the straight line in $\calW(\calY, \calE)$ with respect to the e-connection, that goes through $P_0$ and $P_1$, and forms the simplest kind of ($1$-dimensional) e-family.
Similarly, the m-geodesic that goes through $P_0, P_1$ with respective edge measures $Q_0, Q_1$ is defined by 
$$\gamma^{(m)}_{P_0, P_1} \colon I \to \calW(\calY, \calE), t \mapsto P_t \colon Q_t = (1-t)Q_0 +  t Q_1,$$
where \hl{$I = \set{t \in \R \colon Q_t \text{ is an edge measure positive over $\calE$}}$ and where $P_t$} is the unique stochastic matrix that pertains to $Q_t$. In particular, note that $I \supset [0,1]$ and may extend beyond its endpoints.
A compelling property for an embedding $\hlm{E} \colon \calW(\calX, \calD) \to \calW(\calY, \calE)$,
is to preserve  the geometric structure in the sense of mapping an e-family to an e-family or an m-family to an m-family. This quality reduces to that of being a geodesically affine map.

\begin{definition}[Geodesically affine map]
\label{definition:geodesic-affine-map}
Let $\hlm{E} \colon \calW(\calX, \calD) \to \calW(\calY, \calE)$ be an embedding. When for all $P_0, P_1 \in \calW(\calX, \calD)$,
\begin{equation*}
    \begin{split}
        \hlm{E} \gamma^{(e)}_{P_0, P_1} = \gamma^{(e)}_{\hlm{E} P_0, \hlm{E} P_1},
    \end{split}
\end{equation*}
then $\hlm{E}$ is said to be e-geodesic affine. When for all $P_0, P_1 \in \calW(\calX, \calD)$,
\begin{equation*}
    \begin{split}
        \hlm{E} \gamma^{(m)}_{P_0, P_1} = \gamma^{(m)}_{\hlm{E} P_0, \hlm{E} P_1},
    \end{split}
\end{equation*}
then $\hlm{E}$ is said to be m-geodesic affine.
\label{term:eEmb} \label{term:mEmb}
\end{definition}

\begin{theorem}
\label{theorem:exponential-embedding-geodesic-affinity}
Let \hl{$\Phi$} be the exponential embedding with origin $P_{\origin}$ .
\begin{enumerate}
    \item[$(i)$] \hl{$\Phi$} is an e-geodesic affine map.
    \item[$(ii)$] \hl{$\Phi$} is generally not an m-geodesic affine map.
\end{enumerate}
\end{theorem}
\label{theorem:exponential-embeddings-are-e-geodesic-affine-maps}
\begin{proof}
To prove $(i)$, we rely on the fact that the mapping $\stoch$ induces an equivalence class for diagonally similar matrices ---see \eqref{eq:stoch-properties}.
For any $\bar{P}_0, \bar{P}_1 \in \calW(\calX, \calD)$ and 
$t \in \R$,
\begin{equation*}
    \begin{split}
        \hlm{\Phi}\gamma^{(e)}_{\bar{P}_0, \bar{P}_1}(t) &= \stoch\left(P_{\origin} \circ \left[\gamma^{(e)}_{\bar{P}_0, \bar{P}_1}(t) (\kappa(y),\kappa(y'))\right]_{y,y' \in \calY}\right) \\
        &= \stoch\left(P_{\origin} \circ \left[\bar{P}_0 (\kappa(y),\kappa(y'))^{ 1 - t} \bar{P}_1 (\kappa(y),\kappa(y'))^{t}\right]_{y,y' \in \calY}\right) \\
        &= \stoch\left(\left(P_{\origin} \circ \left[\bar{P}_0(\kappa(y), \kappa(y')) \right]_{y,y' \in \calY}\right)^{\circ 1 - t} \circ \left(P_{\origin} \circ \left[\bar{P}_1(\kappa(y), \kappa(y')) \right]_{y,y' \in \calY}\right)^{\circ  t}\right) \\
         &= \stoch\left((\hlm{\Phi} \bar{P}_0)^{\circ 1 - t} \circ (\hlm{\Phi} \bar{P}_1)^{\circ t}\right) \\
        &= \gamma^{(e)}_{\hlm{\Phi} \bar{P}_0, \hlm{\Phi} \bar{P}_1}(t). \\
    \end{split}
\end{equation*}
Statement $(ii)$ stems from the later Lemma~\ref{lemma:hudson-breaks-m-structure}.
\end{proof}

\begin{remark}
The example in Lemma~\ref{lemma:hudson-breaks-m-structure} actually shows the stronger, and somewhat surprising statement that even Markov embeddings are not generally m-geodesic affine.
This is in stark contrast with Markov morphisms in the context of distributions, which can be shown to be geodesically affine for both the m-connection and e-connection.
We will later construct (Section~\ref{section:embeddings-memoryless}) a non-trivial subset of Markov embeddings that also preserves the m-structure.
\end{remark}

\begin{remark}
In addition, it is not difficult to see that
extending the invariance
Lemma~\ref{lemma:markov-embedding-preserve-information-geometry} to all exponential embeddings is not possible, as the latter distort the Fisher metric and affine connections. In particular, for $P,P'$ embeddings of $\bar{P}, \bar{P}'$, it holds that
\begin{equation*}
    \kl{P}{P'} \geq \kl{\kappa_\star P}{\kappa_\star P'} =  \kl{\stoch(\bar{P}_{\origin} \circ \bar{P})}{\stoch(\bar{P}_{\origin} \circ \bar{P}')}.
\end{equation*}
\end{remark}
Let us introduce the special element $\bar{\maxentmc} \in \calW(\calX, \calD)$,
\begin{equation}
\label{eq:null-origin}
    \bar{\maxentmc} = \stoch(\delta_{\calD}), \text{ with } \delta_{\calD}(x,x') = \pred{(x,x') \in \calD}.
\end{equation}
Observe that when $\calD = \calX^2$, $\bar{\maxentmc} = \frac{1}{\abs{\calX}} 1 \trn 1$ is the \hl{stochastic matrix} that induces a uniform iid process over $\calX$. We will later recall in Section~\ref{section:max-entropy} that $\bar{\maxentmc}$ corresponds to the maximum entropy rate \hl{stochastic matrix} defined in $\calW(\calX, \calD)$.

\begin{theorem}
\label{theorem:congruent-exponential-embeddings}
The $\kappa$-compatible exponential embeddings whose origin lumps into $\bar{\maxentmc}$, defined at \eqref{eq:null-origin}, are exactly the $\kappa$-congruent embeddings.
\end{theorem}
\begin{proof}
We will show that
the $\kappa$-compatible exponential embeddings whose origin lumps into $\bar{\maxentmc}$ are the $\kappa$-compatible Markov embeddings,
and invoke Theorem~\ref{theorem:congruent-embeddings-are-lambda-embeddings}. 
Let 
$$\hlm{\Phi} \colon \calW(\calX, \calD) \to \calW(\calY, \calE)$$ 
be the exponential embedding with origin $P_{\origin} \in \calW_\kappa(\calY, \calE)$ such that $\kappa_\star P_{\origin} = \bar{\maxentmc}$.
For $\bar{P} \in \calW(\calX, \calD)$ and $y,y' \in \calY$,
\begin{equation*}
    \hlm{\Phi} \bar{P}(y,y') = \rho^{-1} v(\kappa(y))^{-1} P_{\origin}(y,y') \bar{P}(\kappa(y), \kappa(y')) v(\kappa(y')),
\end{equation*}
where $(\rho, v)$ is the right PF pair of $\bar{\maxentmc} \circ \bar{P}$. 
By definition of $\bar{\maxentmc}$, $\rho = \rho_{\maxentmc}^{-1} $ and $v = v_{\maxentmc}^{-1}$, where $(\rho_{\maxentmc}, v_{\maxentmc})$ is the right PF pair of the matrix $\delta_{\calD}$ defined at \eqref{eq:null-origin} thus
\begin{equation*}
\label{eq:exponential-as-canonical}
    \hlm{\Phi} \bar{P}(y,y') = \rho_{\maxentmc} P_{\origin}(y,y') \bar{P}(\kappa(y), \kappa(y')) \frac{v_{\maxentmc}(\kappa(y))}{v_{\maxentmc}(\kappa(y'))} =\bar{P}(\kappa(y), \kappa(y')) \frac{P_{\origin}(y,y')}{\kappa_\star P_{\origin}(\kappa(y),\kappa(y')) }.
\end{equation*}
This corresponds to the canonical Markov embedding (Lemma~\ref{lemma:construct-canonical-embedding}) constructed from $P_{\origin}$.
Conversely, for any $\kappa$-compatible Markov embedding $\Lambda_\star$, by setting $P_{\origin} = \Lambda_\star \bar{\maxentmc}$, we can create the exponential embedding \hl{$\Phi$}.
\end{proof}

\subsection{Examples and notable sub-classes of embeddings}
\label{section:embeddings-notable-examples}

\hl{In this section, we will demonstrate that existing notions of embeddings align with our framework, and construct sub-classes of Markov embeddings which preserve additional structure.}

\subsubsection{Hudson expansions}
\label{section:embeddings-hudson}
In this section, we discuss a particular expansion of a Markov process that appears in \citet[Section~6.5,p.140]{kemeny1983finite}, which they consider to be the natural inverse of lumping.
The first analysis of this expansion being credited to S. Hudson, we henceforth refer to it as the Hudson \hl{expansion}, and denote it by $H_\star$.
We invite the reader to consult  \citep[Example~6.5.1]{kemeny1983finite} for an illustrative example of this expansion. 

\paragraph{Hudson expansions as a Markov embedding.}
\label{section:hudson-embedding-as-markov-embedding}

Our first order of business is to show that the Hudson \hl{expansion} is a very particular Markov embedding where the target space is nothing but the set of edges in the directed graph defined by the original chain. 
Namely,
$$H_\star \colon \calW(\calX, \calD) \to \calW_{h}(\calD, \calE),$$
for some $\calE \subset \calD^2$, \hl{and where the Hudson lumping function $h$ outputs the destination vertex of an edge,} 
$$h \colon \calD \to \calX, (x_1, x_2) \mapsto x_2.$$
For $x \in \calX$, we let $\calS_x = \set{ e = (x_1, x_2) \in \calD \colon x_2 = x }$. Then  $\biguplus_{x \in \calX} \calS_x = \calD$ is the partition associated with $h$.
We further define
\begin{equation*}
    H_\calD \eqdef \set{ (e = (x_1, x_2),e' = (x_1', x_2')) \in \calD^2 \colon x_2 = x_1' },
\end{equation*} 
and introduce the characteristic function $H \in \calF_{+}(\calD, H_\calD)$,
\begin{equation*}
H(e,e') 
 = \pred{ (e,e') \in H_\calD }. %
\end{equation*}
\begin{theorem}
\label{theorem:hudson-as-markov-embedding}
The Hudson \hl{expansion}
\begin{equation*}
\begin{split}
    H_\star \colon \calW(\calX, \calD) &\to \calW_h(\calD, H_\calD), \qquad \bar{P} \mapsto H_\star \bar{P},
\end{split}
\end{equation*}
with 
\begin{equation*}
\begin{split}
    H_\star \bar{P} (e,e') = H(e,e') P(h(e), h(e')),
\end{split}
\end{equation*}
is a Markov embedding congruent with the Hudson lumping $h$.%
\end{theorem}

\begin{proof}
We verify that for $x, x' \in \calX$ and $e = (x_1, x_2) \in \calS_{x}$,
\begin{equation*}
    \begin{split}
        \sum_{e' \in \calS_{x'}} H(e,e') &= \sum_{(x_1',x_2') \in \calD, x_2' = x' } \pred{(x_2, x') \in \calD}\pred{x_2 = x_1'} = \pred{(x, x') \in \calD},
    \end{split}
\end{equation*}
which is independent of $e$.
\end{proof}
\hl{Interestingly, Hudson expansions can be performed without any additional randomness.}
It is easy to obtain \citep[Theorem~6.5.2]{kemeny1983finite} a closed form expression for the stationary distribution of the embedded chain $\pi = H_\star \bar{\pi}$ in terms of the edge measure $\bar{Q}$ that pertains to $\bar{P}$,
\begin{equation*}
H_\star \bar{\pi} (e) = \bar{Q}(e). 
\end{equation*}
Moreover, being a Markov embedding, $H_\star$ is isometric, and preserves the dual affine connections $\nabla^{(e)}, \nabla^{(m)}$ defined in \eqref{eq:mc-affine-connection}.
However, although $H_\star$ is e-geodesic affine (Theorem~\ref{theorem:exponential-embeddings-are-e-geodesic-affine-maps}), we now show that it fails to preserve the m-structure, proving that Markov embeddings --and a fortiori exponential embeddings-- are not generally m-geodesic affine (Theorem~\ref{theorem:exponential-embedding-geodesic-affinity}).

\begin{lemma}
\label{lemma:hudson-breaks-m-structure}
The Hudson \hl{expansion} is not m-geodesic affine.
\end{lemma}

\begin{proof}
See Section~\ref{proof:lemma-hudson-breaks-m-structure}.
\end{proof}

\paragraph{Hudson expansion as a sliding windows of observations.}
\label{section:hudson-embedding-sliding-window}

We can view the Hudson \hl{expansion} more operationally by considering sliding windows of observations
of a Markov chain.
Namely, for $X_1, X_2, \dots,$ an irreducible Markov chain with dynamics governed by $\bar{P} \in \calW(\calX, \calD)$ and stationary $\bar{\pi}$, the stochastic process defined by
$$(X_1, X_2), (X_2, X_3), \dots, (X_t, X_{t+1}), \dots$$ 
also defines a Markov chain, with transition matrix $H_\star \bar{P}$ and stationary distribution $\bar{Q}$ (see for example \citep[Lemma~6.1]{wolfer2021} or \citet{qiu2020matrix}).
In particular, it is straightforward to simulate a trajectory of $H_\star \bar{P}$ as a deterministic function of a trajectory from $\bar{P}$.
Furthermore, let us consider second-order Markov chains over $\calX$, whose dynamics can be encoded in $P^{(2)}$ such that for any $t \in N$, and for any $x,x',x'' \in \calX$,
\begin{equation*}
    P^{(2)}(x,x',x'') = \PR{X_{t+1} = x'' | X_{t-1} = x, X_{t} = x'}.
\end{equation*}
Following the identification in \citet[Section~IV]{csiszar1987conditional}, we introduce $P \in \calW(\calD, H_\calD)$ such that 
\begin{equation*}
    P((x_1, x_2), (x'_1, x'_2)) = P^{(2)}(x_1, x_2, x_2') \pred{ x'_1 = x_2 },
\end{equation*}
i.e. we regard a second-order \hl{stochastic matrix} on $\calX$ as a first-order \hl{stochastic matrix} on $(\calD, H_\calD)$.
This allows us to view the Hudson \hl{expansion} of $\calW(\calX, \calD)$ as a first-order Markov sub-family of $\calW(\calD, H_{\calD})$, the family identified with second-order \hl{stochastic matrices}. 
Note that Lemma~\ref{lemma:hudson-breaks-m-structure} also implies that the Hudson \hl{expansion} of $\calW(\calX, \calD)$ does not form an m-family in $\calW(\calD, H_{\calD})$.

\paragraph{Higher-order Hudson expansions.}
\label{section:hudson-embedding-higher-order}
For an observation window of size $k > 1$,
$$(X_1, X_2, \dots, X_k), (X_2, X_3 \dots, X_{k+1}), \dots,$$
still defines a Markov chain,
inviting us to extend the definition of $H_\star$ to higher orders. 
We write $\calD^{(k)}$ the collection of all possible paths $s = (x_1, x_2, \dots, x_k)$ of length $k$ over the connection graph of $P$. In particular, $\calD^{(2)} = \calD$ and $\calD^{(1)} = \calX$.
The definition of the Hudson lumping extends seamlessly as follows,
\begin{equation*}
\begin{split}
    h^{(k)} \colon \calD^{(k)} &\to \calX \\
    s = (x_1, x_2,\dots, x_k) &\mapsto x_k,
\end{split}
\end{equation*}
and the edge set of the target space is naturally defined by 
\begin{equation*}
    H_\calD^{(k)} \eqdef \set{ (s, s') \in \calD^{(k)} \times \calD^{(k)} \colon \forall t \in [k -1], x_{t+1} = x_t' \text{ \hl{and} } (x_k, x'_k) \in \calD }.
\end{equation*}
The $k$-th order Hudson \hl{expansion} \label{term:k-th-order-hudson-embedding} is then
\begin{equation*}
H^{(k)}_\star \colon \calW(\calX, \calD) \to \calW_{h^{(k)}}\left(\calD^{(k)}, H_\calD^{(k)}\right),
\end{equation*}
where
\begin{equation*}
    H^{(k)}_\star \bar{P}(s, s') = \pred{ (s, s') \in H^{(k)}_\calD } P\left(h^{(k)}(s), h^{(k)}(s')\right),
\end{equation*}
and Theorem~\ref{theorem:hudson-as-markov-embedding} can be extended to any $k$th order. 
Finally, we can also view the $k$th order Hudson \hl{expansion} of $\calW(\calX, \calD)$ as a first-order subfamily of $\calW(\calD^{(k)}, H_{\calD}^{(k)})$, the family identified  \citep[Section~IV]{csiszar1987conditional} with $k$th order \hl{stochastic matrices}.

\subsubsection{Memoryless Markov embeddings}
\label{section:embeddings-memoryless}
Recall that there is a natural one-to-one embedding of the positive simplex $\calP_+(\calX)$ into positive \hl{stochastic matrices} that forms $\calW \iid(\calX)$, the family of irreducible memoryless \hl{stochastic matrices} \citep[Section~8]{wolfer2021information}. 
In the same spirit, we now
define memoryless Markov embeddings, \label{term:MMEmb} as the one-to-one embedding of all Markov morphisms in the context of distribution into Markov morphisms in the context of \hl{stochastic matrices}.
For some fixed family of \hl{stochastic matrices} $\calW (\calX, \calD)$, every memoryless Markov embedding, as a Markov embedding, is associated with 
a lumping function $\kappa$, and
a linear operator $\hlm{\Lambda} \in \calF_+(\calY, \calE)$(Definition~\ref{definition:markov-embeddings}), with the additional property that for any $y,y' \in \calY$,
$$\hlm{\Lambda}(y,y') = L(y') \pred{ (\kappa(y), \kappa(y')) \in \calD},$$
where $L \colon \calY \to \R$ essentially \hl{induces a Markov channel} in the context of distributions.
\label{term:MMEmbk} 
We write 
\begin{equation*}
\begin{split}
    L_\star \colon \calW(\calX, \calD) &\to \calW(\calY, \calE) \\
    \bar{P}(x,x') &\mapsto P(y,y') = P(\kappa(y), \kappa(y')) L(y').
\end{split}
\end{equation*}
The stationary distribution of a \hl{stochastic matrix} embedded with $L_\star$
has a closed-form expression.
\hl{Memorylessly embedded chains are easier to simulate, since they do not require to store the previous point, as depicted in  Figure~}\hlc{\ref{figure:trajectory-simulation-memoryless}}\hl{.}
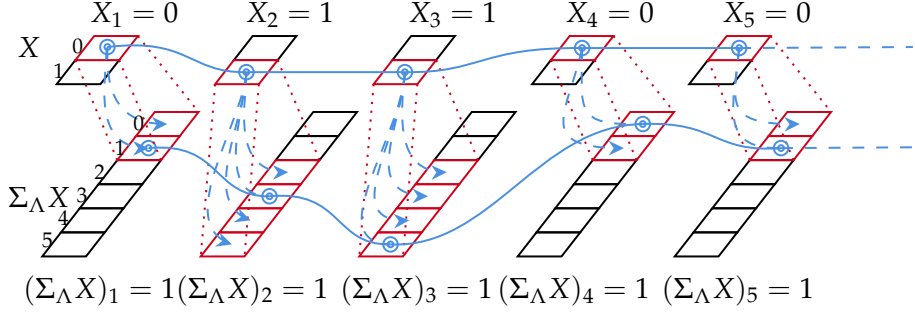
\begin{figure}
    \centering
    \hlp{

\tikzset{every picture/.style={line width=0.75pt}} %

\begin{tikzpicture}[x=0.55pt,y=0.55pt,yscale=-1,xscale=1]

\draw   (74,74) -- (104.1,74) -- (91.2,90.57) -- (61.1,90.57) -- cycle ;
\draw   (89.9,159.43) -- (120,159.43) -- (107.1,176) -- (77,176) -- cycle ;
\draw   (77,176) -- (107.1,176) -- (94.2,192.57) -- (64.1,192.57) -- cycle ;
\draw   (102.8,142.86) -- (132.9,142.86) -- (120,159.43) -- (89.9,159.43) -- cycle ;
\draw  [color={rgb, 255:red, 208; green, 2; blue, 27 }  ,draw opacity=1 ] (128.6,109.71) -- (158.7,109.71) -- (145.8,126.29) -- (115.7,126.29) -- cycle ;
\draw   (64.1,192.57) -- (94.2,192.57) -- (81.3,209.14) -- (51.2,209.14) -- cycle ;
\draw [color={rgb, 255:red, 208; green, 2; blue, 27 }  ,draw opacity=1 ] [dash pattern={on 0.84pt off 2.51pt}]  (117,57.43) -- (158.7,109.71) ;
\draw [color={rgb, 255:red, 208; green, 2; blue, 27 }  ,draw opacity=1 ] [dash pattern={on 0.84pt off 2.51pt}]  (74,74) -- (102.8,142.86) ;
\draw  [color={rgb, 255:red, 208; green, 2; blue, 27 }  ,draw opacity=1 ][fill={rgb, 255:red, 255; green, 255; blue, 255 }  ,fill opacity=1 ] (86.9,57.43) -- (117,57.43) -- (104.1,74) -- (74,74) -- cycle ;
\draw  [color={rgb, 255:red, 208; green, 2; blue, 27 }  ,draw opacity=1 ] (115.7,126.29) -- (145.8,126.29) -- (132.9,142.86) -- (102.8,142.86) -- cycle ;
\draw  [color={rgb, 255:red, 208; green, 2; blue, 27 }  ,draw opacity=1 ] (198.9,159.43) -- (229,159.43) -- (216.1,176) -- (186,176) -- cycle ;
\draw  [color={rgb, 255:red, 208; green, 2; blue, 27 }  ,draw opacity=1 ] (186,176) -- (216.1,176) -- (203.2,192.57) -- (173.1,192.57) -- cycle ;
\draw  [color={rgb, 255:red, 0; green, 0; blue, 0 }  ,draw opacity=1 ] (237.6,109.71) -- (267.7,109.71) -- (254.8,126.29) -- (224.7,126.29) -- cycle ;
\draw  [color={rgb, 255:red, 208; green, 2; blue, 27 }  ,draw opacity=1 ] (173.1,192.57) -- (203.2,192.57) -- (190.3,209.14) -- (160.2,209.14) -- cycle ;
\draw [color={rgb, 255:red, 208; green, 2; blue, 27 }  ,draw opacity=1 ] [dash pattern={on 0.84pt off 2.51pt}]  (200.2,90.57) -- (190.3,209.14) ;
\draw [color={rgb, 255:red, 208; green, 2; blue, 27 }  ,draw opacity=1 ] [dash pattern={on 0.84pt off 2.51pt}]  (213.1,74) -- (241.9,142.86) ;
\draw [color={rgb, 255:red, 208; green, 2; blue, 27 }  ,draw opacity=1 ] [dash pattern={on 0.84pt off 2.51pt}]  (170.1,90.57) -- (160.2,209.14) ;
\draw  [color={rgb, 255:red, 0; green, 0; blue, 0 }  ,draw opacity=1 ][fill={rgb, 255:red, 255; green, 255; blue, 255 }  ,fill opacity=1 ] (195.9,57.43) -- (226,57.43) -- (213.1,74) -- (183,74) -- cycle ;
\draw  [color={rgb, 255:red, 0; green, 0; blue, 0 }  ,draw opacity=1 ] (224.7,126.29) -- (254.8,126.29) -- (241.9,142.86) -- (211.8,142.86) -- cycle ;
\draw   (401,74) -- (431.1,74) -- (418.2,90.57) -- (388.1,90.57) -- cycle ;
\draw   (416.9,159.43) -- (447,159.43) -- (434.1,176) -- (404,176) -- cycle ;
\draw   (404,176) -- (434.1,176) -- (421.2,192.57) -- (391.1,192.57) -- cycle ;
\draw   (429.8,142.86) -- (459.9,142.86) -- (447,159.43) -- (416.9,159.43) -- cycle ;
\draw  [color={rgb, 255:red, 208; green, 2; blue, 27 }  ,draw opacity=1 ] (455.6,109.71) -- (485.7,109.71) -- (472.8,126.29) -- (442.7,126.29) -- cycle ;
\draw   (391.1,192.57) -- (421.2,192.57) -- (408.3,209.14) -- (378.2,209.14) -- cycle ;
\draw [color={rgb, 255:red, 208; green, 2; blue, 27 }  ,draw opacity=1 ] [dash pattern={on 0.84pt off 2.51pt}]  (444,57.43) -- (485.7,109.71) ;
\draw [color={rgb, 255:red, 208; green, 2; blue, 27 }  ,draw opacity=1 ] [dash pattern={on 0.84pt off 2.51pt}]  (431.1,74) -- (459.9,142.86) ;
\draw [color={rgb, 255:red, 208; green, 2; blue, 27 }  ,draw opacity=1 ] [dash pattern={on 0.84pt off 2.51pt}]  (401,74) -- (429.8,142.86) ;
\draw  [color={rgb, 255:red, 208; green, 2; blue, 27 }  ,draw opacity=1 ][fill={rgb, 255:red, 255; green, 255; blue, 255 }  ,fill opacity=1 ] (413.9,57.43) -- (444,57.43) -- (431.1,74) -- (401,74) -- cycle ;
\draw  [color={rgb, 255:red, 208; green, 2; blue, 27 }  ,draw opacity=1 ] (442.7,126.29) -- (472.8,126.29) -- (459.9,142.86) -- (429.8,142.86) -- cycle ;
\draw  [color={rgb, 255:red, 208; green, 2; blue, 27 }  ,draw opacity=1 ] (183,74) -- (213.1,74) -- (200.2,90.57) -- (170.1,90.57) -- cycle ;
\draw  [color={rgb, 255:red, 208; green, 2; blue, 27 }  ,draw opacity=1 ] (211.8,142.86) -- (241.9,142.86) -- (229,159.43) -- (198.9,159.43) -- cycle ;
\draw  [color={rgb, 255:red, 74; green, 144; blue, 226 }  ,draw opacity=1 ] (94,65) .. controls (94,63.9) and (94.9,63) .. (96,63) .. controls (97.1,63) and (98,63.9) .. (98,65) .. controls (98,66.1) and (97.1,67) .. (96,67) .. controls (94.9,67) and (94,66.1) .. (94,65)(91,65) .. controls (91,62.24) and (93.24,60) .. (96,60) .. controls (98.76,60) and (101,62.24) .. (101,65) .. controls (101,67.76) and (98.76,70) .. (96,70) .. controls (93.24,70) and (91,67.76) .. (91,65) ;
\draw  [color={rgb, 255:red, 74; green, 144; blue, 226 }  ,draw opacity=1 ] (122.3,134.57) .. controls (122.3,133.47) and (123.2,132.57) .. (124.3,132.57) .. controls (125.4,132.57) and (126.3,133.47) .. (126.3,134.57) .. controls (126.3,135.68) and (125.4,136.57) .. (124.3,136.57) .. controls (123.2,136.57) and (122.3,135.68) .. (122.3,134.57)(119.3,134.57) .. controls (119.3,131.81) and (121.54,129.57) .. (124.3,129.57) .. controls (127.06,129.57) and (129.3,131.81) .. (129.3,134.57) .. controls (129.3,137.33) and (127.06,139.57) .. (124.3,139.57) .. controls (121.54,139.57) and (119.3,137.33) .. (119.3,134.57) ;
\draw [color={rgb, 255:red, 208; green, 2; blue, 27 }  ,draw opacity=1 ] [dash pattern={on 0.84pt off 2.51pt}]  (104.1,74) -- (132.9,142.86) ;
\draw  [color={rgb, 255:red, 208; green, 2; blue, 27 }  ,draw opacity=1 ] (307.9,159.43) -- (338,159.43) -- (325.1,176) -- (295,176) -- cycle ;
\draw  [color={rgb, 255:red, 208; green, 2; blue, 27 }  ,draw opacity=1 ] (295,176) -- (325.1,176) -- (312.2,192.57) -- (282.1,192.57) -- cycle ;
\draw  [color={rgb, 255:red, 0; green, 0; blue, 0 }  ,draw opacity=1 ] (346.6,109.71) -- (376.7,109.71) -- (363.8,126.29) -- (333.7,126.29) -- cycle ;
\draw  [color={rgb, 255:red, 208; green, 2; blue, 27 }  ,draw opacity=1 ] (282.1,192.57) -- (312.2,192.57) -- (299.3,209.14) -- (269.2,209.14) -- cycle ;
\draw [color={rgb, 255:red, 208; green, 2; blue, 27 }  ,draw opacity=1 ] [dash pattern={on 0.84pt off 2.51pt}]  (309.2,90.57) -- (299.3,209.14) ;
\draw [color={rgb, 255:red, 208; green, 2; blue, 27 }  ,draw opacity=1 ] [dash pattern={on 0.84pt off 2.51pt}]  (322.1,74) -- (350.9,142.86) ;
\draw [color={rgb, 255:red, 208; green, 2; blue, 27 }  ,draw opacity=1 ] [dash pattern={on 0.84pt off 2.51pt}]  (279.1,90.57) -- (269.2,209.14) ;
\draw  [color={rgb, 255:red, 0; green, 0; blue, 0 }  ,draw opacity=1 ][fill={rgb, 255:red, 255; green, 255; blue, 255 }  ,fill opacity=1 ] (304.9,57.43) -- (335,57.43) -- (322.1,74) -- (292,74) -- cycle ;
\draw  [color={rgb, 255:red, 0; green, 0; blue, 0 }  ,draw opacity=1 ] (333.7,126.29) -- (363.8,126.29) -- (350.9,142.86) -- (320.8,142.86) -- cycle ;
\draw  [color={rgb, 255:red, 208; green, 2; blue, 27 }  ,draw opacity=1 ] (292,74) -- (322.1,74) -- (309.2,90.57) -- (279.1,90.57) -- cycle ;
\draw  [color={rgb, 255:red, 208; green, 2; blue, 27 }  ,draw opacity=1 ] (320.8,142.86) -- (350.9,142.86) -- (338,159.43) -- (307.9,159.43) -- cycle ;
\draw  [color={rgb, 255:red, 74; green, 144; blue, 226 }  ,draw opacity=1 ] (205.5,167.71) .. controls (205.5,166.61) and (206.4,165.71) .. (207.5,165.71) .. controls (208.6,165.71) and (209.5,166.61) .. (209.5,167.71) .. controls (209.5,168.82) and (208.6,169.71) .. (207.5,169.71) .. controls (206.4,169.71) and (205.5,168.82) .. (205.5,167.71)(202.5,167.71) .. controls (202.5,164.95) and (204.74,162.71) .. (207.5,162.71) .. controls (210.26,162.71) and (212.5,164.95) .. (212.5,167.71) .. controls (212.5,170.48) and (210.26,172.71) .. (207.5,172.71) .. controls (204.74,172.71) and (202.5,170.48) .. (202.5,167.71) ;
\draw   (509,74) -- (539.1,74) -- (526.2,90.57) -- (496.1,90.57) -- cycle ;
\draw   (524.9,159.43) -- (555,159.43) -- (542.1,176) -- (512,176) -- cycle ;
\draw   (512,176) -- (542.1,176) -- (529.2,192.57) -- (499.1,192.57) -- cycle ;
\draw   (537.8,142.86) -- (567.9,142.86) -- (555,159.43) -- (524.9,159.43) -- cycle ;
\draw  [color={rgb, 255:red, 208; green, 2; blue, 27 }  ,draw opacity=1 ] (563.6,109.71) -- (593.7,109.71) -- (580.8,126.29) -- (550.7,126.29) -- cycle ;
\draw   (499.1,192.57) -- (529.2,192.57) -- (516.3,209.14) -- (486.2,209.14) -- cycle ;
\draw [color={rgb, 255:red, 208; green, 2; blue, 27 }  ,draw opacity=1 ] [dash pattern={on 0.84pt off 2.51pt}]  (552,57.43) -- (593.7,109.71) ;
\draw [color={rgb, 255:red, 208; green, 2; blue, 27 }  ,draw opacity=1 ] [dash pattern={on 0.84pt off 2.51pt}]  (539.1,74) -- (567.9,142.86) ;
\draw [color={rgb, 255:red, 208; green, 2; blue, 27 }  ,draw opacity=1 ] [dash pattern={on 0.84pt off 2.51pt}]  (509,74) -- (537.8,142.86) ;
\draw  [color={rgb, 255:red, 208; green, 2; blue, 27 }  ,draw opacity=1 ][fill={rgb, 255:red, 255; green, 255; blue, 255 }  ,fill opacity=1 ] (521.9,57.43) -- (552,57.43) -- (539.1,74) -- (509,74) -- cycle ;
\draw  [color={rgb, 255:red, 208; green, 2; blue, 27 }  ,draw opacity=1 ] (550.7,126.29) -- (580.8,126.29) -- (567.9,142.86) -- (537.8,142.86) -- cycle ;
\draw  [color={rgb, 255:red, 74; green, 144; blue, 226 }  ,draw opacity=1 ] (288.7,200.86) .. controls (288.7,199.75) and (289.6,198.86) .. (290.7,198.86) .. controls (291.8,198.86) and (292.7,199.75) .. (292.7,200.86) .. controls (292.7,201.96) and (291.8,202.86) .. (290.7,202.86) .. controls (289.6,202.86) and (288.7,201.96) .. (288.7,200.86)(285.7,200.86) .. controls (285.7,198.1) and (287.94,195.86) .. (290.7,195.86) .. controls (293.46,195.86) and (295.7,198.1) .. (295.7,200.86) .. controls (295.7,203.62) and (293.46,205.86) .. (290.7,205.86) .. controls (287.94,205.86) and (285.7,203.62) .. (285.7,200.86) ;
\draw  [color={rgb, 255:red, 74; green, 144; blue, 226 }  ,draw opacity=1 ] (462.2,118) .. controls (462.2,116.9) and (463.1,116) .. (464.2,116) .. controls (465.3,116) and (466.2,116.9) .. (466.2,118) .. controls (466.2,119.1) and (465.3,120) .. (464.2,120) .. controls (463.1,120) and (462.2,119.1) .. (462.2,118)(459.2,118) .. controls (459.2,115.24) and (461.44,113) .. (464.2,113) .. controls (466.96,113) and (469.2,115.24) .. (469.2,118) .. controls (469.2,120.76) and (466.96,123) .. (464.2,123) .. controls (461.44,123) and (459.2,120.76) .. (459.2,118) ;
\draw  [color={rgb, 255:red, 74; green, 144; blue, 226 }  ,draw opacity=1 ] (189.6,82.29) .. controls (189.6,81.18) and (190.5,80.29) .. (191.6,80.29) .. controls (192.7,80.29) and (193.6,81.18) .. (193.6,82.29) .. controls (193.6,83.39) and (192.7,84.29) .. (191.6,84.29) .. controls (190.5,84.29) and (189.6,83.39) .. (189.6,82.29)(186.6,82.29) .. controls (186.6,79.52) and (188.84,77.29) .. (191.6,77.29) .. controls (194.36,77.29) and (196.6,79.52) .. (196.6,82.29) .. controls (196.6,85.05) and (194.36,87.29) .. (191.6,87.29) .. controls (188.84,87.29) and (186.6,85.05) .. (186.6,82.29) ;
\draw  [color={rgb, 255:red, 74; green, 144; blue, 226 }  ,draw opacity=1 ] (298.6,82.29) .. controls (298.6,81.18) and (299.5,80.29) .. (300.6,80.29) .. controls (301.7,80.29) and (302.6,81.18) .. (302.6,82.29) .. controls (302.6,83.39) and (301.7,84.29) .. (300.6,84.29) .. controls (299.5,84.29) and (298.6,83.39) .. (298.6,82.29)(295.6,82.29) .. controls (295.6,79.52) and (297.84,77.29) .. (300.6,77.29) .. controls (303.36,77.29) and (305.6,79.52) .. (305.6,82.29) .. controls (305.6,85.05) and (303.36,87.29) .. (300.6,87.29) .. controls (297.84,87.29) and (295.6,85.05) .. (295.6,82.29) ;
\draw  [color={rgb, 255:red, 74; green, 144; blue, 226 }  ,draw opacity=1 ] (420.5,65.71) .. controls (420.5,64.61) and (421.4,63.71) .. (422.5,63.71) .. controls (423.6,63.71) and (424.5,64.61) .. (424.5,65.71) .. controls (424.5,66.82) and (423.6,67.71) .. (422.5,67.71) .. controls (421.4,67.71) and (420.5,66.82) .. (420.5,65.71)(417.5,65.71) .. controls (417.5,62.95) and (419.74,60.71) .. (422.5,60.71) .. controls (425.26,60.71) and (427.5,62.95) .. (427.5,65.71) .. controls (427.5,68.48) and (425.26,70.71) .. (422.5,70.71) .. controls (419.74,70.71) and (417.5,68.48) .. (417.5,65.71) ;
\draw  [color={rgb, 255:red, 74; green, 144; blue, 226 }  ,draw opacity=1 ] (528.5,65.71) .. controls (528.5,64.61) and (529.4,63.71) .. (530.5,63.71) .. controls (531.6,63.71) and (532.5,64.61) .. (532.5,65.71) .. controls (532.5,66.82) and (531.6,67.71) .. (530.5,67.71) .. controls (529.4,67.71) and (528.5,66.82) .. (528.5,65.71)(525.5,65.71) .. controls (525.5,62.95) and (527.74,60.71) .. (530.5,60.71) .. controls (533.26,60.71) and (535.5,62.95) .. (535.5,65.71) .. controls (535.5,68.48) and (533.26,70.71) .. (530.5,70.71) .. controls (527.74,70.71) and (525.5,68.48) .. (525.5,65.71) ;
\draw  [color={rgb, 255:red, 74; green, 144; blue, 226 }  ,draw opacity=1 ] (557.3,134.57) .. controls (557.3,133.47) and (558.2,132.57) .. (559.3,132.57) .. controls (560.4,132.57) and (561.3,133.47) .. (561.3,134.57) .. controls (561.3,135.68) and (560.4,136.57) .. (559.3,136.57) .. controls (558.2,136.57) and (557.3,135.68) .. (557.3,134.57)(554.3,134.57) .. controls (554.3,131.81) and (556.54,129.57) .. (559.3,129.57) .. controls (562.06,129.57) and (564.3,131.81) .. (564.3,134.57) .. controls (564.3,137.33) and (562.06,139.57) .. (559.3,139.57) .. controls (556.54,139.57) and (554.3,137.33) .. (554.3,134.57) ;
\draw [color={rgb, 255:red, 74; green, 144; blue, 226 }  ,draw opacity=1 ]   (96,65) .. controls (152.5,58.29) and (152,83) .. (191.6,82.29) ;
\draw [color={rgb, 255:red, 74; green, 144; blue, 226 }  ,draw opacity=1 ]   (300.6,82.29) .. controls (353,83) and (352,65) .. (422.5,65.71) ;
\draw [color={rgb, 255:red, 74; green, 144; blue, 226 }  ,draw opacity=1 ]   (191.6,82.29) -- (300.6,82.29) ;
\draw [color={rgb, 255:red, 74; green, 144; blue, 226 }  ,draw opacity=1 ]   (124.3,134.57) .. controls (175,130) and (167.9,168.43) .. (207.5,167.71) ;
\draw [color={rgb, 255:red, 74; green, 144; blue, 226 }  ,draw opacity=1 ]   (207.5,167.71) .. controls (258.2,163.14) and (251.1,201.57) .. (290.7,200.86) ;
\draw [color={rgb, 255:red, 74; green, 144; blue, 226 }  ,draw opacity=1 ]   (290.7,200.86) .. controls (360,204) and (392,119) .. (464.2,118) ;
\draw [color={rgb, 255:red, 74; green, 144; blue, 226 }  ,draw opacity=1 ]   (422.5,65.71) -- (530.5,65.71) ;
\draw [color={rgb, 255:red, 74; green, 144; blue, 226 }  ,draw opacity=1 ]   (464.2,118) .. controls (516,116) and (512,138) .. (559.3,134.57) ;
\draw [color={rgb, 255:red, 74; green, 144; blue, 226 }  ,draw opacity=1 ] [dash pattern={on 4.5pt off 4.5pt}]  (530.5,65.71) -- (650,65) ;
\draw [color={rgb, 255:red, 74; green, 144; blue, 226 }  ,draw opacity=1 ] [dash pattern={on 4.5pt off 4.5pt}]  (559.3,134.57) -- (648.01,134.03) -- (651,134) ;
\draw [color={rgb, 255:red, 74; green, 144; blue, 226 }  ,draw opacity=1 ] [dash pattern={on 4.5pt off 4.5pt}]  (191.6,82.29) .. controls (191.02,104.43) and (181.3,149.53) .. (217.52,151.11) ;
\draw [shift={(220.4,151.14)}, rotate = 178.97] [fill={rgb, 255:red, 74; green, 144; blue, 226 }  ,fill opacity=1 ][line width=0.08]  [draw opacity=0] (10.72,-5.15) -- (0,0) -- (10.72,5.15) -- (7.12,0) -- cycle    ;
\draw [color={rgb, 255:red, 74; green, 144; blue, 226 }  ,draw opacity=1 ] [dash pattern={on 4.5pt off 4.5pt}]  (191.6,82.29) .. controls (190.04,117.11) and (160.53,169.31) .. (192.05,183.29) ;
\draw [shift={(194.6,184.29)}, rotate = 199.04] [fill={rgb, 255:red, 74; green, 144; blue, 226 }  ,fill opacity=1 ][line width=0.08]  [draw opacity=0] (10.72,-5.15) -- (0,0) -- (10.72,5.15) -- (7.12,0) -- cycle    ;
\draw [color={rgb, 255:red, 74; green, 144; blue, 226 }  ,draw opacity=1 ] [dash pattern={on 4.5pt off 4.5pt}]  (191.6,82.29) .. controls (189.05,117.29) and (141.94,184.25) .. (179.31,199.96) ;
\draw [shift={(181.7,200.86)}, rotate = 198.38] [fill={rgb, 255:red, 74; green, 144; blue, 226 }  ,fill opacity=1 ][line width=0.08]  [draw opacity=0] (10.72,-5.15) -- (0,0) -- (10.72,5.15) -- (7.12,0) -- cycle    ;
\draw [color={rgb, 255:red, 74; green, 144; blue, 226 }  ,draw opacity=1 ] [dash pattern={on 4.5pt off 4.5pt}]  (300.6,82.29) .. controls (297.11,114.02) and (295.12,149.78) .. (326.4,151.12) ;
\draw [shift={(329.4,151.14)}, rotate = 178.57] [fill={rgb, 255:red, 74; green, 144; blue, 226 }  ,fill opacity=1 ][line width=0.08]  [draw opacity=0] (10.72,-5.15) -- (0,0) -- (10.72,5.15) -- (7.12,0) -- cycle    ;
\draw [color={rgb, 255:red, 74; green, 144; blue, 226 }  ,draw opacity=1 ] [dash pattern={on 4.5pt off 4.5pt}]  (300.6,82.29) .. controls (298.05,120.23) and (270.15,165.16) .. (313.74,167.62) ;
\draw [shift={(316.5,167.71)}, rotate = 180.86] [fill={rgb, 255:red, 74; green, 144; blue, 226 }  ,fill opacity=1 ][line width=0.08]  [draw opacity=0] (10.72,-5.15) -- (0,0) -- (10.72,5.15) -- (7.12,0) -- cycle    ;
\draw [color={rgb, 255:red, 74; green, 144; blue, 226 }  ,draw opacity=1 ] [dash pattern={on 4.5pt off 4.5pt}]  (300.6,82.29) .. controls (298.05,127.09) and (254.78,184.65) .. (300.68,184.38) ;
\draw [shift={(303.6,184.29)}, rotate = 176.93] [fill={rgb, 255:red, 74; green, 144; blue, 226 }  ,fill opacity=1 ][line width=0.08]  [draw opacity=0] (10.72,-5.15) -- (0,0) -- (10.72,5.15) -- (7.12,0) -- cycle    ;
\draw [color={rgb, 255:red, 74; green, 144; blue, 226 }  ,draw opacity=1 ] [dash pattern={on 4.5pt off 4.5pt}]  (422.5,65.71) .. controls (413.14,109.34) and (388.75,134.25) .. (448.51,134.57) ;
\draw [shift={(451.3,134.57)}, rotate = 179.61] [fill={rgb, 255:red, 74; green, 144; blue, 226 }  ,fill opacity=1 ][line width=0.08]  [draw opacity=0] (10.72,-5.15) -- (0,0) -- (10.72,5.15) -- (7.12,0) -- cycle    ;
\draw [color={rgb, 255:red, 74; green, 144; blue, 226 }  ,draw opacity=1 ] [dash pattern={on 4.5pt off 4.5pt}]  (530.5,65.71) .. controls (527.05,102.44) and (518.27,118.52) .. (569.8,118.03) ;
\draw [shift={(572.2,118)}, rotate = 178.94] [fill={rgb, 255:red, 74; green, 144; blue, 226 }  ,fill opacity=1 ][line width=0.08]  [draw opacity=0] (10.72,-5.15) -- (0,0) -- (10.72,5.15) -- (7.12,0) -- cycle    ;
\draw [color={rgb, 255:red, 74; green, 144; blue, 226 }  ,draw opacity=1 ] [dash pattern={on 4.5pt off 4.5pt}]  (191.6,82.29) .. controls (191,105) and (167.9,168.43) .. (207.5,167.71) ;
\draw [color={rgb, 255:red, 74; green, 144; blue, 226 }  ,draw opacity=1 ] [dash pattern={on 4.5pt off 4.5pt}]  (300.6,82.29) .. controls (298,128) and (242,199) .. (290.7,200.86) ;
\draw [color={rgb, 255:red, 74; green, 144; blue, 226 }  ,draw opacity=1 ] [dash pattern={on 4.5pt off 4.5pt}]  (422.5,65.71) .. controls (422,110) and (417,119) .. (464.2,118) ;
\draw [color={rgb, 255:red, 74; green, 144; blue, 226 }  ,draw opacity=1 ] [dash pattern={on 4.5pt off 4.5pt}]  (530.5,65.71) .. controls (527,103) and (515,136) .. (559.3,134.57) ;
\draw [color={rgb, 255:red, 74; green, 144; blue, 226 }  ,draw opacity=1 ] [dash pattern={on 4.5pt off 4.5pt}]  (95.5,65.71) .. controls (94.91,87.86) and (85.2,132.96) .. (121.42,134.54) ;
\draw [shift={(124.3,134.57)}, rotate = 178.97] [fill={rgb, 255:red, 74; green, 144; blue, 226 }  ,fill opacity=1 ][line width=0.08]  [draw opacity=0] (10.72,-5.15) -- (0,0) -- (10.72,5.15) -- (7.12,0) -- cycle    ;
\draw [color={rgb, 255:red, 74; green, 144; blue, 226 }  ,draw opacity=1 ] [dash pattern={on 4.5pt off 4.5pt}]  (96,65) .. controls (95.42,87.15) and (97.49,117.17) .. (134.3,117.99) ;
\draw [shift={(137.2,118)}, rotate = 178.97] [fill={rgb, 255:red, 74; green, 144; blue, 226 }  ,fill opacity=1 ][line width=0.08]  [draw opacity=0] (10.72,-5.15) -- (0,0) -- (10.72,5.15) -- (7.12,0) -- cycle    ;

\draw (70,57) node [anchor=north west][inner sep=0.75pt]  [color={rgb, 255:red, 0; green, 0; blue, 0 }  ,opacity=1 ] [align=left] {{\scriptsize $0$}};
\draw (57,74) node [anchor=north west][inner sep=0.75pt]  [color={rgb, 255:red, 0; green, 0; blue, 0 }  ,opacity=1 ] [align=left] {{\scriptsize $1$}};
\draw (112,110) node [anchor=north west][inner sep=0.75pt]  [color={rgb, 255:red, 0; green, 0; blue, 0 }  ,opacity=1 ] [align=left] {{\scriptsize $0$}};
\draw (99,127) node [anchor=north west][inner sep=0.75pt]  [color={rgb, 255:red, 0; green, 0; blue, 0 }  ,opacity=1 ] [align=left] {{\scriptsize $1$}};
\draw (85,144) node [anchor=north west][inner sep=0.75pt]  [color={rgb, 255:red, 0; green, 0; blue, 0 }  ,opacity=1 ] [align=left] {{\scriptsize $2$}};
\draw (73,160) node [anchor=north west][inner sep=0.75pt]  [color={rgb, 255:red, 0; green, 0; blue, 0 }  ,opacity=1 ] [align=left] {{\scriptsize $3$}};
\draw (60,175) node [anchor=north west][inner sep=0.75pt]  [color={rgb, 255:red, 0; green, 0; blue, 0 }  ,opacity=1 ] [align=left] {{\scriptsize $4$}};
\draw (49,191) node [anchor=north west][inner sep=0.75pt]  [color={rgb, 255:red, 0; green, 0; blue, 0 }  ,opacity=1 ] [align=left] {{\scriptsize $5$}};

\draw (82,32) node [anchor=north west][inner sep=0.75pt]  [color={rgb, 255:red, 0; green, 0; blue, 0 }  ,opacity=1 ] [align=left] {$X_1=0$};
\draw (189,32) node [anchor=north west][inner sep=0.75pt]  [color={rgb, 255:red, 0; green, 0; blue, 0 }  ,opacity=1 ] [align=left] {$X_2=1$};
\draw (301,32) node [anchor=north west][inner sep=0.75pt]  [color={rgb, 255:red, 0; green, 0; blue, 0 }  ,opacity=1 ] [align=left] {$X_3=1$};
\draw (409,32) node [anchor=north west][inner sep=0.75pt]  [color={rgb, 255:red, 0; green, 0; blue, 0 }  ,opacity=1 ] [align=left] {$X_4=0$};
\draw (517,32) node [anchor=north west][inner sep=0.75pt]  [color={rgb, 255:red, 0; green, 0; blue, 0 }  ,opacity=1 ] [align=left] {$X_5=0$};
\draw (36,220) node [anchor=north west][inner sep=0.75pt]  [color={rgb, 255:red, 0; green, 0; blue, 0 }  ,opacity=1 ] [align=left] {$(\Sigma_\Lambda X)_1 = 1$};
\draw (141,220) node [anchor=north west][inner sep=0.75pt]  [color={rgb, 255:red, 0; green, 0; blue, 0 }  ,opacity=1 ] [align=left] {$(\Sigma_\Lambda X)_2 = 1$};
\draw (254,220) node [anchor=north west][inner sep=0.75pt]  [color={rgb, 255:red, 0; green, 0; blue, 0 }  ,opacity=1 ] [align=left] {$(\Sigma_\Lambda X)_3 = 1$};
\draw (362,220) node [anchor=north west][inner sep=0.75pt]  [color={rgb, 255:red, 0; green, 0; blue, 0 }  ,opacity=1 ] [align=left] {$(\Sigma_\Lambda X)_4 = 1$};
\draw (475,220) node [anchor=north west][inner sep=0.75pt]  [color={rgb, 255:red, 0; green, 0; blue, 0 }  ,opacity=1 ] [align=left] {$(\Sigma_\Lambda X)_5 = 1$};
\draw (32,56) node [anchor=north west][inner sep=0.75pt]   [align=left] {$X$};
\draw (26,159) node [anchor=north west][inner sep=0.75pt]   [align=left] {$\Sigma_\Lambda X$};

\end{tikzpicture}}
    \caption{\hlc{Trajectory simulation with a memoryless embedding, where $\Sigma_\Lambda$ is defined in Section~\ref{section:embeddings} (p.\pageref{operator:randomized-simulation}).}}
    \label{figure:trajectory-simulation-memoryless}
\end{figure}

\begin{lemma}
\label{lemma:memoryless-embedding-stationary-distribution}
Let $L_\star \colon  \calW(\calX, \calD) \to \calW(\calY, \calE)$ be a memoryless Markov embedding, and $\kappa$ the lumping function associated with $L_\star$. Let $\bar{P} \in \calW (\calX, \calD)$ with stationary distribution $\bar{\pi}$.
Then $P = L_\star \bar{P}$ has stationary distribution $\pi$ given by
\begin{equation*}
\pi(y) = \bar{\pi}(\kappa(y))L(y).
\end{equation*}
\end{lemma}
\begin{proof}
We simply verify that $\pi P = P$. Let $y' \in \calY$,
\begin{equation*}
\begin{split}
    \sum_{y \in \calY} \pi(y) P(y,y') &= \sum_{y \in \calY} \bar{\pi}(\kappa(y))L(y) \bar{P}(\kappa(y),\kappa(y')) L(y') \\
    &= \sum_{x \in \calX} \sum_{y \in \calS_x} \bar{\pi}(x)L(y) \bar{P}(x,\kappa(y')) L(y') \\
    &= \sum_{x \in \calX} \bar{\pi}(x) \bar{P}(x,\kappa(y')) L(y')  = \bar{\pi}(\kappa(y')) L(y') = \pi(y'). \\
\end{split}
\end{equation*}
\end{proof}
A direct consequence of Lemma~\ref{lemma:memoryless-embedding-stationary-distribution} is that for Markov chains with rational stationary distributions, we can construct a natural embedding that produces a doubly stochastic matrix.
\begin{corollary}
\label{corollary:embedding-to-doubly-stochastic-kernel}
Let $\bar{P} \in \calW([n], \calD)$ with stationary distribution $\bar{\pi}$ such that
\begin{equation*}
    \bar{\pi} = \left( \frac{p_1}{m}, \dots, \frac{p_n}{m}\right),
\end{equation*}
for $m \in \N$ and $p_1, \dots, p_n \in \N$ with $\sum_{i = 1}^{n} p_i = m$.
There exists a lumping function 
$$\kappa \colon [m] \to [n],$$ 
and a memoryless embedding $L_\star$ such that \begin{equation*}
    L_\star \bar{P} \in \calW \bis ([m], \calE),
\end{equation*}
where $\calE= \set{ (y,y') \in [m]^2 \colon \kappa_2(y,y') \in \calD }$.
\end{corollary}
\begin{proof}
We construct $\kappa$ and $L$ as follows. For any $j \in [m]$,
\begin{equation*}
    \kappa(j) = \argmin_{i \in [n]} \set{ \sum_{k = 1}^{i} p_k \geq j}, \qquad L(j) = p_{\kappa(j)}^{-1}.
\end{equation*}
Then, for any $j \in \N$, $L_\star \bar{\pi}(j) = \bar{\pi}(\kappa(j))L(j) = \frac{1}{m}$. The stationary distribution is uniform, thus $L_\star \bar{P}$ is bistochastic.
\end{proof}

\begin{remark}
Since a transition \hl{stochastic matrix} with rational entries enjoys a rational stationary distribution, any transition \hl{stochastic matrix} can be embedded into a doubly stochastic one, modulo some rational approximation.
\end{remark}

\begin{lemma}
\label{lemma:memoryless-embedding-m-geodesic-affine}
Memoryless Markov embeddings are m-geodesic affine.
\end{lemma}
\begin{proof}
Let $\bar{P}_0, \bar{P}_1 \in \calW(\calX, \calD)$, a memoryless Markov embedding $L_\star$, and write $P_0 = L_\star \bar{P}_0, P_1 = L_\star \bar{P}_1$. 
For simplicity, denote also $\bar{P}_t = \gamma^{(m)}_{\bar{P}_0, \bar{P}_1}(t)$ and $P_t = \gamma^{(m)}_{P_0, P_1}(t)$ for points on the geodesics at parameter time $t$.
Write $\bar{Q}_0, \bar{Q}_1, \bar{Q}_t, Q_0, Q_1, Q_t$  $\bar{\pi}_0, \bar{\pi}_1, \bar{\pi}_t, \pi_0, \pi_1, \pi_t$ for the respective edge measures and stationary distributions.
By definition, $P_t$ is such that $Q_t$ satisfies for any $y,y' \in \calY$,
\begin{equation*}
\begin{split}
    Q_t(y,y') &= (1-t) Q_0(y,y')  + t Q_1(y,y') \\
    &= (1-t) \bar{\pi}_0(\kappa(y)) L(y) \bar{P}_0(\kappa(y),\kappa(y'))L(y')  + t \bar{\pi}_1(\kappa(y))L(y) \bar{P}_1(\kappa(y),\kappa(y'))L(y') \\
    &= L(y) L(y') \left[(1-t)  \bar{Q}_0(\kappa(y),\kappa(y'))  + t \bar{Q}_1(\kappa(y),\kappa(y')) \right] \\
    &= L(y) L(y') \bar{Q}_t(\kappa(y),\kappa(y')),
\end{split}
\end{equation*}
and by marginalization,
\begin{equation*}
\begin{split}
    \pi_t(y) &= L(y) \sum_{x' \in \calX} \sum_{y' \in \calS_{x'}}  L(y') \bar{Q}_t(\kappa(y), x') = L(y) \sum_{x' \in \calX} \bar{Q}_t(\kappa(y), x') = L(y) \bar{\pi}_t(\kappa(y)). \\
\end{split}
\end{equation*}
Consequently,
\begin{equation*}
\begin{split}
    P_t(y,y') &= \frac{Q_t(y,y')}{\pi_t(y)} = \frac{L(y') \bar{Q}_t(\kappa(y),\kappa(y'))}{\bar{\pi}_t(\kappa(y))} = \bar{P}_t(\kappa(y), \kappa(y')) L(y') = L_\star \bar{P}_t(y,y'),\\
\end{split}
\end{equation*}
and $L_\star$ is m-geodesically affine.
\end{proof}

Since memoryless Markov embeddings are also e-geodesic affine,
they preserve the entire \hl{e-structure and m-structure} \hlc{(Section~\ref{section:preliminaries-information-geometry})} of \hl{stochastic matrices}.

\subsubsection{Reversible embeddings}
\label{section:embeddings-reversible}
A Markov chain is reversible when its transition \hl{matrix} $P \in \calW(\calY, \calE)$ with stationary distribution $\pi$ satisfies the detailed balance equation, i.e. for any $y,y' \in \calY$,
\begin{equation*}
    \pi(y) P(y,y') = \pi(y') P(y', y).
\end{equation*}
Recall that $\calW \rev (\calY, \calE)$ denotes the subset of $\calW(\calY, \calE)$ of reversible Markov chains.
While lumping always preserves reversibility of \hl{stochastic matrices} (Proposition~\ref{proposition:lumping-reversibility}), embedding a reversible chain ---even by a Markov embedding--- can yield a non-reversible one.
To illustrate this fact, consider for example the Hudson \hl{expansion} of $\bar{P}_0 \in \calW \rev(\calX = \set{0, 1}, \calX^2)$ as defined in Lemma~\ref{lemma:hudson-breaks-m-structure}, and notice that $H(\calX^2)$ is not symmetric, precluding the reversibility\footnote{It is however easy to verify that ``\emph{the [time] reverse process for the expanded process is simply the reverse process for the origin chain expanded} '' \citep[p.144]{kemeny1983finite}.} of $H_\star \bar{P}_0$.
Some Markov embeddings, however, do preserve reversibility (Lemma~\ref{lemma:memoryless-embedding-reversibility}). 

\begin{proposition}[{\citet[Theorem~6.4.7]{kemeny1983finite}}]
\label{proposition:lumping-reversibility}
Let $P \in \calW_\kappa(\calY, \calE) \cap \calW \rev(\calY, \calE)$.
Then $\kappa_\star P \in \calW \rev(\calX, \calD)$.
\end{proposition}
\begin{proof}

Write $\bar{P} = \kappa_\star P$, and $\bar{\pi}$ for the stationary distribution of $\bar{P}$.
It immediately follows from Corollary~\ref{corollary:lumped-stationary-distribution-and-edge-measure}, that for any $(x,x') \in \calD$,
\begin{equation*}
    \bar{\pi}(x)\bar{P}(x,x') = \sum_{y \in \calS_x} \pi(y) \sum_{y' \in \calS_{x'}}P(y,y') = \bar{\pi}(x')\bar{P}(x',x).
\end{equation*}
\end{proof}

\begin{remark}
Interestingly, Proposition~\ref{proposition:lumping-reversibility} implies that Markov embeddings are ``non-reversibility preserving" in the sense that a non-reversible \hl{stochastic matrix} cannot be congruently embedded into a reversible one.
\end{remark}

\begin{lemma}
\label{lemma:memoryless-embedding-reversibility}
Memoryless Markov embeddings preserve reversibility.
Let 
$$L_\star \colon  \calW(\calX, \calD) \to \calW(\calY, \calE)$$ be a Memoryless Markov embedding. Then,
\begin{equation*}
    \bar{P} \in \calW \rev(\calX, \calD) \implies L_\star \bar{P} \in \calW \rev(\calY, \calE).
\end{equation*}
\end{lemma}
\begin{proof}
Let $\bar{P} \in \calW \rev(\calX, \calD)$, with corresponding stationary distribution $\bar{\pi}$, and let $P = L_\star \bar{P}$, with stationary distribution $\pi$.
We verify that $P$ satisfies the detailed-balance equation. For $y,y' \in \calY$,
\begin{equation*}
\begin{split}
    \pi(y)P(y,y') &\stackrel{(i)}{=} \bar{\pi}(\kappa(y))L(y) \bar{P}(\kappa(y), \kappa(y'))L(y') \\
    &\stackrel{(ii)}{=} \bar{\pi}(\kappa(y'))L(y) \bar{P}(\kappa(y'), \kappa(y))L(y') = \pi(y')P(y',y),\\
\end{split}
\end{equation*}
where $(i)$ follows from Lemma~\ref{lemma:memoryless-embedding-stationary-distribution}, and $(ii)$ from reversibility of $\bar{P}$.
\end{proof}

This last observation enables us to isometrically embed elements of $\calW \rev (\calX)$ with rational stationary distributions into $\calW \sym (\calY)$.

\begin{corollary}
\label{corollary:embed-reversible-to-symmetric}
Let $\bar{P} \in \calW([n], \calD) \cap \calW \rev([n])$, with stationary distribution $\bar{\pi}$ such that
\begin{equation*}
    \bar{\pi} = \left( \frac{p_1}{m}, \dots, \frac{p_n}{m}\right),
\end{equation*}
for $m \in \N$ and $p_1, \dots, p_n \in \N$ with $\sum_{i = 1}^{n} p_i = m$.
There exists a lumping function $\kappa \colon [m] \to [n]$ and a memoryless embedding $L_\star$ such that \begin{equation*}
    L_\star \bar{P} \in \calW \sym ([m], \calE),
\end{equation*}
where $\calE= \set{ (y,y') \in [m]^2 \colon \kappa_2(y,y') \in \calD }$.
\end{corollary}

\begin{proof}
Noting that 
$$\calW \sym([m], \calE) = \calW \bis([m], \calE) \cap \calW \rev([m], \calE),$$
Lemma~\ref{lemma:memoryless-embedding-reversibility} and  Corollary~\ref{corollary:embedding-to-doubly-stochastic-kernel} yield the claim.
\end{proof}

Finally, we show that the canonical Markov embedding induced from a reversible \hl{stochastic matrix} is reversibility preserving.

\begin{lemma}
\label{lemma:reversibility-preserving-embedding-stationary-distribution}
Let $P_0 \in \calW_\kappa(\calY, \calE) \cap \calW \rev (\calY, \calE)$, and let the canonical embedding (Lemma~\ref{lemma:construct-canonical-embedding}), such that $y,y' \in \calY$,
\begin{equation*}
    \Lambda^{(P_0)}(y,y') = \frac{P_0(y,y')}{\bar{P}_0(\kappa(y), \kappa(y'))}.
\end{equation*}
Then for any $\bar{P} \in \calW \rev(\calX, \calD)$,  $\Lambda^{(P_0)}_\star \bar{P}$ is also reversible.
Moreover, the stationary distribution $\pi$ of $\Lambda^{(P_0)}_\star \bar{P}$ is given by
\begin{equation*}
\pi(y)  =  \bar{\pi}(\kappa(y)) \frac{\pi_0(y)}{\bar{\pi}_0(\kappa(y))}, y \in \calY,
\end{equation*}
where $\bar{\pi}$, $\bar{\pi}_0$ and $\pi_0$ are the respective stationary distributions of $\bar{P}$, $\bar{P}_0$ and $P_0$.
\end{lemma}
\begin{proof}
We verify that the embedded \hl{stochastic matrix} verifies the detailed-balance equation for the proposed distribution \footnote{A more constructive proof consists in first invoking the Kolmogorov criterion to show that $P$ is reversible. The expression of the stationary distribution may subsequently be derived from the detailed-balance equation.}.
Recall from Proposition~\ref{proposition:lumping-reversibility} that when $P_0$ is reversible, $\bar{P}_0$ is also reversible.
Let $(y, y') \in \calE$,
\begin{equation*}
\begin{split}
    \pi(y) P(y,y') &= \bar{\pi}(\kappa(y)) \bar{P}(\kappa(y), \kappa(y')) \frac{\pi_0(y)P_0(y, y')}{\bar{\pi}_0(\kappa(y)\bar{P}_0(\kappa(y, y')))} \\
    &= \bar{\pi}(\kappa(y')) \bar{P}(\kappa(y'), \kappa(y)) \frac{\pi_0(y')P_0(y', y)}{\bar{\pi}_0(\kappa(y')\bar{P}_0(\kappa(y', y)))} = \pi(y')P(y', y).\\
\end{split}
\end{equation*}
As a result, $P$ is reversible and $\pi$ is the stationary distribution (\citet[Proposition~1.20]{levin2009markov}).
\end{proof}

\hlp{
\subsubsection{Application to inference in Markov chains from a single trajectory of observations}
\label{section:application-identity-testing-reversible-markov-chains}
We briefly describe an application of Markov embeddings for identity testing of Markov chains within the context of the property testing framework
\citep{daskalakis2018testing, pmlr-v99-cherapanamjeri19a, pmlr-v108-wolfer20a, pmlr-v151-fried22a}.
Given a fixed proximity parameter $\eps \in (0,1)$ and a reference stochastic matrix $P_0 \in \calW(\calX, \calD)$, along with a single trajectory of observations sampled according to an unknown stochastic matrix $P \in \calW(\calX, \calD)$, the task consists in constructing an algorithm to distinguish between the two hypotheses,
\begin{equation*}
    P = P_0 \;\;\;\text{ or }\;\;\; K(P, P_0) > \eps.
\end{equation*}
Here
\begin{equation*}
    K(P,P') \eqdef 1 - \rho\left( P^{\circ 1/2} \circ P'^{\circ 1/2} \right),
\end{equation*}
based on the work of \citet{kazakos1978bhattacharyya}, serves as a measure of contrast, and was first introduced in this context by \citet{daskalakis2018testing}.
Inspired by \citet{goldreich2016uniform}, who reduced the problem of identity testing to an arbitrary distribution \citep{valiant2017automatic} to the problem of uniformity testing \citep{paninski2008coincidence} in the iid setting,
\citep{wolfer2023geometric}
reduced the problem of testing identity of $\pi$-reversible Markov chains to testing identity of symmetric Markov chains. To achieve this, the authors of \citep{wolfer2023geometric} relied on the embedding of Corollary~\ref{corollary:embed-reversible-to-symmetric}, and demonstrated its preservation of the contrast function $K$. 
Not only does this approach simplify the methodology, but it also recovers the state-of-the-art result of \citep{pmlr-v151-fried22a} from \citep{pmlr-v99-cherapanamjeri19a}.

}

\begin{table}
    \centering
    \begin{tabular}{r|l|l}
Symbol & Description & First apparition \\
\hline
        $\Emb$ & Totality of embeddings from $\calV \subset \calW(\calX, \calD)$ & Section~\ref{term:Emb} (p.\pageref{term:Emb})\\
        $\Embk$ & $\kappa$-compatible embeddings & Section~\ref{term:Embk} (p.\pageref{term:Embk}) \\
        $\Embkc$ & $\kappa$-congruent embeddings & Definition~\ref{term:Embkc} \\ %
        $\MEmb$ & Markov embeddings & Definition~\ref{term:MEmb} \\ %
        $\MEmbk$ & $\kappa$-compatible Markov embeddings & Definition~\ref{term:MEmbk} \\ %
        $\mEmb$ & m-geodesic affine embeddings & Definition~\ref{definition:geodesic-affine-map} \\ %
        $\eEmb$ & e-geodesic affine embeddings &  Definition~\ref{definition:geodesic-affine-map} \\ %
        $\EEmb$ & Exponential embeddings & Definition~\ref{definition:exponential-embedding}  \\ %
        $\EEmbk$ & $\kappa$-compatible exponential embeddings & Definition~\ref{definition:exponential-embedding}  \\%(p.\pageref{term:EEmbk}) \\
        $\EEmbkc$ & $\kappa$-compatible exp. emb. whose carrier lumps into $C$& Definition~\ref{definition:exponential-embedding}  \\% (p.\pageref{term:EEmbkc}) \\
        $\MMEmbk$ & $\kappa$-compatible memoryless Markov embeddings & Section~\ref{term:MMEmbk} (p.\pageref{term:MMEmb}) \\
        $H^{(k)}_\star$ & The Hudson \hl{expansion} of order $k$ &  Section~\ref{section:embeddings-hudson} (p.\pageref{term:k-th-order-hudson-embedding}) \\
        \end{tabular}
        \caption{Nomenclature of embedding classes. \label{table:nomenclature}}
\end{table}

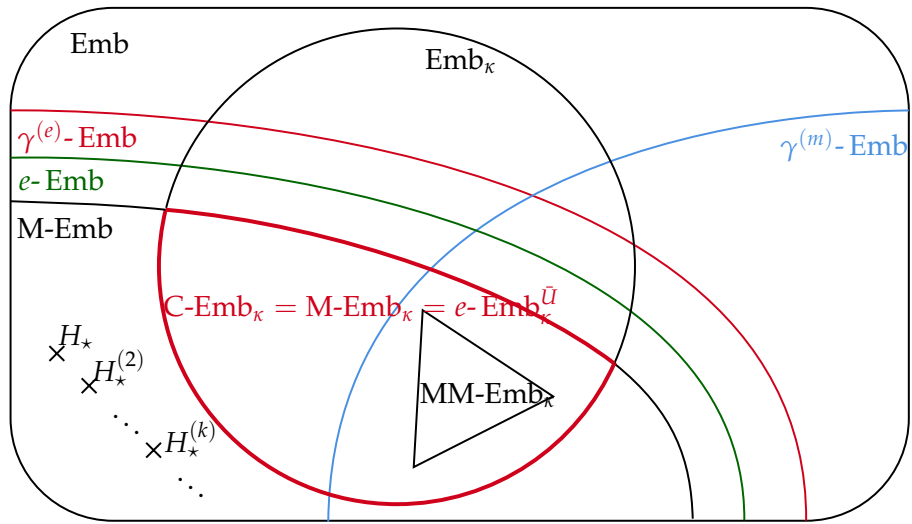
\begin{figure}%
\centering

\tikzset{every picture/.style={line width=0.75pt}} %

\begin{tikzpicture}[x=0.61pt,y=0.61pt,yscale=-1,xscale=1]

\draw   (94.8,70.66) .. controls (94.8,35.5) and (123.3,7) .. (158.46,7) -- (588.14,7) .. controls (623.3,7) and (651.8,35.5) .. (651.8,70.66) -- (651.8,261.63) .. controls (651.8,296.79) and (623.3,325.29) .. (588.14,325.29) -- (158.46,325.29) .. controls (123.3,325.29) and (94.8,296.79) .. (94.8,261.63) -- cycle ;
\draw [color={rgb, 255:red, 208; green, 2; blue, 27 }  ,draw opacity=1 ]   (94.8,70.66) .. controls (120.8,70) and (589.8,70) .. (588.14,325.29) ;
\draw [color={rgb, 255:red, 74; green, 144; blue, 226 }  ,draw opacity=1 ]   (291.8,326) .. controls (294.14,74.71) and (621.8,71) .. (651.8,70.66) ;
\draw [color={rgb, 255:red, 0; green, 102; blue, 0 }  ,draw opacity=1 ]   (94.8,100) .. controls (247.8,97) and (550.8,145) .. (549.8,325) ;
\draw    (469.16,227.71) .. controls (506.8,258) and (516.8,288) .. (517.8,324) ;
\draw  [draw opacity=0][line width=1.5]  (469.49,226.95) .. controls (446.6,278.81) and (394.73,315) .. (334.4,315) .. controls (252.88,315) and (186.8,248.92) .. (186.8,167.4) .. controls (186.8,155.3) and (188.26,143.54) .. (191,132.28) -- (334.4,167.4) -- cycle ; \draw  [color={rgb, 255:red, 208; green, 2; blue, 27 }  ,draw opacity=1 ][line width=1.5]  (469.49,226.95) .. controls (446.6,278.81) and (394.73,315) .. (334.4,315) .. controls (252.88,315) and (186.8,248.92) .. (186.8,167.4) .. controls (186.8,155.3) and (188.26,143.54) .. (191,132.28) ;
\draw  [draw opacity=0] (191.29,131.11) .. controls (207.47,67.14) and (265.41,19.8) .. (334.4,19.8) .. controls (415.92,19.8) and (482,85.88) .. (482,167.4) .. controls (482,188.88) and (477.41,209.3) .. (469.16,227.71) -- (334.4,167.4) -- cycle ; \draw   (191.29,131.11) .. controls (207.47,67.14) and (265.41,19.8) .. (334.4,19.8) .. controls (415.92,19.8) and (482,85.88) .. (482,167.4) .. controls (482,188.88) and (477.41,209.3) .. (469.16,227.71) ;
\draw [color={rgb, 255:red, 208; green, 2; blue, 27 }  ,draw opacity=1 ][line width=1.5]    (191,132.28) .. controls (298.8,142) and (410.8,183) .. (469.16,227.71) ;
\draw    (94.8,127) .. controls (112.31,127.89) and (149.8,128) .. (191,132.28) ;
\draw   (118.9,216.9) -- (128.1,226.1)(128.1,216.9) -- (118.9,226.1) ;
\draw   (138.9,236.9) -- (148.1,246.1)(148.1,236.9) -- (138.9,246.1) ;
\draw   (178.9,276.9) -- (188.1,286.1)(188.1,276.9) -- (178.9,286.1) ;
\draw   (431.22,248.1) -- (344.8,292) -- (350.36,194.64) -- cycle ;

\draw (130,20) node [anchor=north west][inner sep=0.75pt]   [align=left] {$\Emb$};
\draw (97,75) node [anchor=north west][inner sep=0.75pt]  [color={rgb, 255:red, 208; green, 2; blue, 27 }  ,opacity=1 ] [align=left] {$\eEmb$};
\draw (570,80) node [anchor=north west][inner sep=0.75pt]  [color={rgb, 255:red, 74; green, 144; blue, 226 }  ,opacity=1 ] [align=left] {$\mEmb$};
\draw (98,105) node [anchor=north west][inner sep=0.75pt]  [color={rgb, 255:red, 0; green, 102; blue, 0 }  ,opacity=1 ] [align=left] {$\EEmb$};
\draw (350,30) node [anchor=north west][inner sep=0.75pt]  [color={rgb, 255:red, 0; green, 0; blue, 0 }  ,opacity=1 ] [align=left] {$\Embk$};
\draw (97,135) node [anchor=north west][inner sep=0.75pt]   [align=left] {$\MEmb$};
\draw (188,178) node [anchor=north west][inner sep=0.75pt]  [color={rgb, 255:red, 208; green, 2; blue, 27 }  ,opacity=1 ] [align=left] {$\Embkc = \MEmbk = \EEmbk^{\bar{U}}$};
\draw (122,202) node [anchor=north west][inner sep=0.75pt]   [align=left] {$H_\star$};
\draw (143,217) node [anchor=north west][inner sep=0.75pt]   [align=left] {$H_\star^{(2)}$};
\draw (188,261) node [anchor=north west][inner sep=0.75pt]   [align=left] {$H_\star^{(k)}$};
\draw (155,246) node [anchor=north west][inner sep=0.75pt]   [align=left] {$\ddots$};
\draw (195,285) node [anchor=north west][inner sep=0.75pt]   [align=left] {$\ddots$};
\draw (347,237) node [anchor=north west][inner sep=0.75pt]   [align=left] {$\MMEmbk$};

\end{tikzpicture}

\caption{Landscape of the different classes of embeddings, as summarized in Table~\ref{table:nomenclature} ($\kappa \neq h^{(k)}, k \in \N$).
It is instructive to observe that unlike the distribution setting, congruent Markov embeddings are not necessarily m-geodesic affine.
}
\label{figure:embeddings-landscape}
\end{figure}

\section{Information projection on geodesically convex sets}
\label{section:information-projection}
Geodesic convexity generalizes the familiar notion of convexity in the Euclidean space,
to Riemannian manifolds. Recall that in the Riemannian setting, \emph{straight lines}, termed geodesics, are defined with respect to some affine connection $\nabla$. A submanifold $\calC$ is geodesically convex with respect to $\nabla$ whenever all $\nabla$-geodesics joining two points in $\calC$ remain in $\calC$ at all times.

\begin{definition}[Geodesically convex family]
\label{definition:geodesically-convex-family}
Let $\calC$ be a submanifold of $\calW(\calX, \calD)$. The family
$\calC$ is said to be e-convex (resp. m-convex), when for any $P_0, P_1 \in \calC$ and any $t \in [0,1]$, it holds that $\gamma_{P_0, P_1}^{(e)}(t) \in \calC$ (resp. $\gamma_{P_0, P_1}^{(m)}(t) \in \calC$).
\end{definition}

\begin{remark}
Note that an e-family (resp. m-family)  $\calV \subset \calW(\calX, \calD)$ is e-convex (resp. m-convex) with respect to $\nabla^{(e)}$ (resp. $\nabla^{(m)}$).
\end{remark}

\begin{definition}[Geodesically convex function]
Let $\calC \subset \calW(\calX, \calD)$ be geodesically convex. 
A function $\phi: \calC \to \R$ is said to be a (strictly) geodesically convex if the composition
$\phi \circ \gamma \colon [0,1] \to \R$
is a (strictly) convex function for any geodesic arc $\gamma \colon [0, 1] \to \calC$ contained within $\calC$.
\end{definition}

Recall that the information (KL) divergence, as defined in \eqref{equation:kl-divergence}, is strictly m-convex in its first argument \citep[Theorem~3.3]{hayashi2014information}, and strictly e-convex in its second argument \citep[Lemma~4.5]{hayashi2014information}\footnote{\citet[Theorem~3.3,Lemma~4.5]{hayashi2014information} are only stated in terms of convexity, but a straightforward adaptation of the proofs yields the strong convexity.}, i.e. for $t \in (0,1)$, $P, P_0, P_1 \in \calW(\calX, \calD)$, with $P_0 \neq P_1$,
\begin{equation*}
\begin{split}
    \kl{\gamma_{P_0, P_1}^{(m)}(t)}{P} &< (1 -t) \kl{P_0}{P} + t \kl{P_1}{P}, \\
    \kl{P}{\gamma_{P_0, P_1}^{(e)}(t)} &< (1-t) \kl{P}{P_0} + t \kl{P}{P_1}.
\end{split}
\end{equation*}
We complement these results by noting that for the e-geodesic 
we readily obtain the general expression for any $t \in \R$,
\begin{equation*}
    \kl{P}{\gamma_{P_0, P_1}^{(e)}(t)} = (1 - t) \kl{P}{P_0} + t\kl{P}{P_1} + \log \rho_t,
\end{equation*}
where $\rho_t$ is the PF root of $P_0^{\circ 1 -t} \circ P_1^{\circ t}$.
When $\abs{t} > 1$, the strong convexity of $t \mapsto \rho_t$ and $\rho_0 = \rho_1 = 1$ immediately implies the following property of the information divergence,
\begin{equation*}
\begin{split}
    \kl{P}{\gamma_{P_0, P_1}^{(e)}(t)} &>  (1 - t)\kl{P}{P_0} + t \kl{P}{P_1}.
\end{split}
\end{equation*}

\begin{remark}
It is also noteworthy that the information divergence for \hl{stochastic matrices} does not seem to enjoy any joint m-convexity. Consider
\begin{equation*}
    P_0 = \begin{pmatrix}
    1 / 2 & 1/ 2 \\ 1 / 2 & 1/ 2
    \end{pmatrix},
    P'_0 = \begin{pmatrix}
    1 / 4 & 3/ 4 \\ 1 / 4 & 3/ 4
    \end{pmatrix},
    P_1 = P'_1  = \begin{pmatrix}
    1 / 8 & 7/ 8 \\ 7 / 8 & 1/ 8
    \end{pmatrix}.
\end{equation*}
Then simple calculations show that at $t = 1/2$,
\begin{equation*}
    \kl{\gamma^{(m)}_{P_0, P_1}(t)}{\gamma^{(m)}_{P'_0, P'_1}(t)} > (1 - t) \kl{P_0}{P_0'} + t \kl{P_1}{P'_1}.
\end{equation*}
This contrasts with the distribution setting, where the information divergence belongs to the class of $f$-divergences, hence is jointly m-convex.
\end{remark}

\subsection{Pythagorean inequalities}

In Euclidean geometry, the projection of a point onto a convex body yields a natural Pythagorean inequality involving the point, its projection, and other points on the surface of the convex body.
An information geometric analogue of this fact, with the information divergence in lieu of the squared Euclidean distance, is also well-known to hold in the simplex (see e.g. \citet[Theorem~3.1]{csiszar2004information}).
We briefly recall an extension of this Pythagorean inequality for m-convex sets of \hl{stochastic matrices} \citet[Lemma~1]{csiszar1987conditional}.
Let $\calC \subset \calW(\calX, \calD')$ with $\calD' \subset \calD$ be non-empty, closed and m-convex,
we define the \emph{$\IPr$-projection} (information projection) onto $\calC$ as the mapping\footnote{
When $\dim \calC < \dim \calW(\calX, \calD)$, note that $\hlm{\IPr^{(\calC)}}$ cannot be diffeomorphic.}
\begin{equation*}
\begin{split}
\hlm{\IPr^{(\calC)}} \colon \calW(\calX, \calD) \to \calC, \qquad 
P \mapsto \argmin_{\bar{P} \in \calC} \kl{\bar{P}}{P}.
\end{split}
\end{equation*}
For a fixed $P \in \calW(\calX, \calD)$, the function $\bar{P} \mapsto \kl{\bar{P}}{P}$ is continuous and strictly m-convex.

\begin{proposition}[Pythagorean inequality -- $\IPr$-projection onto m-convex -- {\citep[Lemma~1]{csiszar1987conditional}}]
\label{proposition:pythagorean-inequality-m-convex}
Let $P \in \calW(\calX, \calD)$, and
let $\calC \subset \calW(\calX, \calD')$ with $\calD' \subset \calD$, be non-empty, closed and m-convex (Definition~\ref{definition:geodesically-convex-family}).
Then 
$\hlm{\IPr^{(\calC)}P}$ exists in the sense where the minimum is attained for a unique element in $\calC$. Let $P_0 \in \calC$, 
the following two statements are equivalent:
\begin{enumerate}
    \item[$(i)$] For any $\bar{P} \in \calC$, 
$D( \bar{P} \| P) \geq D( \bar{P} \| P_0) + D( P_0 \| P)$. See Figure~\ref{figure:pythagorean-inequalities} (left).
\item[$(ii)$] $P_0 = \hlm{\IPr^{(\calC)} P}$.
\end{enumerate}
\end{proposition}

\begin{figure}%
\centering

\tikzset{every picture/.style={line width=0.75pt}} %

\begin{tikzpicture}[x=0.50pt,y=0.50pt,yscale=-1,xscale=1]

\draw   (197.3,59) -- (333.8,278) -- (60.8,278) -- cycle ;
\draw    (128.8,171) .. controls (168.6,186) and (175.8,224) .. (168.8,278) ;
\draw   (151.72,190.16) -- (162.28,195.84)(159.84,187.72) -- (154.16,198.28) ;
\draw   (150.76,243.76) -- (159.24,252.24)(159.24,243.76) -- (150.76,252.24) ;
\draw   (216.37,157.94) -- (227.63,162.06)(224.06,154.37) -- (219.94,165.63) ;
\draw [color={rgb, 255:red, 208; green, 2; blue, 27 }  ,draw opacity=1 ]   (221.8,160) -- (156.8,193) ;
\draw [shift={(189.3,176.5)}, rotate = 333.08] [fill={rgb, 255:red, 208; green, 2; blue, 27 }  ,fill opacity=1 ][line width=0.08]  [draw opacity=0] (10.72,-5.15) -- (0,0) -- (10.72,5.15) -- (7.12,0) -- cycle    ;
\draw [color={rgb, 255:red, 208; green, 2; blue, 27 }  ,draw opacity=1 ]   (157.8,193) -- (154.8,248) ;
\draw [shift={(156.3,220.5)}, rotate = 273.12] [fill={rgb, 255:red, 208; green, 2; blue, 27 }  ,fill opacity=1 ][line width=0.08]  [draw opacity=0] (10.72,-5.15) -- (0,0) -- (10.72,5.15) -- (7.12,0) -- cycle    ;
\draw [color={rgb, 255:red, 208; green, 2; blue, 27 }  ,draw opacity=1 ]   (154.8,248) .. controls (189.8,216) and (210.8,203) .. (221.8,160) ;
\draw [shift={(196.3,209.26)}, rotate = 312.87] [fill={rgb, 255:red, 208; green, 2; blue, 27 }  ,fill opacity=1 ][line width=0.08]  [draw opacity=0] (10.72,-5.15) -- (0,0) -- (10.72,5.15) -- (7.12,0) -- cycle    ;

\draw (117,200) node [anchor=north west][inner sep=0.75pt]   [align=left] {$\calC$};
\draw (40,108) node [anchor=north west][inner sep=0.75pt]   [align=left] {$\calW(\calX, \calD)$};
\draw (203,139) node [anchor=north west][inner sep=0.75pt]   [align=left] {$P$};
\draw (140,160) node [anchor=north west][inner sep=0.75pt]   [align=left] {$\hlm{\IPr P}$};
\draw (130,240) node [anchor=north west][inner sep=0.75pt]   [align=left] {$\bar{P}$};
\end{tikzpicture}
\hspace{1cm}
\begin{tikzpicture}[x=0.50pt,y=0.50pt,yscale=-1,xscale=1]

\draw   (197.3,59) -- (333.8,278) -- (60.8,278) -- cycle ;
\draw    (128.8,171) .. controls (168.6,186) and (175.8,224) .. (168.8,278) ;
\draw   (151.72,190.16) -- (162.28,195.84)(159.84,187.72) -- (154.16,198.28) ;
\draw   (150.76,243.76) -- (159.24,252.24)(159.24,243.76) -- (150.76,252.24) ;
\draw   (216.37,157.94) -- (227.63,162.06)(224.06,154.37) -- (219.94,165.63) ;
\draw [color={rgb, 255:red, 208; green, 2; blue, 27 }  ,draw opacity=1 ]   (221.8,160) -- (156.8,193) ;
\draw [shift={(189.3,176.5)}, rotate = 153.08] [fill={rgb, 255:red, 208; green, 2; blue, 27 }  ,fill opacity=1 ][line width=0.08]  [draw opacity=0] (10.72,-5.15) -- (0,0) -- (10.72,5.15) -- (7.12,0) -- cycle    ;
\draw [color={rgb, 255:red, 208; green, 2; blue, 27 }  ,draw opacity=1 ]   (157.8,193) -- (154.8,248) ;
\draw [shift={(156.3,220.5)}, rotate = 93.12] [fill={rgb, 255:red, 208; green, 2; blue, 27 }  ,fill opacity=1 ][line width=0.08]  [draw opacity=0] (10.72,-5.15) -- (0,0) -- (10.72,5.15) -- (7.12,0) -- cycle    ;
\draw [color={rgb, 255:red, 208; green, 2; blue, 27 }  ,draw opacity=1 ]   (154.8,248) .. controls (189.8,216) and (210.8,203) .. (221.8,160) ;
\draw [shift={(196.3,209.26)}, rotate = 132.87] [fill={rgb, 255:red, 208; green, 2; blue, 27 }  ,fill opacity=1 ][line width=0.08]  [draw opacity=0] (10.72,-5.15) -- (0,0) -- (10.72,5.15) -- (7.12,0) -- cycle    ;

\draw (117,200) node [anchor=north west][inner sep=0.75pt]   [align=left] {$\calC$};
\draw (40,108) node [anchor=north west][inner sep=0.75pt]   [align=left] {$\calW(\calX, \calD)$};
\draw (203,139) node [anchor=north west][inner sep=0.75pt]   [align=left] {$P$};
\draw (140,160) node [anchor=north west][inner sep=0.75pt]   [align=left] {\hlm{$\rIPr P$}};
\draw (130,240) node [anchor=north west][inner sep=0.75pt]   [align=left] {$\bar{P}$};
\end{tikzpicture}

\caption{Pythagorean inequalities by $\IPr$-projection of $P$ onto m-convex (left), and $\rIPr$-projection of $P$ onto e-convex (right).}
\label{figure:pythagorean-inequalities}
\end{figure}
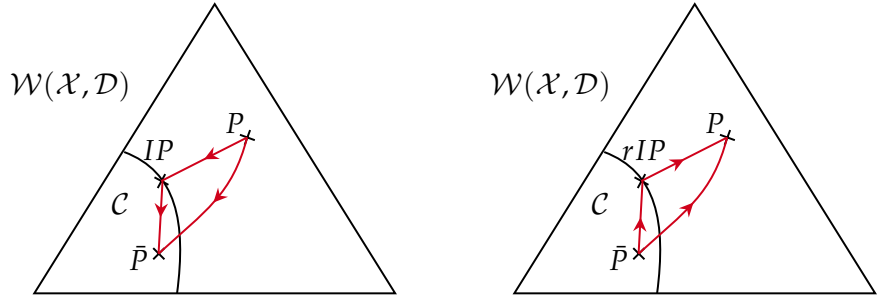

Let now $\calC \subset \calW(\calX, \calD)$ be non-empty, closed and e-convex.
In the distribution setting, a log-convexity counterpart of Proposition~\ref{proposition:pythagorean-inequality-m-convex} is known to hold \citep[Theorem~1]{csiszar2003information}.
For fixed $P \in \calW(\calX, \calD)$, adapting the terminology therein to the Markovian setting,
we define the \emph{$\rIPr$-projection} (reverse information projection) onto $\calC$ as the mapping
\begin{equation*}
\begin{split}
\hlm{\rIPr^{(\calC)}} \colon \calW(\calX, \calD) \to \calC, \qquad
P \mapsto \argmin_{\bar{P} \in \calC} \kl{P}{\bar{P}}.
\end{split}
\end{equation*}
For a fixed $P \in \calW(\calX, \calD)$, the function $\bar{P} \mapsto \kl{P}{\bar{P}}$ is continuous and strictly e-convex.
We show the following counterpart to Proposition~\ref{proposition:pythagorean-inequality-m-convex}.

\begin{proposition}[Pythagorean inequality -- $rI$-projection onto e-convex]
\label{proposition:pythagorean-inequality-e-convex}
Let $P \in \calW(\calX, \calD)$, and
let $\calC \subset \calW(\calX, \calD')$ with $\calD \subset \calD'$, be non-empty, closed and e-convex (Definition~\ref{definition:geodesically-convex-family}).
Then 
$\hlm{\rIPr^{(\calC)}P}$ exists in the sense where the minimum is attained for a unique element in $\calC$. Let $P_0 \in \calC$, the following two statements are equivalent:
\begin{enumerate}
    \item[$(i)$] For any $\bar{P} \in \calC$, $\kl{P}{\bar{P}} \geq \kl{P}{P_0} + \kl{ P_0}{\bar{P}}$. See Figure~\ref{figure:pythagorean-inequalities} (right).
\item[$(ii)$] $P_0 = \hlm{\rIPr^{(\calC)} P}$.
\end{enumerate}
\end{proposition}

\begin{proof}
See Section~\ref{proof:proposition-pythagorean-inequality-e-convex}.
\end{proof}

\begin{remark}
When $\calC$ forms an m-family (resp. e-family), the Pythagorean inequality in Proposition~\ref{proposition:pythagorean-inequality-m-convex} (resp. Proposition~\ref{proposition:pythagorean-inequality-e-convex}) becomes an equality \citep[Corollary~4.7, Corollary~4.8]{hayashi2014information}. When $\calC$ is m-convex (resp. e-convex), the $\IPr$-projection (resp. $\rIPr$-projection) is commonly referred to as the e-projection (resp. m-projection) \citep{amari2007methods}. This can be understood by verifying via the Pythagorean inequality that for an m-convex $\calC$ and $P_0 = \hlm{\IPr P}$, letting $P_t = \stoch(P^{\circ 1 - t} \circ P_0^{\circ t})$, leads to $\hlm{\IPr P_t} = P_0$ for any $0 \leq t \leq 1$. Similarly, for an e-convex $\calC$, and $\hlm{\rIPr P} = P_0$, letting $Q_t = (1 - t) P + t P_0$ yields $\hlm{\rIPr P_t} = P_0$ for any $0 \leq t \leq 1$.
\end{remark}

\subsection{Information projection as a form of data processing}

A natural question is whether taking the $\rIPr$-projection onto an m-convex set also yields practical inequalities.
This is the case for example in the context of distributions, where the following \emph{four-point property} is known to hold with respect to the informational divergence.
\begin{proposition}[Four-point property, distribution setting {(\citet{csiszar1984information},
\citet[Theorem~3.4]{csiszar2004information}, \citet[Theorem~8]{adamvcik2014information})}]
\label{proposition:four-point-property-distribution}
Let $\calC \subset \calP(\calX)$ be non-empty, closed and convex. For any $\mu, \mu' \in \calP(\calX)$ and $\bar{\mu} \in \calC$. When $\mu_0 = \hlm{\rIPr^{(\calC)} \mu}$, 
it holds that
\begin{equation*}
    \kl{\mu'}{\mu_0} \leq \kl{\mu' }{\bar{\mu}} + \kl{\mu'}{\mu}.
\end{equation*}
\end{proposition}
We first wish to extend the aforementioned property to the geometry of \hl{stochastic matrices}.
We let $\calC \subset \calW(\calX, \calD)$ be non-empty, closed and m-convex, and
for any $P, P' \in \calW(\calX, \calD)$ and $\bar{P} \in \mathcal{C}$,
we will say that the four-point property holds for the
quadruple $(P, P', \calC, \bar{P})$ whenever 
it holds that
\begin{equation}
\label{eq:four-point-property}
    \kl{P'}{P_0} \leq \kl{P' }{\bar{P}} + \kl{P'}{P},
\end{equation}
with $P_0 = \hlm{\rIPr^{(\calC)} P}$.

\begin{example}
When $\calC \subset \calW_{\bis}(\calX, \calD)$, $P, P' \in \calW_{\bis}(\calX, \calD)$, the
four-point property always holds.
\end{example}

\begin{example}
\label{example:divergence-inequality-multiplicative-additive-reversiblizations}
We derive an inequality for divergences involving two irreducible chains and their respective additive and multiplicative reversiblizations (Figure~\ref{figure:four-point-property}, left).
Let $\calC = \calW \rev(\calX, \calD)$, which is known to form an em-family (both an e-family and an m-family) \citep{wolfer2021information}.
Consider $P_+ = \frac{P + P^\star }{2}$the m-projection (i.e. the $\rIPr$-projection) of $P$ onto $\calW \rev$, and $P'_\times = P'(P')^\star$ the multiplicative reversiblization of $P'$ (note that $P'_\times$ is not the $\IPr$-projection of $P'$).
Writing $\pi, \pi', \pi'_\times, \pi_+$ for the stationary distributions of the chains under consideration,
we note that by construction, $\pi = \pi_+$, $\pi' = \pi'_\times$, while $\pi$ and $\pi'$ generally need not be equal.
Straightforward calculations yield
\begin{equation*}
\kl{P'}{P_+} \leq \kl{P'}{P'_\times} + \kl{P' }{ P}.
\end{equation*}
\end{example}
We now show that under the four-point property, the operation of taking m-projections onto a doubly autoparallel submanifold of \hl{stochastic matrices} can only bring \hl{stochastic matrices} closer together (Figure~\ref{figure:four-point-property}, right).
\begin{proposition}[Contraction property for m-projection onto em-family]
\label{proposition:data-processing-by-projecting}
Let $\calV_{em}$ be an em-family in $\calW(\calX, \calD)$ (doubly autoparallel submanifold),  $P, P' \in \calW(\calX, \calD)$, and  let $P_m$ (resp. $P'_m$) be the m-projection of $P$ (resp. $P'$) onto $\calV_{em}$.
Suppose that the four-point property holds for the quadruple $(P, P', \calV_{em}, P'_m)$. Then it holds that
$$\kl{P_m'}{P_m} \leq \kl{P'}{P}.$$
\end{proposition}
\begin{proof}
Since $P'_m$ is the m-projection onto the e-family $\calV_{em}$, the Pythagorean identity yields
\begin{equation*}
    \kl{P'}{P_m} = \kl{P'}{P'_m} + \kl{P'_m}{P_m}.
\end{equation*}
Combining with \eqref{eq:four-point-property} yields the claim.
\end{proof}
\begin{figure}[ht!]
\centering

\tikzset{every picture/.style={line width=0.75pt}} %

\begin{tikzpicture}[x=0.60pt,y=0.60pt,yscale=-1,xscale=1]

\draw   (175.84,151) -- (352.8,151) -- (276.96,242.03) -- (100,242.03) -- cycle ;
\draw   (204.76,75.76) -- (213.24,84.24)(213.24,75.76) -- (204.76,84.24) ;
\draw   (177.76,185.76) -- (186.24,194.24)(186.24,185.76) -- (177.76,194.24) ;
\draw   (285.76,101.76) -- (294.24,110.24)(294.24,101.76) -- (285.76,110.24) ;
\draw   (225.76,208.76) -- (234.24,217.24)(234.24,208.76) -- (225.76,217.24) ;
\draw  [dash pattern={on 0.84pt off 2.51pt}]  (182.8,184) .. controls (186.8,146) and (196.8,108) .. (205.8,86) ;
\draw [shift={(191.21,134.16)}, rotate = 102.54] [fill={rgb, 255:red, 0; green, 0; blue, 0 }  ][line width=0.08]  [draw opacity=0] (10.72,-5.15) -- (0,0) -- (10.72,5.15) -- (7.12,0) -- cycle    ;
\draw  [color={rgb, 255:red, 0; green, 0; blue, 0 }  ,draw opacity=1 ] (176,178) -- (176,188) -- (166,188) ;
\draw [color={rgb, 255:red, 208; green, 2; blue, 27 }  ,draw opacity=1 ]   (185.8,184) .. controls (203.8,148) and (244.8,122) .. (283.8,108) ;
\draw [shift={(227.97,136.99)}, rotate = 143.22] [fill={rgb, 255:red, 208; green, 2; blue, 27 }  ,fill opacity=1 ][line width=0.08]  [draw opacity=0] (10.72,-5.15) -- (0,0) -- (10.72,5.15) -- (7.12,0) -- cycle    ;
\draw [color={rgb, 255:red, 208; green, 2; blue, 27 }  ,draw opacity=1 ]   (216,80) .. controls (247.8,76) and (269.8,86) .. (283.8,100) ;
\draw [shift={(252.44,82)}, rotate = 192.99] [fill={rgb, 255:red, 208; green, 2; blue, 27 }  ,fill opacity=1 ][line width=0.08]  [draw opacity=0] (10.72,-5.15) -- (0,0) -- (10.72,5.15) -- (7.12,0) -- cycle    ;
\draw [color={rgb, 255:red, 208; green, 2; blue, 27 }  ,draw opacity=1 ]   (230.8,207) .. controls (234.6,186) and (255.8,145) .. (283.8,112) ;
\draw [shift={(252.29,156.78)}, rotate = 119.64] [fill={rgb, 255:red, 208; green, 2; blue, 27 }  ,fill opacity=1 ][line width=0.08]  [draw opacity=0] (10.72,-5.15) -- (0,0) -- (10.72,5.15) -- (7.12,0) -- cycle    ;

\draw (188,70) node [anchor=north west][inner sep=0.75pt]   [align=left] {$P$};
\draw (299,88) node [anchor=north west][inner sep=0.75pt]   [align=left] {$P'$};
\draw (161,196) node [anchor=north west][inner sep=0.75pt]   [align=left] {$P_+$};
\draw (123,218) node [anchor=north west][inner sep=0.75pt]   [align=left] {$\calW \rev(\calY, \calE)$};
\draw (239,209) node [anchor=north west][inner sep=0.75pt]   [align=left] {$P'_\times$};

\end{tikzpicture}
\hspace{1cm}
\begin{tikzpicture}[x=0.60pt,y=0.60pt,yscale=-1,xscale=1]

\draw   (175.84,151) -- (352.8,151) -- (276.96,242.03) -- (100,242.03) -- cycle ;
\draw   (219.76,48.76) -- (228.24,57.24)(228.24,48.76) -- (219.76,57.24) ;
\draw   (177.76,185.76) -- (186.24,194.24)(186.24,185.76) -- (177.76,194.24) ;
\draw   (294.76,116.76) -- (303.24,125.24)(303.24,116.76) -- (294.76,125.24) ;
\draw   (268.76,189.76) -- (277.24,198.24)(277.24,189.76) -- (268.76,198.24) ;
\draw  [dash pattern={on 0.84pt off 2.51pt}]  (182.8,184) .. controls (186.8,146) and (210.8,82) .. (219.8,60) ;
\draw [shift={(197.83,120.96)}, rotate = 107.32] [fill={rgb, 255:red, 0; green, 0; blue, 0 }  ][line width=0.08]  [draw opacity=0] (10.72,-5.15) -- (0,0) -- (10.72,5.15) -- (7.12,0) -- cycle    ;
\draw  [color={rgb, 255:red, 0; green, 0; blue, 0 }  ,draw opacity=1 ] (176,178) -- (176,188) -- (166,188) ;
\draw [color={rgb, 255:red, 208; green, 2; blue, 27 }  ,draw opacity=1 ]   (185.8,187) .. controls (203.8,179) and (246.8,180) .. (266,189) ;
\draw [shift={(226.37,181.65)}, rotate = 180.68] [fill={rgb, 255:red, 208; green, 2; blue, 27 }  ,fill opacity=1 ][line width=0.08]  [draw opacity=0] (10.72,-5.15) -- (0,0) -- (10.72,5.15) -- (7.12,0) -- cycle    ;
\draw [color={rgb, 255:red, 208; green, 2; blue, 27 }  ,draw opacity=1 ]   (229.8,54) .. controls (256.8,62) and (286.8,95) .. (296.8,114) ;
\draw [shift={(268.21,78.49)}, rotate = 222.06] [fill={rgb, 255:red, 208; green, 2; blue, 27 }  ,fill opacity=1 ][line width=0.08]  [draw opacity=0] (10.72,-5.15) -- (0,0) -- (10.72,5.15) -- (7.12,0) -- cycle    ;
\draw [color={rgb, 255:red, 0; green, 0; blue, 0 }  ,draw opacity=1 ] [dash pattern={on 0.84pt off 2.51pt}]  (273.8,186) .. controls (276.8,166) and (281.8,148) .. (294.8,128) ;
\draw [shift={(281.05,155.57)}, rotate = 108.23] [fill={rgb, 255:red, 0; green, 0; blue, 0 }  ,fill opacity=1 ][line width=0.08]  [draw opacity=0] (10.72,-5.15) -- (0,0) -- (10.72,5.15) -- (7.12,0) -- cycle    ;
\draw  [color={rgb, 255:red, 0; green, 0; blue, 0 }  ,draw opacity=1 ] (289,194) -- (279,194) -- (279,184) ;

\draw (225,60) node [anchor=north west][inner sep=0.75pt]   [align=left] {$P$};
\draw (306,113) node [anchor=north west][inner sep=0.75pt]   [align=left] {$P'$};
\draw (161,196) node [anchor=north west][inner sep=0.75pt]   [align=left] {$P_m$};
\draw (123,218) node [anchor=north west][inner sep=0.75pt]   [align=left] {$\calV_{em}$};
\draw (271,199) node [anchor=north west][inner sep=0.75pt]   [align=left] {$P_m'$};

\end{tikzpicture}
\caption{
Illustration of the four-point property for
multiplicative and additive reversiblization (left). Interpretation of m-projection as a form of data-processing (right).}
\label{figure:four-point-property}
\end{figure}
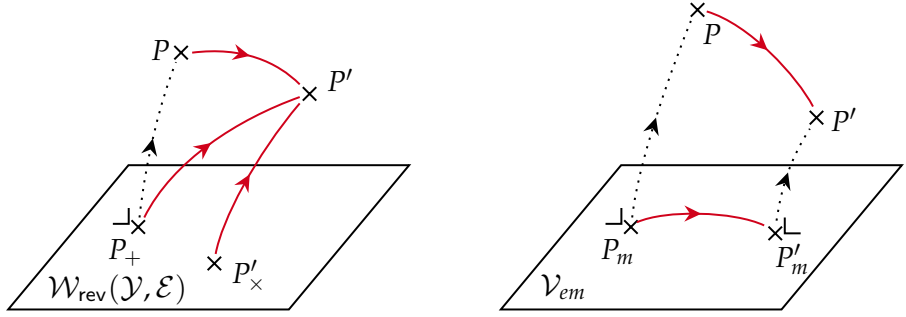
This contractive property is similar to the one enjoyed by nearest-point projections onto convex sets in Hilbert spaces.
We now briefly see how 
m-projecting
can be interpreted as a form of data-processing.
Let $\Lambda_\star$ be a memoryless Markov embedding, and define
\begin{equation*}
    \calJ \eqdef \set{ \Lambda_\star \bar{P} \colon \bar{P} \in \calW(\calX, \calD) }.
\end{equation*}
Since $\calW(\calX, \calD)$ forms an em-family, and $\Lambda_\star$ is e-geodesically and m-geodesically affine (Theorem~\ref{theorem:exponential-embedding-geodesic-affinity}, Lemma~\ref{lemma:memoryless-embedding-m-geodesic-affine}),
$\calJ$ also forms an em-family.
Let $P, P' \in \calW_\kappa(\calY , \calE)$,
and $P_m ,P'_m$ the m-projections of $P, P'$ onto $\calJ$.
We henceforth suppose that the four-point property holds in our context for the quadruple $(P,P', \calJ, P'_m)$. In this case, the m-projections of $P$ and $P'$ are readily obtained by composition of lumping and embedding by $\Lambda_\star$ (Lemma~\ref{lemma:m-projection-as-lumping-embedding}).
By Proposition~\ref{proposition:data-processing-by-projecting}, we then recover the data-processing inequality
\begin{equation*}
    \kl{P}{P'} \geq \kl{\Lambda_\star \kappa_\star P}{\Lambda_\star \kappa_\star P'} = \kl{\kappa_\star P}{\kappa_\star P'}.
\end{equation*}
\begin{lemma}
\label{lemma:m-projection-as-lumping-embedding}
\begin{equation*}
\begin{split}
    P_m &\eqdef  \argmin_{\tilde{P} \in \calJ} \kl{P}{\tilde{P}} = \Lambda_\star \kappa_\star P, \\
    P'_m &\eqdef  \argmin_{\tilde{P} \in \calJ} \kl{P'}{\tilde{P}} = \Lambda_\star \kappa_\star P'.
\end{split}
\end{equation*}
\end{lemma}
\begin{proof}
Let $\tilde{P} \in \calJ$.
\begin{equation*}
\begin{split}
    \kl{P}{\tilde{P}} &= \sum_{y,y' \in \calE} Q(y,y') \log \frac{P(y,y')}{\tilde{P}(y,y')} \\ 
    &= \sum_{y,y' \in \calE} Q(y,y') \log \frac{P(y,y')}{\Lambda_\star \kappa_\star P(y,y')} + \sum_{y,y' \in \calE} Q(y,y') \log \frac{\Lambda_\star \kappa_\star P(y,y')}{\tilde{P}(y,y')} \\
    &= \kl{P}{\Lambda_\star \kappa_\star P} + \sum_{y,y' \in \calE} Q(y,y') \log \frac{\kappa_\star P(\kappa(y),\kappa(y')) \Lambda(y,y')}{\kappa_\star \tilde{P}(\kappa(y),\kappa(y')) \Lambda(y,y')} \\ 
    &= \kl{P}{\Lambda_\star \kappa_\star P} + \sum_{x,x' \in \calD} \left( \sum_{y \in \calS_x, y' \in \calS_{x'}} Q(y,y') \right) \log \frac{\kappa_\star P(\kappa(y),\kappa(y'))}{\kappa_\star \tilde{P}(\kappa(y),\kappa(y'))} \\ 
    &= \kl{P}{\Lambda_\star \kappa_\star P} + \kl{\kappa_\star P}{ \kappa_\star \tilde{P}} \geq \kl{P}{\Lambda_\star \kappa_\star P},\\
\end{split}
\end{equation*}
and the claim holds since $\Lambda_\star \kappa_\star P \in \calJ$.
\end{proof}

\begin{example}
\label{example:additive-reversiblizations-data-processing}
Irreducible \hl{stochastic matrices} can only be brought closer together by taking their additive reversiblization, which corresponds to the
m-projection onto $\calW \rev(\calX, \calD)$.
Indeed, without even relying on the four-point property, we directly prove by joint convexity of the KL divergence in the context of distributions that
\begin{equation*}
\begin{split}
    D((P_0 + P_0^\star)/2 \| (P_1 + P_1^\star)/2) &= D((Q_0 + Q_0^\star)/2 \| (Q_1 + Q_1^\star)/2) - D(\pi_0 \| \pi_1) \\ 
    &\leq \frac{1}{2} D(Q_0 \| Q_1) +  \frac{1}{2} D(Q_0^\star \| Q_1^\star) - D(\pi_0 \| \pi_1) \\
    &= D(Q_0 \| Q_1)   - D(\pi_0 \| \pi_1) = D(P_0 \| P_1).
\end{split}
\end{equation*}
\end{example}

\section{Information geometry of lumpable stochastic matrices}
\label{section:information-geometry-lumpable}
Many important families of \hl{stochastic matrices} (e.g. doubly stochastic matrices, symmetric matrices, reversible matrices,...) are known to enjoy favorable geometrical features  \citep[Table~1]{wolfer2021information}.
In this section we analyze the geometrical structure of the family $\calW_\kappa(\calY, \calE)$ of lumpable \hl{stochastic matrices}.

\hlp{\subsection{The foliated manifold of lumpable stochastic matrices}}
\label{section:foliations}
Recall that in the example at Lemma~\ref{lemma:hudson-breaks-m-structure}, we inspected the nature of the geodesic midpoint $\gamma^{(m)}_{P_0, P_1}(1/2)$ where $P_0, P_1 \in \calW_{h}(\calX, H_{\calX^2})$ with $\calX = \set{0, 1}$ and $h$ is the Hudson lumping. There, we found that this point is not lumpable, i.e. $\gamma^{(m)}_{P_0, P_1}(1/2) \not \in \calW_\kappa(\calX, H_{\calX^2})$, thereby already showing that lumpable \hl{stochastic matrices} do not generally form m-families \footnote{It was incorrectly stated in \citet[Example~3]{hayashi2014information} that a special class of lumpable \hl{stochastic matrices} (therein referred to as ``non-hidden") forms an m-family.}.
Similarly, it is possible to numerically produce pairs of lumpable transition \hl{matrices} over a ternary alphabet, with an e-geodesic passing through them that leaves the manifold of lumpable \hl{stochastic matrices}.
It is then a consequence of \citet[Corollary~3]{nagaoka2017exponential} that lumpable \hl{stochastic matrices} do not form e-families either.

A foliation is the decomposition of a manifold into a union of connected but disjoint submanifolds, called leaves, all sharing the same dimension.
See for example \citet[Chapter~19]{lee2013smooth} for a thorough exposition.
The concept of mutually dual foliations and mixed coordinate systems play a significant role in information geometry \citep[Section~3.7]{amari2007methods}.
Let us fix some origin $\bar{P}_{0} \in \calW(\calX, \calD)$, and let
\begin{equation}
\label{eq:family-L}
    \calL(\bar{P}_{0}) \eqdef \set{ P \in \calW_\kappa(\calY, \calE) \colon \kappa_\star P = \bar{P}_{0}},
\end{equation}
be the submanifold in $\calW_\kappa(\calY, \calE)$ of all \hl{stochastic matrices} that $\kappa$-lump into $\bar{P}_{0}$.
For any $P \in \calL(\bar{P}_{0})$, following Lemma~\ref{lemma:construct-canonical-embedding}, recall that we can construct the canonical embedding $\Lambda^{(P)}_\star$, which verifies $\Lambda^{(P)}_\star \kappa_\star P = P$.
For $P_{\origin} \in \calW_\kappa(\calY, \calE)$, we can then define
\begin{equation}
\label{eq:family-J}
    \calJ(P_{\origin}) \eqdef \set{ \Lambda_\star^{(P_{\origin})}\bar{P} \colon \bar{P} \in \calW(\calX, \calD) } \subset \calW_\kappa(\calY, \calE)
\end{equation}
to be the image of the entire manifold $\calW(\calX, \calD)$ by the embedding $\Lambda^{(P_{\origin})}_\star$. 

\begin{figure}%
\centering

\tikzset{every picture/.style={line width=0.75pt}} %

\begin{tikzpicture}[x=0.40pt,y=0.40pt,yscale=-1,xscale=1]

\draw   (307,181) .. controls (307,128.53) and (365.16,86) .. (436.9,86) .. controls (508.64,86) and (566.8,128.53) .. (566.8,181) .. controls (566.8,233.47) and (508.64,276) .. (436.9,276) .. controls (365.16,276) and (307,233.47) .. (307,181) -- cycle ;
\draw   (50,91) .. controls (50,61.18) and (84.88,37) .. (127.9,37) .. controls (170.92,37) and (205.8,61.18) .. (205.8,91) .. controls (205.8,120.82) and (170.92,145) .. (127.9,145) .. controls (84.88,145) and (50,120.82) .. (50,91) -- cycle ;
\draw [color={rgb, 255:red, 208; green, 2; blue, 27 }  ,draw opacity=1 ][line width=0.08]    (311.8,206) .. controls (337.8,156) and (518.8,151) .. (564.8,197) ;
\draw    (436.8,165) .. controls (476.6,135.15) and (263.94,86.49) .. (156.41,92.9) ;
\draw [shift={(154.8,93)}, rotate = 356.26] [fill={rgb, 255:red, 0; green, 0; blue, 0 }  ][line width=0.08]  [draw opacity=0] (10.72,-5.15) -- (0,0) -- (10.72,5.15) -- (7.12,0) -- cycle    ;
\draw   (145.76,88.76) -- (154.24,97.24)(154.24,88.76) -- (145.76,97.24) ;
\draw [color={rgb, 255:red, 74; green, 144; blue, 226 }  ,draw opacity=1 ][line width=0.08]    (436.9,276) .. controls (338.8,201) and (396.9,116) .. (436.9,86) ;
\draw   (379.76,166.76) -- (388.24,175.24)(388.24,166.76) -- (379.76,175.24) ;
\draw    (156.8,96) .. controls (217.19,124.71) and (299.14,155.38) .. (373.55,167.63) ;
\draw [shift={(375.8,168)}, rotate = 189.09] [fill={rgb, 255:red, 0; green, 0; blue, 0 }  ][line width=0.08]  [draw opacity=0] (10.72,-5.15) -- (0,0) -- (10.72,5.15) -- (7.12,0) -- cycle    ;
\draw    (127.9,145) .. controls (189.18,234.1) and (319.17,247.73) .. (389.68,225.68) ;
\draw [shift={(391.8,225)}, rotate = 161.81] [fill={rgb, 255:red, 0; green, 0; blue, 0 }  ][line width=0.08]  [draw opacity=0] (10.72,-5.15) -- (0,0) -- (10.72,5.15) -- (7.12,0) -- cycle    ;

\draw (66,06) node [anchor=north west][inner sep=0.75pt]   [align=left] {$\calW(\calX, \calD)$};
\draw (447,56) node [anchor=north west][inner sep=0.75pt]   [align=left] {$\calW_\kappa(\calY, \calE)$};
\draw (276,80) node [anchor=north west][inner sep=0.75pt]   [align=left] {$\kappa_\star$};
\draw (115,65) node [anchor=north west][inner sep=0.75pt]   [align=left] {$\bar{P}_{0}$};
\draw (479,140) node [anchor=north west][inner sep=0.75pt]  [color={rgb, 255:red, 208; green, 2; blue, 27 }  ,opacity=1 ] [align=left] {$\calL(\bar{P}_{0})$};
\draw (424,233) node [anchor=north west][inner sep=0.75pt]   [align=left] {\textcolor[rgb]{0.29,0.56,0.89}{$\calJ(P_{\origin})$}};
\draw (393,172) node [anchor=north west][inner sep=0.75pt]   [align=left] {$P_{\origin}$};
\draw (191,136) node [anchor=north west][inner sep=0.75pt]   [align=left] {$\Lambda^{(P_{\origin})}_\star$};
\draw (142,209) node [anchor=north west][inner sep=0.75pt]   [align=left] {$\Lambda^{(P_{\origin})}_\star$};

\end{tikzpicture}

\caption{The submanifolds $\calL(\bar{P}_{0})$ and $\calJ(P_{\origin})$.}
\end{figure}

\begin{lemma}
\label{lemma:L-and-J-are-special-families}
For any $\bar{P}_0 \in \calW(\calX, \calD)$, and $P_{\origin} \in \calW(\calY, \calE)$,
\begin{enumerate}
    \item[$(i)$] $\calJ(P_{\origin})$ forms an e-family in $\calW(\calY, \calE)$, with $\dim \calJ(P_{\origin}) = \abs{\calD} - \abs{\calX}$.
    \item[$(ii)$] $\calL(\bar{P}_0)$ forms an m-family in $\calW(\calY, \calE)$, with $\dim \calL(\bar{P}_0) =
    \abs{\calE} - \sum_{(x,x') \in \calD} \abs{\calS_x}$.
\end{enumerate}
\end{lemma}

\begin{proof}
See Section~\ref{proof:lemma-L-and-J-are-special-families}.
\end{proof}

\begin{remark}
When $\calD = \calX^2, \calE = \calY^2$,
\begin{equation*}
    \dim \calJ(P_{\origin}) = \abs{\calX} (\abs{\calX} - 1), \;\;\; \dim \calL(\bar{P}_{0}) = \abs{\calY} (\abs{\calY} - \abs{\calX}).
\end{equation*}
\end{remark}

\begin{remark}
In the case where $\bar{P}_0$ is reversible, there is a one-to-one correspondence between a reversibility preserving embedding 
$\Lambda^{(P)}$ and $P \in \calL(\bar{P}_0) \cap \calW \rev(\calX)$. In this case, it follows from Lemma~\ref{lemma:reversibility-preserving-embedding-stationary-distribution} that $\calJ(P)$ forms a reversible e-family, as defined in \citet[Section~4]{wolfer2021information}.
\end{remark}

We now prove that the manifold of $\kappa$-lumpable \hl{stochastic matrices} can be foliated, with the collection of submanifolds $\set{\calJ(P)}_{P \in \calL(\bar{P}_0)}$ acting as leaves, for any base point $\bar{P}_0$.
Fixing $\bar{P}_0$, we can then refer to a $\kappa$-lumpable \hl{stochastic matrix} $P$ in two steps. We first specify the leaf it belongs to, i.e. its coordinate along the family $\calL(\bar{P}_0)$, to which corresponds to some lumpable $P_{\origin}$. As a second step, we indicate the coordinates of $P$ in $\calJ(P_{\origin})$.

\begin{theorem}[Foliated structure  of lumpable \hl{stochastic matrices}]
\label{theorem:foliation-of-lumpable-kernels}
For any fixed $\bar{P}_0 \in \calW(\calX, \calD)$,
\begin{equation*}
    \calW_\kappa(\calY, \calE) = \biguplus_{P \in \calL(\bar{P}_0)} \calJ(P).
\end{equation*}
\begin{equation*}
    \dim \calW_\kappa(\calY, \calE) = \abs{\calE} - \sum_{(x,x') \in \calD} \abs{\calS_x} + \abs{\calD} - \abs{\calX}.
\end{equation*}
\end{theorem}

\begin{figure}%
\label{figure:foliations-representation}
\centering

\tikzset{every picture/.style={line width=0.75pt}} %

\begin{tikzpicture}[x=0.65pt,y=0.65pt,yscale=-1,xscale=1]

\draw   (368.7,37) -- (407.8,37) .. controls (403.69,37) and (400.35,82.22) .. (400.35,138) .. controls (400.35,193.78) and (403.69,239) .. (407.8,239) -- (368.7,239) .. controls (364.58,239) and (361.25,193.78) .. (361.25,138) .. controls (361.25,82.22) and (364.58,37) .. (368.7,37) -- cycle ;
\draw   (306.7,37) -- (345.8,37) .. controls (341.69,37) and (338.35,82.22) .. (338.35,138) .. controls (338.35,193.78) and (341.69,239) .. (345.8,239) -- (306.7,239) .. controls (302.58,239) and (299.25,193.78) .. (299.25,138) .. controls (299.25,82.22) and (302.58,37) .. (306.7,37) -- cycle ;
\draw   (243.7,37) -- (282.8,37) .. controls (278.69,37) and (275.35,82.22) .. (275.35,138) .. controls (275.35,193.78) and (278.69,239) .. (282.8,239) -- (243.7,239) .. controls (239.58,239) and (236.25,193.78) .. (236.25,138) .. controls (236.25,82.22) and (239.58,37) .. (243.7,37) -- cycle ;
\draw    (403,61) -- (449.8,61) ;
\draw    (400,130) -- (450.8,130) ;
\draw    (402,200) -- (450.8,200) ;
\draw    (341,60) -- (381.8,60) ;
\draw    (338,130) -- (379.8,130) ;
\draw    (340,200) -- (381.8,200) ;
\draw    (278,60) -- (321.8,60) ;
\draw    (276,130) -- (320.8,130) ;
\draw    (277,200) -- (320.8,200) ;
\draw    (220,60) -- (257.8,60) ;
\draw    (217,130) -- (254.8,130) ;
\draw    (220,200) -- (257.8,200) ;
\draw   (123,148) .. controls (123,74.55) and (216.75,15) .. (332.4,15) .. controls (448.05,15) and (541.8,74.55) .. (541.8,148) .. controls (541.8,221.45) and (448.05,281) .. (332.4,281) .. controls (216.75,281) and (123,221.45) .. (123,148) -- cycle ;

\draw (426,43) node [anchor=north west][inner sep=0.75pt]  [color={rgb, 255:red, 208; green, 2; blue, 27 }  ,opacity=1 ] [align=left] {$\calL(\bar{P}_0)$};
\draw (426,112) node [anchor=north west][inner sep=0.75pt]  [color={rgb, 255:red, 208; green, 2; blue, 27 }  ,opacity=1 ] [align=left] {$\calL(\bar{P}_1)$};
\draw (426,183) node [anchor=north west][inner sep=0.75pt]  [color={rgb, 255:red, 208; green, 2; blue, 27 }  ,opacity=1 ] [align=left] {$\calL(\bar{P}_k)$};
\draw (232,245) node [anchor=north west][inner sep=0.75pt]  [color={rgb, 255:red, 74; green, 144; blue, 226 }  ,opacity=1 ] [align=left] {$\calJ(P_{\origin, 0})$};
\draw (438,143) node [anchor=north west][inner sep=0.75pt]  [color={rgb, 255:red, 208; green, 2; blue, 27 }  ,opacity=1 ] [align=left] {$\vdots$};
\draw (293,245) node [anchor=north west][inner sep=0.75pt]  [color={rgb, 255:red, 74; green, 144; blue, 226 }  ,opacity=1 ] [align=left] {$\calJ(P_{\origin, 1})$};
\draw (349,250) node [anchor=north west][inner sep=0.75pt]  [color={rgb, 255:red, 74; green, 144; blue, 226 }  ,opacity=1 ] [align=left] {$\hdots$};
\draw (365,245) node [anchor=north west][inner sep=0.75pt]  [color={rgb, 255:red, 74; green, 144; blue, 226 }  ,opacity=1 ] [align=left] {$\calJ(P_{\origin, k})$};
\draw (87,59) node [anchor=north west][inner sep=0.75pt]   [align=left] {$\calW_\kappa(\calY, \calE)$};

\end{tikzpicture}

\caption{Mutually dual foliated structures of $\calW_\kappa(\calY, \calE)$.}
\end{figure}

\begin{proof}
For any $P \in \calW_\kappa(\calY, \calE)$, $\Lambda^{(P)}_\star$ is $\kappa$-congruent by construction, thus $\calJ(P) \subset \calW_\kappa(\calY, \calE)$.
To prove the other direction, we first let $P' \in \calW_\kappa(\calY, \calE)$, and construct the canonical $\Lambda^{(P')}_\star$, with
$$\Lambda^{(P')}(y , y') = \frac{P'(y,y')}{\kappa_\star P'(\kappa(y), \kappa(y'))}.$$
We then introduce $P \in \calL(\bar{P}_0)$ such that $P = \Lambda_\star^{(P')} \bar{P}_0$.
$$P(y , y') = \bar{P}_{0}(\kappa(y),\kappa(y')) \Lambda^{(P')}(y,y').$$
Note that in general, $P \neq P'$.
We proceed to construct the corresponding canonical $\Lambda^{(P)}$, and we observe that since $\kappa_\star P = \bar{P}_{0}$, and $P = \Lambda_\star^{(P')}\bar{P}_0$,
\begin{equation*}
    \begin{split}
        \Lambda^{(P)}(y , y') = \frac{P(y,y')}{\kappa_\star P(\kappa(y), \kappa(y'))} = \frac{\bar{P}_{0}(\kappa(y),\kappa(y')) P'(y,y')}{\kappa_\star P(\kappa(y), \kappa(y'))\kappa_\star P'(\kappa(y), \kappa(y'))} = \Lambda^{(P')}(y,y').
    \end{split}
\end{equation*}
It follows that $P' = \Lambda_\star^{(P')} \kappa_\star P' = \Lambda_\star^{(P)} \kappa_\star P'$, and
$P' \in \calJ(P)$.
It remains to prove the
 disjointedness of the leaves.
 We let $P_{\origin}, P'_{\origin} \in \calL(\bar{P}_0)$, and let
 $P \in \calJ(P_{\origin}) \cap \calJ(P'_{\origin})$.
 Since $P \in \calJ(P_{\origin})$, there exists $\bar{P} \in \calW (\calX, \calD)$ such that $P = \Lambda_\star^{(P_{\origin})} \bar{P}$, and similarly, since $P \in \calJ(P'_{\origin})$, there exists $\bar{P}' \in \calW (\calX, \calD)$ such that $P = \Lambda_\star^{(P'_{\origin})} \bar{P}'$.
 It follows that for any $y,y' \in \calE$,
 \begin{equation}
 \label{eq:disjointedness-leaves}
    \bar{P}(\kappa(y), \kappa(y')) \frac{P_{\origin}(y,y')}{\kappa_\star \bar{P}_{\origin}(\kappa(y)\kappa(y'))} = P(y,y') = \bar{P}'(\kappa(y), \kappa(y')) \frac{P'_{\origin}(y,y')}{\kappa_\star \bar{P}'_{\origin}(\kappa(y)\kappa(y'))}.
 \end{equation}
 But since $\Lambda_\star^{(P_{\origin})}$ and $\Lambda_\star^{(P'_{\origin})}$ are $\kappa$-congruent embeddings, 
 $$\kappa_\star P = \kappa_\star \Lambda_\star^{(P'_{\origin})} \bar{P}' = \bar{P}', \qquad \kappa_\star P = \kappa_\star \Lambda_\star^{(P_{\origin})} \bar{P} = \bar{P},$$ thus $\bar{P} = \bar{P}'$. Moreover, $\kappa_\star P_{\origin} = \bar{P}_0 = \kappa_\star P'_{\origin}$, and from \eqref{eq:disjointedness-leaves},
  $P_{\origin} = P'_{\origin}$, whence the foliation structure.
The dimension of $\calW_\kappa(\calY,\calE)$ is then readily obtained by summing the dimensions of $\calL(\bar{P}_0)$ and $\calJ(P)$ for $P \in \calL(\bar{P}_0)$, which are both given in Lemma~\ref{lemma:L-and-J-are-special-families}.
\end{proof}

\begin{remark}
We can intuitively relate the dimension of $\calW_\kappa(\calY, \calE)$ as a manifold to that of
the vector space $\calF_\kappa(\calY, \calE)$ (Lemma~\ref{lemma:lumpable-matrix-vector-space}). Indeed, setting $\abs{\calX}$ additional constraints to
a lumpable matrix ensures it is row-stochastic:
$$\dim \calW_\kappa(\calY, \calE) = \dim \calF_\kappa(\calY, \calE) - \abs{\calX}.$$
\end{remark}

\subsection{Interpretations \& Applications}
\label{section:foliation-interpretations-applications}

In this section, we first illustrate how $\calL(\bar{P}_0)$ forming an m-family in $\calW(\calY, \calE)$ enables us to efficiently select a chain on a finer state space that lumps into $\bar{P}_0$, while making the fewest additional assumptions (Section~\ref{section:max-entropy}). We then proceed to show how the foliation
 introduced in  Theorem~\ref{theorem:foliation-of-lumpable-kernels}, leads to an orthogonality result for projections onto leaves (Section~\ref{section:leaf-projection}). 
Finally, we offer an interpretation of the foliated structure in the context of maximum likelihood estimation of embedded models (Section~\ref{section:embedded-models}).

\subsubsection{The e-projection \& maximum entropy principle}
\label{section:max-entropy}
Fix $\bar{P}_0 \in \calW(\calX, \calD)$, and recall that although 
$\calW_\kappa(\calY, \calE)$ is not even m-convex,
$\calL(\bar{P}_0)$ forms an m-family in $\calW(\calY, \calE)$ of dimension $d = \abs{\calE} - \sum_{(x,x') \in \calD} \abs{\calS_x}$ (Lemma~\ref{lemma:L-and-J-are-special-families}).
As a consequence, there exist $g_1, \dots, g_d \in \calF(\calY, \calE)$ and $c= (c_1, \dots, c_d) \in \R^d$ such that we can express $\calL(\bar{P}_0)$ as the polytope generated by the set of linear constraints $\set{g_i = c_i}$,
\begin{equation*}
    \calL(\bar{P}_0) = \set{ P \in \calW(\calY, \calE) \colon \sum_{(y,y') \in \calE} Q(y,y') g_i(y,y') = c_i, \forall i \in [d]} \subset \calW_\kappa(\calY, \calE).
\end{equation*}
We let $P \in \calW(\calY, \calE)$, and take interest into its orthogonal e-projection onto $\calL(\bar{P}_0)$.
\begin{equation*}
    P_e \eqdef \argmin_{P' \in \calL(\bar{P}_0)} \kl{P'}{P}.
\end{equation*}
It is well-known that the solution to this minimization problem belongs to an exponential family \citep{csiszar1984sanov, csiszar1987conditional}.
In fact, introducing for $(y,y') \in \calE$ and $\lambda \in \R^d$,
\begin{equation*}
    \tilde{P}_\lambda(y,y') = \log P(y,y') - \sum_{i \in [d]} \lambda^i g_i(y,y'),
\end{equation*}
and denoting $\psi(\lambda)$ for the logarithm of the PF root of $\tilde{P}_\lambda$,
the minimizer is given by $P_e = \stoch(\tilde{P}_{\lambda_\star})$ with 
\begin{equation*}
\begin{split}\lambda^\star &= \argmax_{\lambda \in \R^d} \set{  \lambda \cdot c - \psi(\lambda) }. \\
\end{split}
\end{equation*}
Furthermore, recall the following expression for the \emph{entropy rate} $H$ of a Markov chain,
\begin{equation*}
    H(P) \eqdef \lim_{k \to \infty} \frac{1}{k} H(Y_1, Y_2, \dots, Y_k) = - \sum_{(y,y') \in \calE} Q(y,y') \log P(y,y'),
\end{equation*}
and rewrite
\begin{equation*}
\begin{split}
    \argmin_{P' \in \calL(\bar{P}_0)} \kl{P'}{P} &= \argmax_{P' \in \calL(\bar{P}_0)} \set{ H(P') + \E[(Y,Y') \sim Q']{\log P(Y,Y')} }.
\end{split}
\end{equation*}
Suppose we wish to embed some prescribed \hl{stochastic matrix} $\bar{P}_0$ into a larger state space. Following the \emph{maximum entropy principle}, we should choose the embedded \hl{stochastic matrix} with the largest entropy rate which follows the constraint of belonging to $\calL(\bar{P}_0)$.
Let us define $\delta_\calE \in \calF_+(\calY ,\calE)$ as $\delta_\calE(y,y') = 1$, and introduce the rescaled \hl{stochastic matrix}
$$\maxentmc = \stoch(\delta_\calE) = \frac{1}{\rho_{\maxentmc}} \diag(v_{\maxentmc})^{-1} \delta_\calE \diag (v_{\maxentmc}).$$
It is easy to verify that
$H(\maxentmc) = \log \rho_{\maxentmc}$,
and it is known that $\maxentmc$ is the \emph{maxentropic} chain
\citep{spitzer1972variational, justesen1984maxentropic, duda2007optimal, burda2009localization}
subject solely to graph constraints. In particular, $\maxentmc$ gives uniform probability among paths of a given length linking two prescribed points.
When $P = \maxentmc$, $ \E[(Y,Y') \sim Q']{\log P(Y,Y')} = \log \rho_{\maxentmc}$, which is a property of $\calE$ only. As a result,
\begin{equation*}
\begin{split}
    \argmin_{P' \in \calL(\bar{P}_0)} \kl{P'}{\maxentmc} &= \argmax_{P' \in \calL(\bar{P}_0)} H(P'),
\end{split}
\end{equation*}
i.e. the e-projection of $\maxentmc$ onto $\calL(\bar{P}_0)$ corresponds to the \hl{stochastic matrix} that lumps into $\bar{P}_0$ with maximum entropy rate.

\subsubsection{Pythagorean projection on a leaf}
\label{section:leaf-projection}

\begin{figure}[ht!]

\centering
\tikzset{every picture/.style={line width=0.75pt}} %

\begin{tikzpicture}[x=0.75pt,y=0.75pt,yscale=-1,xscale=1]

\draw   (340,69.39) -- (340,226.63) -- (240,211.23) -- (240,53.99) -- cycle ;
\draw    (69,110) -- (280,110) ;
\draw  [dash pattern={on 4.5pt off 4.5pt}]  (280,110) -- (340,110) ;
\draw    (340,110) -- (421,110) ;
\draw   (285.76,105.76) -- (294.24,114.24)(294.24,105.76) -- (285.76,114.24) ;
\draw   (156.76,105.76) -- (165.24,114.24)(165.24,105.76) -- (156.76,114.24) ;
\draw   (285.76,165.76) -- (294.24,174.24)(294.24,165.76) -- (285.76,174.24) ;
\draw    (290,110) -- (290,170.67) ;
\draw  [color={rgb, 255:red, 208; green, 2; blue, 27 }  ,draw opacity=1 ] (275,115) -- (285,115) -- (285,125) ;

\draw (273,31) node [anchor=north west][inner sep=0.75pt]   [align=left] {$\calJ(P_{\origin})$};
\draw (277,81) node [anchor=north west][inner sep=0.75pt]   [align=left] {$P_{\origin}$};
\draw (88,77) node [anchor=north west][inner sep=0.75pt]   [align=left] {$\calL(\bar{P}_0)$};
\draw (154,120) node [anchor=north west][inner sep=0.75pt]   [align=left] {$P$};
\draw (303,163) node [anchor=north west][inner sep=0.75pt]   [align=left] {$P'$};
\end{tikzpicture}
\caption{Pythagorean projection on a leaf.}
\label{figure:leaf-projection}
\end{figure}
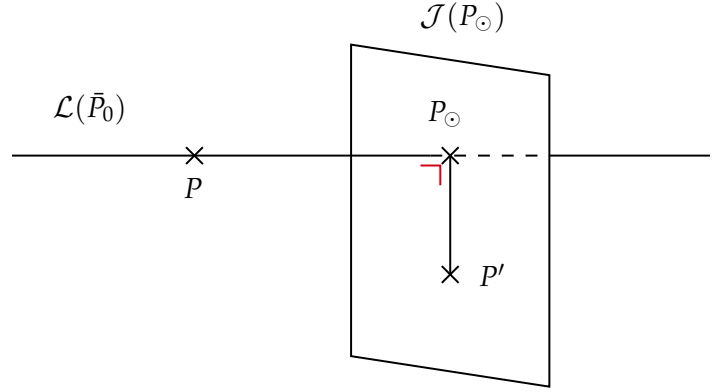
Let us project a lumpable transition \hl{stochastic matrix} onto a leaf, as represented in Figure~\ref{figure:leaf-projection}.

\begin{theorem}
\label{theorem:foliation-pythagorean-theorem}
Fix $\bar{P}_0 \in \calW(\calX, \calD)$. 
Let $P_{\origin}, P \in \calL(\bar{P}_0)$ and $ P' \in \calJ(P_{\origin})$ be arbitrarily chosen.
The following Pythagorean identity holds,
\begin{equation*}
    \kl{P}{P'} = \kl{P}{P_{\origin}} + \kl{P_{\origin}}{P'},
\end{equation*}
and $P_{\origin}$ verifies
\begin{equation*}
    P_{\origin} = \argmin_{P'' \in \calL(\bar{P}_0)} \kl{P''}{P'} = \argmin_{P'' \in \calJ(P_{\origin})} \kl{P}{P''}.
\end{equation*}
\end{theorem}
\begin{proof}
Writing $Q, Q_{\origin}$ for the edge measures of $P, P_{\origin}$, we compute
\begin{equation*}
    \begin{split}
        \kl{P}{P'} -\kl{P}{P_{\origin}} - \kl{P_{\origin}}{P'} = \sum_{(y,y') \in \calE} \left[ Q(y,y') - Q_{\origin}(y,y') \right] \log \frac{P_{\origin}(y,y')}{P'(y,y')}.
    \end{split}
\end{equation*}
Since $P' \in \calJ(P_{\origin})$, there exists $\bar{P}' \in \calW(\calX, \calD)$ such that $P' = \Lambda_\star^{(P_{\origin})} \bar{P}'$. Similarly, since trivially $P_{\origin} \in \calJ(P_{\origin})$, we can also write $P_{\origin} = \Lambda^{(P_{\origin})}_\star \bar{P}_{\origin}$ for some $\bar{P}_{\origin} \in \calW(\calX, \calD)$.
We obtain,
\begin{equation*}
    \begin{split}
         \kl{P}{P'} -\kl{P}{P_{\origin}} &- \kl{P_{\origin}}{P'} = \sum_{(y,y') \in \calE} \left[ Q(y,y') - Q_{\origin}(y,y') \right] \log \frac{\bar{P}_{\origin}(\kappa(y), \kappa(y'))}{ \bar{P}'(\kappa(y),\kappa(y'))} \\
        &= \sum_{(x,x') \in \calD} \log \frac{\bar{P}_{\origin}(x, x')}{ \bar{P}'(x,x')} \sum_{y \in\calS_x, y' \in\calS_{x'}} \left[ Q(y,y') - Q_{\origin}(y,y') \right]  \\
        &= \sum_{(x,x') \in \calD} \log \frac{\bar{P}_{\origin}(x, x')}{ \bar{P}'(x,x')}  \left[ \bar{Q}(x,x') - \bar{Q}_{\origin}(x,x') \right],\\
    \end{split}
\end{equation*}
where the last equality stems from Corollary~\ref{corollary:lumped-stationary-distribution-and-edge-measure}, $P$ an $P_{\origin}$ being $\kappa$-lumpable, and where we wrote $\bar{Q}, \bar{Q}_{\origin}$ for the respective edge measures of $\kappa_\star P$ and $\kappa_\star P_{\origin}$.
But by assumption, $P, P_{\origin} \in \calL(\bar{P}_0)$, thus $\bar{Q} = \bar{Q}_{\origin}$ and the identity holds.
\end{proof}

\subsubsection{The m-projection \& maximum likelihood estimation for embedded models}
\label{section:embedded-models}
The foliation in Theorem~\ref{theorem:foliation-of-lumpable-kernels} has a natural interpretation in the context of the lumpable \hl{stochastic matrix} estimation problem. For $k \in \N$, we observe a Markov chain
\begin{equation*}
    Y_1, Y_2, \dots, Y_k,
\end{equation*}
drawn with respect to some unknown $\kappa$-lumpable \hl{stochastic matrix} $P$ and initial distribution $\mu$. 
We can express the likelihood of a trajectory $y_1, y_2, \dots, y_k \in \calY$ as
\begin{equation*}
    \PR{Y_1 = y_1, Y_2 = y_2, \dots, Y_k = y_k} = \mu(y_1) \prod_{t = 1}^{k-1} P(y_t, y_{t+1}) = \mu(y_1) \prod_{e \in \calE} P(e)^{(k-1) T(e)},
\end{equation*}
where $T(e) = \frac{1}{k-1}\sum_{t = 1}^{k-1} \pred{(y_t, y_{t+1}) = e}$, and $\set{T(e)}_{e \in \calE}$ is called a Markov type.
We consider a base e-family $\calV_e = \set{\bar{P}_\theta \colon \theta \in \Theta} \subset \calW(\calX, \calD)$ following Definition~\ref{definition:e-family-parametric},
\begin{equation*}
    \bar{P}_\theta(x,x') = \exp \left(C(x,x') + \sum_{i = 1}^{d} \theta^i \bar{g}_i(x,x') + R_\theta(x') - R_\theta(x) - \psi_\theta \right),
\end{equation*}
where $\bar{g}_1, \dots, \bar{g}_d \in \calF(\calX, \calD)$ are independent in $\calG(\calX, \calD)$.
We then let $P_{\origin} \in \calW(\calY, \calE)$ and look at the lumpable embedded model defined as in \eqref{eq:family-J},
\begin{equation*}
    \calJ(P_{\origin}) = \set{ P_{\origin, \theta} \eqdef \Lambda^{(P_{\origin})}_\star \bar{P}_\theta \colon  \bar{P}_\theta \in \calV_e }.
\end{equation*}
From Lemma~\ref{lemma:L-and-J-are-special-families}, $\calJ(P_{\origin})$ forms an e-family with $\dim \calJ(P_{\origin}) =  \dim \calV_e$, but that consists of chains over a larger state space.
Furthermore, a basis for $\calJ(P_{\origin})$ is given by the below stated Lemma~\ref{lemma:lumped-family-of-functions}.
\begin{lemma}
\label{lemma:lumped-family-of-functions}
Let $d \in \N$, and let a collection of functions $\bar{g}_1, \dots, \bar{g}_d \in \calF(\calX, \calD)$ that is independent in $\calG(\calX, \calD)$. Then the collection
$g_1, \dots, g_d \in \calF(\calY, \calE)$ defined for $i \in [d]$ by
\begin{equation*}
\begin{split}
    g_i \colon \calE &\to \R, (y,y') \mapsto \bar{g}_i(\kappa(y), \kappa(y')),
\end{split}
\end{equation*}
is independent in $\calG(\calY, \calE)$.
\end{lemma}
\begin{proof}
Let $\alpha_1, \dots, \alpha_d \in \R$ be such that $\sum_{i = 1}^{d} \alpha_i g_i = 0_{\calG(\calY, \calE)}$,
then there exist $f \in \R^{\calY}$ and $c \in \R$ where for any $(y,y') \in \calE$,
\begin{equation*}
    \sum_{i = 1}^{d} \alpha_i g_i(y,y') = f(y') - f(y) + c.
\end{equation*}
For $(x,x') \in \calD$, taking an average on both sides yields
\begin{equation*}
\begin{split}
    \frac{1}{\abs{\calS_x}\abs{\calS_x'}}\sum_{y \in \calS_{x}, y' \in \calS_{x'}}  \sum_{i = 1}^{d} \alpha_i g_i(y,y') &= \frac{1}{\abs{\calS_x}\abs{\calS_x'}} \sum_{y \in \calS_{x}, y' \in \calS_{x'}}  \left(f(y') - f(y) + c\right) \\
     \sum_{i = 1}^{d} \alpha_i \bar{g}_i(x, x') &= \frac{1}{\abs{\calS_{x'}}} f(\calS_{x'}) - \frac{1}{\abs{\calS_x}} f(\calS_{x}) + c = \bar{f}(x') - \bar{f}(x) + c, \\
\end{split}
\end{equation*}
where we introduced $\bar{f} \colon \calX \to \R, x \mapsto f(\calS_{x})/\abs{\calS_{x}}$.
In other words, 
$\sum_{i = 1}^{d} \alpha_i \bar{g}_i(x, x') = 0_{\calG(\calX, \calD)}$, and
since $\bar{g}_1, \dots, \bar{g}_d$ are independent in $\calG(\calX, \calD)$,
$\alpha_i = 0$ for all $i \in [d]$.
\end{proof}
When the observed Markov type $\set{T(e)}_{e \in \calE}$ defines a proper edge measure, we can construct the associated \hl{stochastic matrix} $\hat{P}_T  \in \calW_\kappa(\calY, \calE)$, and
we are often interested in solving the minimization problem,
\begin{equation*}
    P_m \eqdef \argmin_{\theta \in \Theta} \kl{\hat{P}_{T}}{P_{\origin, \theta}}.
\end{equation*}
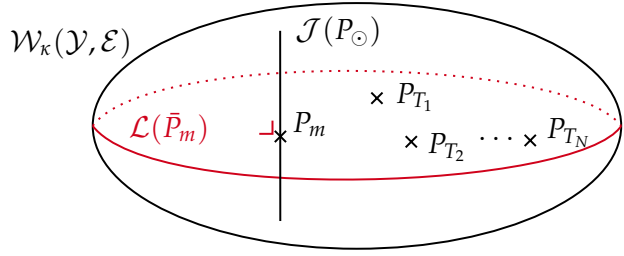
\begin{figure}%

\centering

\tikzset{every picture/.style={line width=0.75pt}} %

\begin{tikzpicture}[x=0.50pt,y=0.50pt,yscale=-1,xscale=1]

\draw   (155,132) .. controls (155,80.64) and (244.5,39) .. (354.9,39) .. controls (465.3,39) and (554.8,80.64) .. (554.8,132) .. controls (554.8,183.36) and (465.3,225) .. (354.9,225) .. controls (244.5,225) and (155,183.36) .. (155,132) -- cycle ;
\draw [color={rgb, 255:red, 208; green, 2; blue, 27 }  ,draw opacity=1 ]   (155,132) .. controls (200.8,192) and (522.8,180) .. (554.8,132) ;
\draw [color={rgb, 255:red, 208; green, 2; blue, 27 }  ,draw opacity=1 ] [dash pattern={on 0.84pt off 2.51pt}]  (155,132) .. controls (197.8,73) and (516.8,80) .. (554.8,132) ;
\draw    (296.8,60) -- (296.9,204) ;
\draw  [color={rgb, 255:red, 208; green, 2; blue, 27 }  ,draw opacity=1 ] (291,128) -- (291,138) -- (281,138) ;
\draw   (292.76,135.76) -- (301.24,144.24)(301.24,135.76) -- (292.76,144.24) ;
\draw   (365.76,106.76) -- (374.24,115.24)(374.24,106.76) -- (365.76,115.24) ;
\draw   (391.76,139.76) -- (400.24,148.24)(400.24,139.76) -- (391.76,148.24) ;
\draw   (481.76,138.76) -- (490.24,147.24)(490.24,138.76) -- (481.76,147.24) ;

\draw (382,97) node [anchor=north west][inner sep=0.75pt]   [align=left] {$P_{T_1}$};
\draw (406,135) node [anchor=north west][inner sep=0.75pt]   [align=left] {$P_{T_2}$};
\draw (444,135) node [anchor=north west][inner sep=0.75pt]   [align=left] {$\cdots$};
\draw (496,124) node [anchor=north west][inner sep=0.75pt]   [align=left] {$P_{T_N}$};
\draw (304,119) node [anchor=north west][inner sep=0.75pt]   [align=left] {$P_m$};
\draw (306,47) node [anchor=north west][inner sep=0.75pt]   [align=left] {$\calJ(P_{\origin})$};
\draw (90,56) node [anchor=north west][inner sep=0.75pt]   [align=left] {$\calW_\kappa(\calY, \calE)$};
\draw [color={rgb, 255:red, 208; green, 2; blue, 27 }  ,draw opacity=1 ] (180,120) node [anchor=north west][inner sep=0.75pt]   [align=left] {$\calL(\bar{P}_m)$};

\end{tikzpicture}
\caption{Embedded model $\calJ(P_{\origin})$. The Pythagorean leaf $\calL(\bar{P}_m)$ contains all lumpable frequency matrices $P_{T_1}, \dots, P_{T_N}$ that would result in $P_m$ being the minimizer.}
\label{figure:embedded-models}
\end{figure}
Then $P_m$ is the m-projection onto $\calL(\bar{P}_m)$, the leaf that contains the collection of all lumpable frequency matrices constructed from types that would have resulted in $P_m$ being selected as the minimizer (see Figure~\ref{figure:embedded-models}). 
In fact, from a straightforward computation,
\begin{equation*}
    \begin{split}
    \argmin_{\theta \in \Theta} \kl{\hat{P}_T}{P_{\origin, \theta}} %
    &= \argmax_{\theta \in \Theta} \sum_{e \in \calE} T(e) \log P_{\origin, \theta}(e) \\
    &= \argmax_{\theta \in \Theta} \log \mu(y_1) \prod_{e \in \calE} P_{\origin, \theta}(e)^{(n-1) T(e)} \\
    &= \argmax_{\theta \in \Theta}  \PR[\origin, \theta]{Y_1 = y_1, \dots, Y_n = y_n},\\
    \end{split}
\end{equation*}
where $Y_1, \dots, Y_n$ is sampled according to $P_{\origin, \theta}$ with arbitrary initial distribution $\mu$ \footnote{If the chain is assumed to be started stationarily, then $\mu = \pi_{\origin, \theta}$ depends on the model, and the projection does no longer correspond to the MLE.}.

\section{Extension to higher-order data processing, and composite embeddings}
\label{section:higher-order}
\hl{We naturally} extend the data-processing model to the multi-letter case.
For a trajectory
\begin{equation*}
    Y_1, Y_2, \dots, Y_t, \dots
\end{equation*}
sampled according to some \hl{stochastic matrix} $P$, we
define the $k$th order lumping 
\begin{equation*}
    \kappa^{(k)} \colon \calY^k \to \calX,
\end{equation*}
that outputs
\begin{equation*}
    \kappa^{(k)}(Y_1, Y_2, \dots, Y_k), \kappa^{(k)}(Y_2, Y_2, \dots, Y_{k+1}), \dots, \kappa^{(k)}(Y_t, \dots, Y_{t+k - 1}), \dots
\end{equation*}
Notice that this operation can be reduced to the composition of some first order lumping $\kappa'_\star$ together with a $k$th order Hudson \hl{expansion} (Section~\ref{section:embeddings-hudson}).
\begin{equation*}
    P \mapsto \kappa^{(k)}_\star P = \kappa'_{\star} H^{(k)}_\star P .
\end{equation*}

\begin{example}
\label{example:non-markov-embedding}
We let $\calY = \Z/m \Z$ be the quotient group equipped with the addition modulo $m$.
\begin{equation}
\label{eq:jumps-on-cycles}
    \calW \cyc (\calY, \calE) = \set{ P  \in \calW(\calY, \calE) \colon \exists  \mu \in \calP(\calY), P(y,y') = \mu(y' - y)}.
\end{equation}
It is easy to see that $\calW \cyc (\calY, \calE) \subset \calW \bis (\calY, \calE)$, thus all elements have uniform stationary distributions.
Furthermore, observe that $\calW \cyc (\calY, \calE) \not \subset \calW \rev (\calY, \calE)$.
We can verify that
$\calW \cyc (\calY, \calE)$ forms both an e-family and an m-family in $\calW (\calY, \calE)$ with 
$$\dim \calW \cyc (\calY, \calE) = \frac{\abs{\calE}}{\abs{\calY}} - 1.$$
Whereas the natural embedding from $\calP_+(\calY)$ to $\calW \iid (\calY)$ tears the m-structure of the simplex \citep[Lemma~8]{wolfer2021information}, the natural embedding 
$$\calP_+(\calY) \to \calW \cyc(\calY) \cap \calW_+(\calY)$$ 
preserves both the e-structure and m-structure.
For simplicity, 
let us consider a full support \hl{stochastic matrix} $P \in \calW \cyc (\calY) \cap \calW_+(\calY)$.
Let $\mu \in \calP_+(\calY)$ be such that for any $y,y' \in \calY$, $P(y,y') = \mu(y' - y)$.
Performing a Hudson \hl{expansion}, for any $y_1, y_2, y_1', y'_2 \in \calY$,
\begin{equation*}
    H_\star P((y_1, y_2), (y_1', y_2')) = \pred{ y_2 = y_1' } P(y_2, y_2') = \pred{ y_2 = y_1' } \mu(y_2' - y_2).
\end{equation*}
Then, if we choose $\kappa_2(y_1, y_2) = y_2 - y_1$, for any $x' \in \calX$ and any $y_2 \in \calY$, we have
\begin{equation*}
    \sum_{\substack{(y_1', y_2') \in \calY^2 \\ y'_2 - y'_1 = x'}} \pred{y_2 = y'_1} \mu(y'_2 - y'_1) = \mu(x') \sum_{\substack{(y_1', y_2') \in \calY^2 \\ y'_2 - y'_1 = x'}} \pred{y_2 = y'_1} = \mu(x').
\end{equation*}
We have thus effectively embedded an element of $\calW \cyc (\calY) \cap \calW_+(\calY)$ into $\calW \iid(\calY)$. Observe that 
$$\dim \calW \cyc (\calY) \cap \calW_+(\calY) = \dim \calW \iid (\calY) = m - 1,$$ and that the resulting embedding cannot represented by a Markov embedding.
\end{example}

The previous example suggests a definition for higher-order embeddings, as a composition of a Markov embedding $\Lambda_\star$, followed by a Hudson lumping $h^{(k)}_\star$ for some order $k$.
\begin{equation*}
    E_\star \colon \calW(\calX, \calD) \stackrel{\Lambda_\star}{\to} \calW_\kappa(\calY,\calE) \subset \calW_{h^{(k)}}(\calY, \calE) \stackrel{h_\star^{(k)}}{\to} \calW(\calZ, \calF).
\end{equation*}

\section{Proofs}
\label{section:proofs}
\subsection{Proof of Lemma~\ref{lemma:markov-embedding-preserve-information-geometry}}
\label{proof:lemma-markov-embedding-preserve-information-geometry}

Let $\kappa$ and $\biguplus_{x \in \calX} \calS_x = \calY$ be the associated lumping function and partition of the Markov embedding.  
For all $y,y' \in \calY$ and $i \in [d]$, it holds that
\begin{equation*}
\begin{split}
 \partial_i \log \Lambda_\star \bar{P}_\theta(y, y') &= \partial_i \log [\Lambda(y,y' )\bar{P}_\theta(\kappa(y), \kappa(y'))] = \partial_i \log \bar{P}_\theta(\kappa(y), \kappa(y')),\\
\end{split}
\end{equation*}
thus
\begin{equation*}
\begin{split}
\fshr_{ij}(\theta) &= \sum_{(y, y') \in \calE} Q_\theta(y, y') \partial_i \log \bar{P}_\theta(\kappa(y), \kappa(y')) \partial_j \log \bar{P}_\theta(\kappa(y), \kappa(y')) \\
&= \sum_{(x, x') \in \calD} \left(\sum_{y \in \calS_{x}, y' \in \calS_{x'}} Q_\theta(y, y')\right) \partial_i \log \bar{P}_\theta(x, x') \partial_j \log \bar{P}_\theta(x, x'). \\
\end{split}
\end{equation*}
It follows from Corollary~\ref{corollary:lumped-stationary-distribution-and-edge-measure} that $\fshr$ is preserved. 
We proceed to prove conservation of the e-connection,
\begin{equation*}
\begin{split}
\Gamma^{(e)}_{ij, k}(\theta) 
=& \sum_{(y, y') \in \calE}  \partial_i \partial_j \log \Lambda(y,y') \bar{P}_\theta(\kappa(y), \kappa(y')) \partial_k Q_\theta(y,  y') \\
=& \sum_{(x,x') \in \calD} \partial_i \partial_j \log \bar{P}_\theta(x, x') \partial_k \left( \sum_{ y \in \calS_x, y' \in \calS_{x'}} Q_\theta(y,  y') \right) \\
=& \sum_{(x,x') \in \calD}  \partial_i \partial_j \log \bar{P}_\theta(x, x') \partial_k \bar{Q}_\theta(x,  x'), \\
\end{split}
\end{equation*}
where the last equality follows from Corollary~\ref{corollary:lumped-stationary-distribution-and-edge-measure}.
Invariance of the m-connection and information divergence can be proven using similar arguments \footnote{A more economical but less elementary proof consists in first proving the invariance claim for the information divergence, and then recover the conjugate connection manifold following the construction of \citet{eguchi1983second, eguchi1985differential}.}.

\qed

\subsection{Proof of Lemma~\ref{lemma:lumpable-matrix-vector-space}}
\label{proof:lemma-lumpable-matrix-vector-space}

Let $A, B \in \calF_\kappa(\calY, \calE)$, and $\alpha, \beta \in \R$.
Then, for all $x,x' \in \calX$, 
and for all $y_1,y_2 \in \calS_x$, by operations on matrices,
\begin{equation*}
    \begin{split}
        (\alpha A +  \beta B)(y_1, \calS_{x'})
        &= \alpha A(y_1, \calS_{x'}) + \beta B(y_1, \calS_{x'}) \\ 
        &= \alpha A(y_2, \calS_{x'}) + \beta B(y_2, \calS_{x'}) = (\alpha A + \beta B)(y_2, \calS_{x'}),
    \end{split}
\end{equation*}
thus $\alpha A + \beta B \in \calF_\kappa(\calY, \calE)$.
Moreover, $0 \in \calF_\kappa(\calY, \calE)$,
hence $\calF_\kappa(\calY, \calE)$ is a subspace of $\calF(\calY, \calE)$, and $(i)$ holds.
Moving on to $(ii)$,
let $A, B \in \calF(\calY, \calE)$, and $\alpha, \beta \in \R$. For any $x,x' \in \calX$, and $y \in \calS_x$,
\begin{equation*}
\begin{split}
    \kappa_\star(\alpha A + \beta B)(x,x') &= (\alpha A + \beta B)(y,\calS_{x'}) \\ &= \alpha A (y,\calS_{x'}) + \beta B(y,\calS_{x'}) = (\alpha \kappa_\star A + \beta \kappa_\star B)(x,x'),
\end{split}
\end{equation*}
thus $\kappa_\star$ is a linear map.
In order to prove surjectivity in $(ii)$ and claim $(iii)$, we proceed to construct a basis.
Taking the total order on $\calY = [m]$ induced from the natural numbers, and
for $(x,x') \in \calD$, we write
$$\calR_{x,x'} \eqdef \set{(y,y') \in \calE \colon y \in \calS_x, y' \in \calS_{x'}, y' \neq \ymax(\calS_{x'}, y)},$$ 
where 
$$\ymax(\calS_{x'}, y) \eqdef \max \set{y' \in \calS_{x'} \colon (y, y') \in \calE}.$$
For simplicity, we will use the shorthands 
$$\ymax = \ymax(\calS_{x'}, y) \text{ and } \ymax_0 = \ymax(\calS_{x_0'}, y_0).$$
Writing $E^{y_0,y'_0}(y,y') = \delta[y_0 = y] \delta[y'_0 = y']$ for $(y, y'), (y_0, y_0') \in \calE$,
we define
\begin{equation*}
\begin{split}
C^{x, x'} &\eqdef  \sum_{y \in \calS_x} E^{y, \ymax} \;\; \text{ for } (x,x') \in \calD,
\\
F^{y,y'} &\eqdef E^{y,y'} - E^{y, \ymax} \;\; \text{ for } (y,y') \in \calR_{x,x'}, \text{ and } (x,x') \in \calD,\\
\end{split}
\end{equation*}
which are all elements of $\calW_\kappa(\calY, \calE)$ by construction, and consider the collection
\begin{equation}
\label{eq:lumpable-matrices-basis}
    \calB(\kappa) \eqdef \set{ C^{x, x'} \colon (x,x') \in \calD } \cup \set{F^{y,y'} \colon (y,y') \in \calR_{x,x'}, (x,x') \in \calD}.
\end{equation}
Let $(\alpha_{x,x'} \in \R \colon (x,x') \in \calD)$ and $(\alpha_{y,y'} \in \R \colon (y,y') \in \calR_{x,x'}, (x,x') \in \calD)$ be such that
\begin{equation}
    \label{eq:null-linear-combination}
    \sum_{(x,x') \in \calD} \alpha_{x,x'} C^{x,x'} + \sum_{(y,y') \in \calR_{x,x'}, (x,x') \in \calD} \alpha_{y,y'} F^{y,y'} = 0_{\calF_{\kappa}(\calY, \calE)}.
\end{equation}
Let $(x_0, x_0') \in \calD$.
Taking \eqref{eq:null-linear-combination} for any $(y_0, y_0') \in \calR_{x_0, x_0'}$ yields $\alpha_{y_0, y_0'} = 0$, while
\eqref{eq:null-linear-combination} taken at $(y_0, \ymax_0)$ for any $y_0 \in \calS_{x_0}$ yields $\alpha_{x_0,x_0'} = 0$. As a result, $\calB(\kappa)$ forms a linearly independent family in $\calF_\kappa(\calY, \calE)$.
Let now $A \in \calF_\kappa(\calY, \calE)$. Then $A \in \calF(\calY, \calE)$, and there exists $\set{a_{y,y'} \in \R \colon (y,y') \in \calE}$ such that $A = \sum_{(y,y') \in \calE} a_{y,y'} E^{y,y'}$. Since $A$ is $\kappa$-lumpable, it holds that for any $(x,x') \in \calD$, and any $y_1, y_2 \in \calS_{x}$,
\begin{equation*}
    \sum_{y' \in \calS_{x'}} a_{y_1, y'} = \sum_{y' \in \calS_{x'}} a_{y_2, y'} \eqdef a_{x,x'}.
\end{equation*}
By decomposition,
\begin{equation*}
    \begin{split}
        A &= \sum_{(x,x') \in \calD} \sum_{(y,y') \in \calR_{x,x'}} a_{y,y'} E^{y,y'} +  \sum_{(x,x') \in \calD} \sum_{y \in \calS_{x}} a_{y,\ymax} E^{y,\ymax} \\
        &= \sum_{(x,x') \in \calD} \sum_{(y,y') \in \calR_{x,x'}} a_{y,y'} E^{y,y'} +  \sum_{(x,x') \in \calD} \sum_{y \in \calS_{x}} \left(a_{x,x'} - \sum_{y' \in \calS_{x'}, y' \neq \ymax}a_{y,y'}\right) E^{y,\ymax} \\
        &= \sum_{(x,x') \in \calD} \sum_{(y,y') \in \calR_{x,x'}} a_{y,y'} \underbrace{\left(E^{y,y'} - E^{y,\ymax}\right)}_{ F^{y,y'}} +  \sum_{(x,x') \in \calD}  a_{x,x'} \underbrace{\sum_{y \in \calS_{x}} E^{y,\ymax}}_{C^{x,x'}}, \\
    \end{split}
\end{equation*}
thus $\calB(\kappa)$ is also a generating family for $\calF_\kappa(\calY, \calE)$.
Further notice that
\begin{equation}
\label{eq:lump-the-basis}
\begin{split}
\kappa_\star C^{x, x'} = E^{x,x'}, \qquad \kappa_\star F^{y,y'} = 0.
\end{split}
\end{equation}
In fact,
\begin{equation*}
\begin{split}
 \Ker \kappa_\star &= \Span\left(\set{F^{y,y'} \colon (y,y') \in \calR_{x,x'}, (x, x') \in \calD}\right), \\
\Range \kappa_\star &= \Span\left(\set{\kappa_\star C^{x, x'} \colon (x,x') \in \calD}\right)= \calF(\calX, \calD),
\end{split}
\end{equation*}
hence the surjectivity of $\kappa_\star$,
and from the rank-nullity theorem,
\begin{equation*}
\begin{split}
    \dim \calF_\kappa(\calY, \calE) &= \dim \Ker \kappa_\star + \dim \Range \kappa_\star
    = \abs{\calE} - \sum_{(x,x') \in \calD} \abs{\calS_x} + \abs{\calD}.
\end{split}
\end{equation*}

\qed

\subsection{Proof of Theorem~\ref{theorem:congruent-embeddings-are-lambda-embeddings}}
\label{proof:theorem-congruent-embeddings-are-lambda-embeddings}

Let $K_\star$ be a $\kappa$-congruent embedding.
Recall the basis $\calB(\kappa)$ of $\calF_\kappa(\calY, \calE)$ introduced in \eqref{eq:lumpable-matrices-basis}.
Since $K_\star$ is a linear map, we can define it uniquely by the coordinates of the image of basis vectors of $\calF(\calX, \calD)$ onto the basis $\calB(\kappa)$. Namely, for $(x_0, x_0') \in \calD$, we write
\begin{equation*}
    \begin{split}
        K_\star E^{x_0,x_0'} = \sum_{\substack{(x,x') \in \calD \\ (y,y') \in \calR_{x,x'} }}  K_{y,y'}^{x_0, x_0'} F^{y,y'} + \sum_{(x,x') \in \calD} K_{x, x'}^{x_0, x_0'} C^{x, x'},
    \end{split}
\end{equation*}
where the $K_{y,y'}^{x_0, x_0'}$ and $K_{x, x'}^{x_0, x_0'}$ are real numbers.
Let $(x,x') \in \calD$, and  $(y,y') \in (\calS_{x} \times \calS_{x'}) \cap \calE$. Since $E^{x_0,x_0'}$ is non-negative, it follows from monotonicity that when $y' \neq \ymax$,
\begin{equation*}
    \begin{split}
        (K_\star E^{x_0,x_0'})(y,y') = K_{y,y'}^{x_0, x_0'},
    \end{split}
\end{equation*}
is non-negative. On the other hand, 
\begin{equation*}
    \begin{split}
        (K_\star E^{x_0,x_0'})(y,\ymax) = K^{x_0,x_0'}_{x, x'} - \sum_{\bar{y} \in \calS_{x'}, (y,\bar{y}) \in \calE, \bar{y} \neq \ymax} K_{y,\bar{y}}^{x_0, x_0'},
    \end{split}
\end{equation*}
must be non-negative, thus
\begin{equation}
\label{eq:all-positive-tool}
    \begin{split}
        K^{x_0,x_0'}_{x, x'} \geq  \max_{y \in \calS_{x}} \sum_{\bar{y} \in \calS_{x'}, (y,\bar{y}) \in \calE, \bar{y} \neq \ymax} K_{y,\bar{y}}^{x_0, x_0'} \geq 0.
    \end{split}
\end{equation}
From the requirement that $\kappa_\star K_\star = \Id_{\calF(\calX, \calD)}$,  linearity of $\kappa_\star$ (Lemma~\ref{lemma:lumpable-matrix-vector-space}-$(ii)$), and  \eqref{eq:lump-the-basis}, we have for any $(x_0, x_0') \in \calD$,
\begin{equation*}
    \begin{split}
        E^{x_0,x_0'} &= \kappa_\star K_\star E^{x_0,x_0'} = \sum_{\substack{(x,x') \in \calD \\ (y,y') \in \calR_{x,x'} }} K_{y,y'}^{x_0, x_0'} \kappa_\star F^{y,y'} + \sum_{(x,x') \in \calD} K_{x, x'}^{x_0, x_0'} \kappa_\star C^{x,x'} \\
        & = \sum_{(x,x') \in \calD} K_{x, x'}^{x_0, x_0'} E^{x,x'}. \\
    \end{split}
\end{equation*}
Thus, on one hand, for $(x, x') \neq (x_0, x_0')$, $K_{x, x'}^{x_0, x_0'} = 0$, and from \eqref{eq:all-positive-tool}, it follows that
for any $(y,y') \in \calR_{x,x'}, K_{y,y'}^{x_0,x_0'} = 0$.
On the other hand, 
$K_{x_0, x_0'}^{x_0, x_0'} = 1$,
hence 
$$\sum_{\bar{y}_0 \in \calS_{x_0'}, (y_0, \bar{y}_0) \in \calE, \bar{y}_0 \neq \ymax_0} K_{y_0,\bar{y}_0}^{x_0,x_0'} \leq 1$$ 
for any $y_0 \in \calS_{x_0}$, and from non-negativity, each individual coefficient $K_{y_0, \bar{y}_0}^{x_0, x_0'}$ is also in $[0,1]$.
We therefore obtain that for any $(x,x') \in \calD$,
\begin{equation*}
    \begin{split}
        K_\star E^{x,x'} = C^{x, x'} +  \sum_{(y,y') \in \calR_{x,x'}} K_{y,y'}^{x, x'} F^{y,y'}. 
    \end{split}
\end{equation*}
Consider $A = \sum_{(x,x') \in \calD}A(x,x')E^{x,x'} \in \calF(\calX, \calD)$. Embedding $A$ yields
\begin{equation*}
    \begin{split}
        K_\star A &= \sum_{(x,x') \in \calD} A(x,x') K_\star E^{x,x'} \\
        &= \sum_{(x,x') \in \calD} A(x,x' ) \left( C^{x, x'} +  \sum_{(y,y') \in \calR_{x,x'}} K_{y,y'}^{x, x'} F^{y,y'} \right) \\
        &= \sum_{(x,x') \in \calD} A(x,x' ) \left( \sum_{y \in \calS_x} E^{y, \ymax} +  \sum_{(y,y') \in \calR_{x,x'}} K_{y,y'}^{x, x'} \left(E^{y,y'} - E^{y, \ymax}\right) \right) \\
        &= \sum_{(x,x') \in \calD} A(x,x' ) \sum_{y \in \calS_x}\left( E^{y, \ymax} +  \sum_{\substack{ y' \in \calS_{x'} \\ (y,y') \in \calE, y' \neq \ymax}} K_{y,y'}^{x, x'} \left(E^{y,y'} - E^{y, \ymax}\right) \right) \\
        &= \sum_{(x,x') \in \calD} A(x,x' ) \sum_{\substack{y \in \calS_x \\ y' \in \calS_{x'} \\ (y,y') \in \calE} } \left[ \delta[y' = \ymax] \left(1 - \sum_{\substack{\bar{y} \in \calS_{x'} \\ (y,\bar{y}) \in \calE \\ \bar{y} \neq \ymax}} K_{y,\bar{y}}^{x, x'} \right) +  \delta[y' \neq \ymax] K_{y,y'}^{x, x'} \right] E^{y, y'}. \\
    \end{split}
\end{equation*}
Recall that $\sum_{\bar{y} \in \calS_{x'}, (y,\bar{y}) \in \calE, \bar{y} \neq \ymax} K_{y,\bar{y}}^{x, x'} \leq 1$.
It is then convenient to define
\begin{equation*}
    K_{y,\ymax}^{x, x'} \eqdef 1 - \sum_{\bar{y} \in \calS_{x'}, (y,\bar{y}) \in \calE, \bar{y} \neq \ymax} K_{y,\bar{y}}^{x, x'} \in [0,1],
\end{equation*}
so that we can write more compactly
\begin{equation*}
    \begin{split}
        K_\star A
        &= \sum_{(x,x') \in \calD} A(x,x' ) \sum_{y \in \calS_x, y' \in \calS_{x'}, (y,y') \in \calE} K_{y,y'}^{x, x'}  E^{y, y'}, \\
        &= \sum_{(y,y') \in \calE} \sum_{(x,x') \in \calD} A(x,x' ) \delta[y' \in \calS_{x'}] \delta[y \in \calS_{x}]  K_{y,y'}^{x, x'}  E^{y, y'}, \\
        &= \sum_{(y,y') \in \calE}  A(\kappa(y),\kappa(y') )  K_{y,y'}^{\kappa(y), \kappa(y')}  E^{y, y'}, \\
    \end{split}
\end{equation*}
where for any $x \in \calX$ and $y \in \calS_{x}$,
\begin{equation}
    \label{eq:is-distribution}
    (K_{y,y'}^{x, x'})_{y' \in \calS_{x'}} \in \calP(\calS_{x'}).
\end{equation}
Let us introduce $\Lambda \in \calF(\calY, \calE)$, with $\Lambda(y,y') = K_{y,y'}^{\kappa(y), \kappa(y')}$ for any $(y,y') \in \calE$.
Then for $(y,y') \in \calY$, $K_\star A(y,y') = A(\kappa(y),\kappa(y') )  \Lambda(y,y')$,
and $K_\star$ satisfies $(i)$ of Definition~\ref{definition:markov-embeddings}.
Suppose now for contradiction that there exists $(y_0,y_0') \in \calE$ such that $\Lambda(y_0,y_0') = 0$, and let $B \in \calF_+(\calX, \calD)$.
Then $K_\star B(y_0,y_0') = B(\kappa(y_0), \kappa(y'_0)) \Lambda(y_0, y_0') = 0$ and $\Lambda_\star B \not \in \calF_+(\calY, \calE)$.
Thus, $\Lambda \in \calF_+(\calY, \calE)$, and requirement $(iii)$ is met.
Finally, \eqref{eq:is-distribution} leads to condition $(iv)$. 
As a result,
any congruent embedding can be expressed as a Markov embedding.
It is straightforward to verify that conversely, a $\kappa$-compatible Markov embedding is always $\kappa$-congruent, whence the theorem.

\qed

\subsection{Proof of Proposition~\ref{proposition:explicit-phi-embedding}}
\label{proof:proposition-explicit-phi-embedding}

Let $y \in \calY$, and let $(\rho, v)$ be the right PF pair of $\bar{P}_{\origin} \circ \bar{P}$. It holds that
\begin{equation*}
\begin{split}
    \sum_{y' \in \calY} \tilde{P}(y,y')v(\kappa(y')) &= \sum_{y' \in \calY} P_{\origin}(y,y')\bar{P}(\kappa(y), \kappa(y'))v(\kappa(y'))\\ 
    &= \sum_{x' \in \calX} \sum_{y' \in \calS_{x'}} P_{\origin}(y,y')\bar{P}(\kappa(y), \kappa(y'))v(\kappa(y'))\\
    &= \sum_{x' \in \calX} \sum_{y' \in \calS_{x'}} P_{\origin}(y,y')\bar{P}(\kappa(y), x')v(x')\\
    &= \sum_{x' \in \calX} \bar{P}_{\origin}(\kappa(y),x')\bar{P}(\kappa(y), x')v(x')\\
    &= \rho v(\kappa(y)), \\
\end{split}
\end{equation*}
where the fourth equality stems from $\kappa_\star P_{\origin} = \bar{P}_{\origin}$, hence $(i)$ holds.
Furthermore, for all $x,x' \in \calX$, and for all $y \in \calS_x$,
\begin{equation*}
\begin{split}
    \sum_{y' \in \calS_{x'}} \hlm{\Phi} \bar{P}(y,y') &= \frac{1}{\rho} \sum_{y' \in \calS_{x'}}v(\kappa(y))^{-1} P_{\origin}(y,y')\bar{P}(\kappa(y), \kappa(y'))v(\kappa(y')) \\
    &= \frac{1}{\rho} \sum_{y' \in \calS_{x'}}v(x)^{-1} P_{\origin}(y,y')\bar{P}(x, x')v(x') \\
    &= \frac{1}{\rho} v(x)^{-1} \bar{P}_{\origin}(\kappa(y),x')\bar{P}(x, x')v(x'), \\
    &= \frac{1}{\rho} v(x)^{-1} \bar{P}_{\origin}(x,x')\bar{P}(x, x')v(x'), \\
\end{split}
\end{equation*}
where the first equality is $(i)$.
Observe that the expression we obtain is independent of $y$, thus the chain is $\kappa$-lumpable, and
in fact, by definition of lumping, $\kappa_\star \hlm{\Phi} \bar{P} = \stoch(\bar{P}_{\origin} \circ \bar{P})$, whence $(ii)$.

\qed

\subsection{Proof of Lemma~\ref{lemma:hudson-breaks-m-structure}}
\label{proof:lemma-hudson-breaks-m-structure}

Let $p \in (0,1), p \neq 1/2$, $\calX = \set{0, 1}$, and consider the two positive \hl{stochastic matrices}
$$\bar{P}_0 = \begin{pmatrix} 1 - p & p \\ p & 1 - p \end{pmatrix} \qquad \bar{P}_1 = \begin{pmatrix} p & 1 - p \\ 1 - p & p \end{pmatrix}.$$
We compute successively,

\begin{equation*}
    P_0 = \begin{array}{c c} &
\begin{array}{c c c c} 00 & 01 & 10 & 11 \\
\end{array}
\\
\begin{array}{c c c c}
00 \\
01 \\
10 \\
11
\end{array}
&
\left(
\begin{array}{c c c c}
1 - p  & p  & 0 & 0 \\
0 & 0 & p & 1 - p \\
1 - p  & p & 0 & 0 \\
0 & 0 & p & 1 - p \\
\end{array}
\right)
\end{array}, P_1 = 
\left(
\begin{array}{c c c c}
p  & 1 - p  & 0 & 0 \\
0 & 0 & 1- p & p \\
p  & 1 - p & 0 & 0 \\
0 & 0 & 1 - p & p \\
\end{array}
\right),
\end{equation*}
\begin{equation*}
\pi(00) = \frac{1 - p}{2}, \; \pi(01) = \frac{p}{2}, \; \pi(10) = \frac{p}{2}, \; \pi(11) = \frac{1 - p}{2},
\end{equation*}

\begin{equation*}
    Q_0 = 
\left(
\begin{array}{c c c c}
\frac{(1 - p)^2}{2}  & \frac{p(1-p)}{2} & 0 & 0 \\
0 & 0 & \frac{p^2}{2} & \frac{p(1-p)}{2} \\
\frac{p(1-p)}{2} & \frac{p^2}{2} & 0 & 0 \\
0 & 0 & \frac{p(1-p)}{2} & \frac{(1 - p)^2}{2} 
\end{array}
\right),
\end{equation*}
\begin{equation*}
Q_1 = 
\left(
\begin{array}{c c c c}
\frac{p^2}{2} & \frac{p(1-p)}{2} & 0 & 0 \\
0 & 0 & \frac{(1 -p)^2}{2} & \frac{p(1-p)}{2} \\
\frac{p(1-p)}{2} & \frac{(1- p)^2}{2} & 0 & 0 \\
0 & 0 & \frac{p(1-p)}{2} & \frac{p^2}{2} \\
\end{array}
\right),
\end{equation*}

\begin{equation*}
    Q_{1/2} = \frac{1}{2} Q_0 + \frac{1}{2} Q_1 = 
\left(
\begin{array}{c c c c}
\frac{(1 - p)^2 + p^2}{4} & \frac{p(1-p)}{2} & 0 & 0 \\
0 & 0 & \frac{(1 -p)^2 + p^2}{4} & \frac{p(1-p)}{2} \\
\frac{p(1-p)}{2} & \frac{(1- p)^2 + p^2}{4} & 0 & 0 \\
0 & 0 & \frac{p(1-p)}{2} & \frac{(1 - p)^2 + p^2}{4} \\
\end{array}
\right),
\end{equation*}
\begin{equation*}
    \pi_{1/2} = (1/4, 1/4, 1/4, 1/4),
\end{equation*}
\begin{equation*}
    \gamma^{(m)}_{P_0, P_1}(1/2) = P_{1/2} = 
\left(
\begin{array}{c c c c}
(1 - p)^2 + p^2 & 2 p(1-p) & 0 & 0 \\
0 & 0 & (1 -p)^2 + p^2 & 2 p(1-p) \\
2 p(1-p) & (1- p)^2 + p^2 & 0 & 0 \\
0 & 0 & 2 p(1-p) & (1 - p)^2 + p^2 \\
\end{array}
\right),
\end{equation*}
which is not lumpable, i.e. $\gamma^{(m)}_{P_0, P_1}(1/2) \not \in \calW_h(\calX^2, H_{\calX^2})$.

\qed
\subsection{Proof of Proposition~\ref{proposition:pythagorean-inequality-e-convex}}
\label{proof:proposition-pythagorean-inequality-e-convex}

The proof follows the same strategy as for the distribution setting.
It is easy to see that $(i)$ implies $(ii)$, and it remains to prove the converse statement.
Set $P_t = \gamma^{(e)}_{P_0, \bar{P}}(t)$. We can write for any $(x,x') \in \calD'$,
$$P_t(x,x') = \frac{P_0(x,x')^{1-t} \bar{P}(x,x')^t v_t(x')}{v_t(x) \rho_t},$$
where $(\rho_t, v_t)$ is the right PF pair of the matrix $P_0^{\circ (1-t)} \circ \bar{P}^{\circ t}$.
We compute,
\begin{equation*}
    \begin{split}
        \partial_t \kl{P}{P_t} =& \sum_{(x,x') \in \calD'} Q(x,x') \partial_t \log \frac{1}{P_t(x,x')} \\
        =& \sum_{(x,x') \in \calD'} Q(x,x') \partial_t \log \frac{v_t(x) \rho_t}{P_0(x,x')^{1-t} \bar{P}(x,x')^t v_t(x')} \\
        \stackrel{(a)}{=}& \sum_{(x,x') \in \calD'} Q(x,x') \partial_t \log \frac{ \rho_t}{P_0(x,x')^{1-t} \bar{P}(x,x')^t} \\
        =& - \sum_{(x,x') \in \calD'} Q(x,x') \partial_t  \left( (1-t) \log P_0(x,x') \right) \\&- \sum_{(x,x') \in \calD'} Q(x,x') \partial_t \left(t \log  \bar{P}(x,x') \right) + \partial_t \log \rho_t \\
        =& \sum_{(x,x') \in \calD'} Q(x,x') \log P_0(x,x') \\
        &- \sum_{(x,x') \in \calD'} Q(x,x')  \log  \bar{P}(x,x')  + \partial_t \log \rho_t \\
        =& \sum_{(x,x') \in \calD'} Q(x,x') \log \frac{ P_0(x,x')}{\bar{P}(x,x')}  + \partial_t \log \rho_t \\
        \stackrel{(b)}{=}& \sum_{(x,x') \in \calD'} \left[ Q(x,x') - Q_t(x,x') \right] \log \frac{ P_0(x,x')}{\bar{P}(x,x')}, \\
    \end{split}
\end{equation*}
where $(a)$ follows from
$$\sum_{(x,x') \in \calD'} Q(x,x') \left[ \partial_t \log v_t(x) - \partial_t \log v_t(x') \right] = 0,$$
and $(b)$ stems from \citet[Theorem~4]{nagaoka2017exponential}, which yields
$$\partial_t \log \rho_t = \sum_{(x,x') \in \calD'}Q_t(x,x') \log \frac{\bar{P}(x,x')}{P_0(x,x')}.$$
By a first-order Taylor expansion around $P_0$, there exists $s \in [0, t]$ such that
\begin{equation*}
    \kl{P}{P_t} = \kl{P}{P_0} + t \frac{\partial}{\partial_t} \kl{P}{P_t} \at{t = s}.
\end{equation*}
Moreover, since $P_0, \bar{P} \in \calC$, it follows by e-convexity that also $P_t \in \calC$, and $P_0$ being the minimizer implies 
\begin{equation*}
    \frac{1}{t}\left( \kl{P}{P_t} - \kl{P}{P_0}\right) \geq 0.
\end{equation*}
Taking the limit $t \to 0$ yields that
\begin{equation*}
    \sum_{(x,x') \in \calD'} \left(Q(x,x') - Q_0(x,x') \right) \log \frac{P_0(x,x')}{\bar{P}(x,x')} \geq 0.
\end{equation*}
Uniqueness follows from strict e-convexity, whence the theorem.

\qed

\subsection{Proof of Lemma~\ref{lemma:L-and-J-are-special-families}}
\label{proof:lemma-L-and-J-are-special-families}

To prove $(i)$, notice that since $\calW(\calX, \calD)$ forms an e-family \citep[Corollary~1]{nagaoka2017exponential}, and $\Lambda^{(P_{\origin})}$ preserves the e-structure (Theorem~\ref{theorem:exponential-embeddings-are-e-geodesic-affine-maps}), $\calJ(P_{\origin})$ also forms an e-family.
Since $\Lambda_\star^{(P_{\origin})}$ is an embedding, it is a diffeomorphism onto its image, thus $\dim \calJ(P_{\origin}) = \dim \calW (\calX, \calD)$.
It remains to prove $(ii)$, i.e. $\calL(\bar{P}_0)$ is closed under affine combination.
Let two Markov embeddings induced by $\Lambda_1, \Lambda_2$ that embed $\bar{P}_0$ respectively into $P_1, P_2 \in \calL(\bar{P}_0)$,
\begin{equation*}
\begin{split}
    P_1(y,y') &= \bar{P}_0(\kappa(y), \kappa(y')) \Lambda_1(y,y'), \\
    P_2(y,y') &= \bar{P}_0(\kappa(y), \kappa(y')) \Lambda_2(y,y'). \\
\end{split}
\end{equation*}
Let $\pi_1$ (resp. $\pi_2$) be the stationary distribution of $P_1$ (resp. $P_2$),
and let $t \in \R$,
such that 
\begin{equation*}
    Q_t(y,y') = t \pi_1(y) P_1(y,y') + (1 - t) \pi_2(y) P_2(y,y'),
\end{equation*}
defines a proper edge measure (note that $t$ is allowed to take value outside of $(0,1)$).
The stationary distribution of the \hl{stochastic matrix} pertaining to $Q_t$ is immediately given by
\begin{equation*}
    \pi_t(y) = t \pi_1(y) + (1 - t) \pi_2(y),
\end{equation*}
and
\begin{equation*}
\begin{split}
    P_t(y,y') &= \frac{Q_t(y,y')}{\pi_t(y)} = t\frac{ \pi_1(y)}{\pi_t(y)} P_1(y,y') + (1-t)\frac{\pi_2(y)}{\pi_t(y)}P_2(y,y'),\\
    &= \bar{P}_0(\kappa(y), \kappa(y')) \Lambda_t(y,y'),
\end{split}
\end{equation*}
where we wrote
\begin{equation*}
    \Lambda_t(y,y') \eqdef  t\frac{ \pi_1(y)}{\pi_t(y)} \Lambda_1(y,y') + (1-t)\frac{\pi_2(y)}{\pi_t(y)}\Lambda_2(y,y').
\end{equation*}
Non-negativity of $P_t(y,y')$ and $\bar{P}_0(y,y')$ implies that $\Lambda_t(y,y')$ is non-negative.
Moreover, for $x' \in \calX$,
\begin{equation*}
\begin{split}
    \sum_{y' \in\calS_{x'}} \Lambda_t(y,y') &= t\frac{ \pi_1(y)}{\pi_t(y)} \sum_{y' \in\calS_{x'}}\Lambda_1(y,y') + (1-t)\frac{\pi_2(y)}{\pi_t(y)}\sum_{y' \in\calS_{x'}}\Lambda_2(y,y') \\
    &= t\frac{ \pi_1(y)}{\pi_t(y)}  + (1-t)\frac{\pi_2(y)}{\pi_t(y)} = 1.\\
\end{split}
\end{equation*}
As a result, $\Lambda_t$ defines a proper Markov embedding, and $P_t \in \calL(\bar{P}_0)$.
The dimension of $\calL(\bar{P}_{0})$ is obtained by considering its one-to-one correspondence with the canonical embedding map,
which has the same number of degrees of freedom as a Markov embedding (Remark~\ref{remark:dof-markov-embedding}).

\qed

\section*{Acknowledgments}
We are thankful to Martin Adam\v{c}\'{i}k, \hl{Jun'ichi Takeuchi and Hiroshi Nagaoka} for the insightful discussions.
We also express our gratitude to the anonymous referees for their insightful comments, which have greatly contributed to improving the quality of this manuscript.

\bibliography{bibliography}
\bibliographystyle{abbrvnat}

\end{document}